\documentclass[11pt]{amsart}

\usepackage[margin=1in]{geometry}
\usepackage{amsmath,amssymb,amsthm}
\usepackage{graphicx}
\usepackage{mathtools}
\usepackage{enumitem}
\usepackage{hyperref}
\usepackage{mathrsfs}
\usepackage{xcolor}
\usepackage{comment}
\hypersetup{colorlinks=true,linkcolor=blue,citecolor=blue,urlcolor=blue}

\theoremstyle{plain}
\newtheorem{theorem}{Theorem}[section]
\newtheorem{lemma}[theorem]{Lemma}
\newtheorem{proposition}[theorem]{Proposition}
\newtheorem{corollary}[theorem]{Corollary}

\theoremstyle{definition}
\newtheorem{definition}[theorem]{Definition}

\theoremstyle{remark}
\newtheorem{remark}[theorem]{Remark}

\numberwithin{equation}{section}

\makeatletter
\newcommand{\multilabel}[1]{\ltx@label{#1}}
\makeatother

\newcommand{\R}{\mathbb R}
\newcommand{\C}{\mathbb C}
\newcommand{\N}{\mathbb N}

\newcommand\warning[1]{{}} 
\newcommand\bwarning[1]{{}} 
\newcommand\gwarning[1]{{}} 

\title[Chernoff Approximation by Centered Differences]{(Super-)Fast Convergence Rates for Chernoff Approximations via Centered Differences and Richardson Extrapolation}

\author[]{\NoCaseChange{Comlan Edmond Koudjinan\\ University of Toronto\\Koudjinanedmond@gmail.com}}

\date{\today}

\begin{document}

\begin{abstract}
We use centered differences to construct Chernoff families and
derive even asymptotic expansions for their iterated products.
Richardson extrapolation is then used to cancel polynomial part and thereby raise the convergence order.

As illustrations, we apply the construction to uniformly parabolic, 
Schr\"odinger-type second-order operators and to a third-order dispersive
operator, obtaining, at extrapolation depth \(m\), the rates
$
O(n^{-2m})
$
on the Sobolev scales \(H^{2m+\delta}\), \(H^{4m}\), and \(H^{5m}\),
respectively, with arbitrary \(\delta>0\). In the analytic category,
quantitative all-order estimates allow the extrapolation depth to grow
with the base resolution and yield the super-fast rates
$
e^{-cn^2}
$
for the two second-order models and
$
e^{-cn}
$
for the third-order model.

{
The constructions are explicit in terms of the coefficients of the
underlying differential operators and do not use the target semigroup
in their definition. This provides a constructive positive
answer 
to the challenge of
super-fast Chernoff approximation formulated by Remizov in
 \cite[Problem~30]{OpenProblems}.
}

\warning{
We prove an abstract Richardson extrapolation principle for Chernoff
products. If $\mathcal S(t/N)^N$ admits an asymptotic expansion on a
prescribed power scale and the base family satisfies the usual Chernoff
exponential power bound, an explicit multiproduct cancels successive
terms and raises the convergence order. We introduce centered-difference Chernoff families and derive the
required product expansion from the local even expansion of the
approximating generators by stability and Duhamel iteration. As
illustrations, we apply the construction to uniformly parabolic and
Schr\"odinger-type second-order operators and to a third-order dispersive
operator, obtaining convergence rates
$O(n^{-2m})$ on the Sobolev scales $H^{2m+\delta}$, $H^{4m}$, and
$H^{5m}$, respectively. In the analytic category, quantitative all-order estimates allow the depth to grow with the base resolution
and yield $e^{-cn^2}$ for the two second-order models and $e^{-cn}$ for
the third-order model. The resulting Richardson--Chernoff formulas give
accelerated representations of the evolution semigroups; their Laplace
transforms give polynomial-rate convergent representations of the
corresponding resolvents under the same pointwise product expansion used
in the resolvent estimate.}
\end{abstract}

\maketitle

\section{Introduction}\label{sec:introduction}
\enlargethispage{6pt}

The Chernoff product formula \cite{Chernoff} provides a general mechanism
for approximating $C_0$-semigroups. 
It is particularly useful since the closed form of the semigroup  is not known in general: one replaces the unbounded generator by a family
of explicitly tractable bounded  operators and recovers the evolution
from its iterated products.

For variable-coefficient differential equations there is, however, a
persistent tension between generality, explicit dependence on the
coefficients, accuracy, and computational complexity. A practically
useful approximation should be written explicitely in term of the coefficients of the PDE. Recent work in this direction within the Chernoff framework includes contributions by Galkin, Remizov, and
Vedenin. In particular, Remizov and Vedenin construct explicit
approximations for parabolic equations from proper Riemann integrals over
a bounded interval and prove accelerated convergence under appropriate
regularity assumptions
\cite{RemizovVedenin,VedeninRemizov26}.

The classical Chernoff
theorem does not provide the rate of convergence of the approximation to the semigroup.
Galkin and Remizov \cite{GR25} showed that, under the general Chernoff
framework alone, there is no universal convergence rate. They construct
natural examples in which the Chernoff products converge arbitrarily
fast and others in which they converge arbitrarily slowly. 

Thus high regularity of the datum alone does not imply a fast Chernoff
convergence rate. Quantitative convergence requires additional structure in the
approximating family.

The examples of \cite{GR25} show, on the other hand, that there is no
abstract obstruction to arbitrarily fast Chernoff convergence. The
substantially harder question is whether such convergence can be
obtained from a meaningful and explicit construction for differential
operators arising in applications. This question was formulated by
Remizov as the problem of constructing \emph{super-fast} Chernoff
approximations in the open-problem collection \cite{OpenProblems}.
Problem~30 asks for a meaningful Chernoff approximation converging to a
$C_0$-semigroup at least exponentially, or for a proof that such a
construction is impossible in a meaningful setting. A central
requirement is that the approximating family must be given explicitly in
terms of the coefficients of the differential operator, rather than in
terms of the unknown semigroup itself, thereby avoiding any
\emph{circulus vitiosus}.

The practical motivation \cite{OpenProblems} is that the
gain in approximation accuracy might ultimately beat  the computational complexity growth. 

The present paper provides a constructive positive answer to this
super-fast approximation challenge. 
 Starting from a single
centered-difference Chernoff function, we combine its iterated products at
different refinement levels by Richardson extrapolation. At every fixed depth
this yields arbitrarily high algebraic orders of convergence. In the
analytic setting, the extrapolation depth grows with the
base resolution, and the resulting errors decay as 
$
 e^{-cn^2}
$
for the two second-order models and as 
$
 e^{-cn}
$
for the third-order model. The construction depends explicitly on the
coefficients of the differential operator and does not use the target
semigroup in defining the approximating function.

The techniques developed here have a historical precedent. In his
1911 paper on the arithmetical solution of differential equations,
Richardson used centered differences
and extrapolated across several mesh sizes to
improve accuracy \cite{Richardson1911}. The present paper may be viewed
as an operator-semigroup quantitative extension of this
principle.

To showcase the method, we treat the following three evolution
equations:
\begin{align*}
 \partial_tu(t,x)
   &=a(x)\partial_x^2u(t,x)+b(x)\partial_xu(t,x)+c(x)u(t,x),\\
 \partial_tu(t,x)
   &=i a(x)\partial_x^2u(t,x)+b(x)\partial_xu(t,x)+c(x)u(t,x),\\
 \partial_tu(t,x)
   &=a(x)\partial_x^3u(t,x)+b(x)\partial_xu(t,x)+c(x)u(t,x),
\end{align*}
all with $u(0,\cdot)=f$. The resulting formulas therefore approximate
the evolution solutions themselves. These three equations are
representative test cases rather than the scope of the techniques developped. The
tools developed below 
apply more broadly to linear
differential operators whenever the required dissipative or
symmetric--skew structure, stability, and regularity estimates are
satisfied.

\subsection{The rate problem}

The classical Chernoff theorem is qualitative: it guarantees convergence
of the iterated products but does not provide a convergence
rate. As the examples of Galkin and Remizov \cite{GR25} shows,
the absence of an universal rate is intrinsic.

A first quantitative approach, developed by Galkin and Remizov
\cite{GR22,GR25}, derives global convergence rates from sufficiently
accurate Taylor expansion matching of the Chernoff function with the semigroup. In a
general form, under the appropriate stability assumptions and a Taylor matching estimate on a suitable invariant regularity space:

\begin{equation}
\label{GRTayMatch}
\left\|
S(h)f-\sum_{k=0}^{m}\frac{h^kA^kf}{k!}
\right\|
\le
h^{m+1}\sum_{j=0}^{m+p}K_j(h)\|A^jf\|,
\end{equation}
they prove
$$
 \|S(t/n)^nf-e^{tA}f\|
 \le
 \frac{M_1M_2t^{m+1}e^{\omega t}}{n^m}
 \sum_{j=0}^{m+p}C_j(t/n)\|A^jf\|,
$$
whence the rate of convergence
\[
 \|S(t/n)^nf-e^{tA}f\|
 =
 O\!\left(n^{-(m-1+\alpha)}\right).
\]

Recent constructions of Remizov and Vedenin
\cite{VedeninRemizov26,RemizovVedenin} applied this local-matching
strategy  to variable-coefficient parabolic operators. For
\[
 A=a(x)\partial_x^2+b(x)\partial_x+c(x),
 \qquad a(x)\ge a_0>0,
\]
their Chernoff function is given explicitly in term of the coefficients
$a,b,c$ of $A$  by
\[
 (G(t)f)(x)
 =
 \int_{-1}^{1}
 \frac12\bigl(1+c(x)t\bigr)
 f\bigl(x+(6a(x)t)^{1/2}y+b(x)t\bigr)\,dy.
\]
Their constructions give accelerated convergence, including
an $O(n^{-2})$ rate under suitable regularity assumptions.

A parallel rate theory has recently been developed in the nonlinear
setting: Blessing, Jiang, Kupper, and Liang
\cite{BlessingJiangKupperLiang} obtain
\[
\|S(t)f-I(t/n)^nf\|\le Cn^{-\gamma}
\]
for Chernoff-type approximations of \emph{convex monotone} semigroups
on bounded Lipschitz data, with applications to Hamilton--Jacobi--Bellman
equations, stochastic control, Nisio semigroups, and limit theorems under
convex expectations. Their setting is fundamentally different:
monotonicity and convexity are structural ingredients there, whereas
Richardson coefficients generally alternate in sign and therefore need
not preserve either property; the present mechanism instead exploits
detailed asymptotic information about the iterated products.


The present paper exploits the linear structure at the level of the iterated products,
starting from a single centered-difference family whose symmetry in the
mesh parameter produces the finite even expansion of the approximating
generator:
\begin{equation}
\label{eqoExpanAh}
 A_h
 =
 A+\sum_{r=1}^{m-1}h^{2r}R_r+h^{2m}Q_m(h).
\end{equation}
Stability and Duhamel iteration are then 
used to get the product expansion
$$
 \mathcal S_2(t/N)^Nf
 =
 e^{tA}f
 +
 \sum_{r=1}^{m-1}
 \frac{E_{2r}(t)f}{N^{2r}}
 +
 R^{\mathrm{prod}}_{m,N}(t)f,
$$
with the pointwise remainder estimate
\begin{equation}
\label{IntrproRem0}
 \|R^{\mathrm{prod}}_{m,N}(t)f\|
 \le C_m e^{\Omega_m t}
 \frac{t^{2m+1}}{N^{2m}}\|f\|_{\mathcal D_m}
 \qquad (t/N\le h_0).
\end{equation}
Richardson extrapolation combines the same base family at several step
counts and cancels \warning{successively} the terms
$
 N^{-2},\,N^{-4},\,\ldots,\,N^{-2m+2},
$
thereby producing the convergence rate
$
 O(N^{-2m}).
$
The mechanism may therefore be summarized as
\[
 \boxed{\begin{aligned}
 &\mbox{even local generator expansion}\\
 &\qquad\qquad \mbox{ + stability}
 \end{aligned}}
 \Longrightarrow
 \boxed{\text{even product expansion}}
 \Longrightarrow
 \boxed{\text{Richardson cancellation}}.
\]



One difference from the Galkin--Remizov theory that is worth
emphasizing is that our argument does not require the 
Chernoff function to satisfy a Taylor-matching estimate whose remainder
is controlled explicitly by powers of the generator as in \eqref{GRTayMatch}.  
%
In this respect, the extrapolation mechanism developed here
allows greater flexibility.

\warning{
At the algebraic level, this cancellation principle is closely related
to Richardson's original extrapolation method \cite{Richardson1911},
to Romberg extrapolation \cite{Romberg1955}, and to symmetric high-order
composition methods \cite{Yoshida}. 
}

\subsection{Main results}

The results of the paper can be summarized in the following two
theorems. For, let $m\ge1$, $a,b,c$ be real-valued functions with $a(x)\ge a_0>0.$ and introduce the following:

\subsubsection{Notations}
\label{subsec:Notations}
\begin{itemize}
\item[$\bullet$] \emph{In the Schr\"odinger case:}

\[
(\mathcal J_2f)(y)
=
a(\chi_2(y))^{-1/4}f(\chi_2(y)),\qquad
\chi_2=\eta_2^{-1} ,\qquad \eta_2(x)=\int_0^x a(s)^{-1/2}\,ds.
\]

Then
\[
\mathcal J_2A_0\mathcal J_2^{-1}
=
iD_y^2+\frac12(vD_y+D_yv)+w,
\]
with
\[
v=
\left(\frac b{\sqrt a}\right)\circ\chi_2,
\qquad 
w=
\left[
c-\frac12b'
+\frac12b\frac{a'}a
+i\left(
\frac14a''
-\frac{3}{16}\frac{(a')^2}{a}
\right)
\right]\circ\chi_2.
\]

\item[$\bullet$] \emph{In the cubic case:}
\[
\alpha
=
-\frac23a''
+\frac59\frac{(a')^2}{a},
\qquad
\beta
=
-\frac13a'''
+\frac89\frac{a'a''}{a}
-\frac59\frac{(a')^3}{a^2},
\]
\[
d=b-\alpha,
\qquad
V=(c-\beta)-\frac12d'
+\frac12d\frac{a'}a,
\]
\[
\eta_3(x)=\int_0^x a(s)^{-1/3}\,ds,
\qquad
\chi_3=\eta_3^{-1},
\]
and
\[
(\mathcal J_3f)(y)
=
a(\chi_3(y))^{-1/3}f(\chi_3(y)).
\]
Then
\[
\mathcal J_3A_0\mathcal J_3^{-1}
=
D_y^3+\frac12(vD_y+D_yv)+w,
\]
where
\[
v=(d\,a^{-1/3})\circ\chi_3,
\qquad
w=V\circ\chi_3.
\]

\end{itemize} 
Let $p=2$ in the two second-order cases and $p=3$ in the
third-order case.

\noindent
For the parabolic model set $\mathcal J=I$; in the other two cases
set $\mathcal J=\mathcal J_p$, $p=2,3$.

\noindent
Let $\mathcal M_{2m,n}(t)$ denotes the
Richardson multiproduct for the normalized generator $\widetilde A_0\coloneqq \mathcal J A_0\mathcal J^{-1}$ (see \eqref{eq:even-richardson-multiproduct}), and define the
original-coordinate approximation by
\begin{equation}
\mathcal M_{2m,n}^{\rm orig}(t)
:=
\mathcal J^{-1}
\mathcal M_{2m,n}(t)
\mathcal J.
\label{eq:original-richardson}
\end{equation}

\subsubsection{Results} \ \\
\emph{\textbf{Sobolev class.}} 
The first concerns rate of convergence in Sobolev
class.
\begin{theorem}[Unified fixed-depth Sobolev rates]
\label{thm:unified-sobolev-rates}
Let $A_0$ be one of the following three model
generators.

\smallskip

\noindent
\emph{(P) Parabolic second order:}
\[
A_0=aD^2+bD+c,
\]
with 
\[
a\in W^{2m+4,\infty}(\R),\qquad
b\in W^{2m+3,\infty}(\R),\qquad
c\in W^{2m+2,\infty}(\R).
\]

\smallskip

\noindent
\emph{(S) Schr\"odinger-type second order:}
\[
A_0=iaD^2+bD+c,
\]
with
\[
a\in W^{4m+2,\infty}(\R),\qquad
b\in W^{4m+1,\infty}(\R),\qquad
c\in W^{4m,\infty}(\R).
\]

\smallskip

\noindent
\emph{(D) Third-order dispersive:}
\[
A_0=aD^3+bD+c,
\]
with 
\[
a\in W^{5m+3,\infty}(\R),\qquad
b\in W^{5m+1,\infty}(\R),\qquad
c\in W^{5m,\infty}(\R).
\]

Let $p\in\{2,3\}$ denote the differential order and let
$\mathcal M_{2m,n}(t)$ denote the depth-$m$ Richardson multiproduct
associated with the corresponding centered-difference Chernoff function.

Then the following hold.

\emph{(a). Parabolic model (P).}
For  every $\delta>0$, setting
\[
\eta_\delta:=\min\{\delta/2,1\},
\]
there exist constants $h_0, C_{m,\delta}, \Omega_{m,\delta}\ge0$ such that
\[
\|\mathcal M_{2m,n}(t)f-e^{tA_0}f\|_{L^2}
\le
C_{m,\delta}e^{\Omega_{m,\delta}t}
\frac{t^{2m+\eta_\delta}}{n^{2m}}
\|f\|_{H^{2m+\delta}},
\qquad
0<t\le h_0n.
\]

\emph{(b). Group cases (S) and (D).}
For the Schr\"odinger-type and third-order dispersive models, let
$\mathcal J$ denote the corresponding normalizing conjugation. Then
there exist constants $h_0, C_m,\Omega_m\ge0$ such that
\[
\left\|
\mathcal J^{-1}\mathcal M_{2m,n}(t)\mathcal J f
-
e^{tA_0}f
\right\|_{L^2}
\le
C_me^{\Omega_mt}
\frac{t^{2m+1}}{n^{2m}}
\|f\|_{H^{(p+2)m}},
\qquad
0\le t\le h_0n.
\]

Thus, at every fixed depth $m$, all three models converge at rate
\[
O(n^{-2m}),
\]
on the respective datum scales
\[
H^{2m+\delta},\qquad H^{4m},\qquad H^{5m}.
\]
\end{theorem}

Theorem~\ref{thm:unified-sobolev-rates} collects the fixed-depth
consequences proved separately for the three models. The parabolic
estimate follows from analytic-semigroup smoothing and scale allocation,
whereas the two nonsmoothing estimates follow from the common
complete-word Sobolev-group argument. 

\medskip
\emph{\textbf{Analytic class.}} 
The second main result concerns analytic data and allows the
extrapolation depth to grow with the base resolution. 
%
For ${\rho}>0$ write\footnote{For $z\in \C$, $\Re z$ (resp. $\Im z$) denotes its real (resp. imaginary) part.}
\[
\mathscr{S}_{\rho}:=\{z\in\C:\ |\Im z|<{\rho}\},
\]
and denote by $\mathcal{H}_{\rho}$ the Banach algebra of bounded
holomorphic functions on $\mathscr S_{\rho}$, equipped with
\[
\|g\|_{\mathcal{H}_{\rho}}
:=
\sup_{z\in \mathscr S_{\rho}}|g(z)|;
\]
by
\[
\mathcal A_\rho
:=
\left\{
f:\ e^{\rho|D|}f\in L^2(\R)
\right\},
\qquad
\|f\|_{\rho}
:=
\|e^{\rho|D|} f\|_{L^2}.
\]

\warning{
\begin{theorem}[Unified super-fast analytic Richardson rates]
\label{thm:unified-analytic-rate}

Let $A_0$ be one of
\[
aD^2+bD+c,
\qquad
iaD^2+bD+c,
\qquad
aD^3+bD+c,
\]
and set $p=2$ in the two second-order cases and $p=3$ in the
third-order case.

Assume that there exists $\rho>0$ such that
\[
a,b,c\in \mathcal{H}_{\rho},
\]
that $a$ and $b$ are real-valued on $\R$, and that
\[
a(x)\ge a_0>0,
\qquad x\in\R.
\]
\bwarning{
No separate upper bound on $a$ is required, since
$a\in H^\infty(S_{R_0})$ already implies
\[
\sup_{x\in\R}a(x)
\le
\|a\|_{\infty,R_0}<\infty.
\]
The coefficient $c$ may be complex-valued.
}

\gwarning{
  There exists $\rho_*>0$ such that all
normalized coefficients are bounded and holomorphic on $\mathscr S_{\rho_*}$,
and the normalized transport coefficient $v$ is real on $\R$.

Define the pulled-back analytic spaces
\[
\mathcal X_\rho^{(0)}
:=
\mathcal J^{-1}\mathcal A_\rho,
\qquad
\|f\|_{\rho,0}
:=
\|\mathcal Jf\|_\rho,
\]
and the corresponding original-coordinate Richardson approximation
\[
\mathcal M_{2m,n}^{(0)}(t)
:=
\mathcal J^{-1}
\mathcal M_{2m,n}(t)
\mathcal J.
\]

Fix $T>0$ and
\[
0<\rho_*<\rho_1e^{-\kappa T}
<
\rho_1<\rho_0<R_*,
\]
where $\kappa=0$ in the parabolic case and $\kappa\ge0$ in the two
nonsmoothing cases.  Then there exist constants
$C_T,K_*,\Gamma>0$ such that
\[
\left\|
\mathcal M_{2m,n}^{(0)}(t)f-e^{tA_0}f
\right\|_{\rho_*,0}
\le
C_Te^{\Gamma t}t
\left(
\frac{K_*t^2m^{p-1}}{n^2}
\right)^m
\|f\|_{\rho_0,0},
\qquad
0\le t\le T,
\]
for all sufficiently large $n$.

Consequently, if
\[
m(n)
=
\left\lfloor
\eta n^{2/(p-1)}
\right\rfloor
\]
with $\eta>0$ sufficiently small, then there are
$C_T,c_T>0$ such that
\[
\left\|
\mathcal M_{2m(n),n}^{(0)}(t)f-e^{tA_0}f
\right\|_{\rho_*,0}
\le
C_Te^{\Gamma t}t
e^{-c_Tn^{2/(p-1)}}
\|f\|_{\rho_0,0},
\qquad
0\le t\le T.
\]
}
For every $T>0$, there exists $C_T,c_T,K_*,\Gamma, N_T>0$ and $0<\rho_*<\rho$ such that for any natural $n>N_T$ and any $f\in \mathcal{J}^{-1}\mathcal{A}_{\rho_*}$:
$$
\left\|
\mathcal{J}^{-1}\mathcal M_{2m(n),n}(t)\mathcal{J} f-e^{tA_0} f
\right\|_{\rho_*}
\le
C_Te^{\Gamma t}t
e^{-c_Tn^{2/(p-1)}}
\|\mathcal{J} f\|_{\rho_*},
\qquad
0\le t\le T.
$$

Hence the two second-order models have rate
\[
e^{-c_T n^2},
\]
whereas the third-order model has rate
\[
e^{-c_T n}.
\]
\end{theorem}
}

{
\begin{theorem}[Unified super-fast analytic Richardson rates]
\label{thm:unified-analytic-rate}
Let $A_0$ be one of the three model generators
\[
aD^2+bD+c,
\qquad
iaD^2+bD+c,
\qquad
aD^3+bD+c,
\]
and let
\[
p=
\begin{cases}
2,&\text{for the two second-order models},\\
3,&\text{for the third-order model}.
\end{cases}
\]
Assume that, for some $\varrho_0>0$,
\[
a,b,c\in\mathcal H_{\varrho_0},
\]
 $a$ and $b$ are real-valued on $\R$, and 
\[
a(x)\ge a_0>0,
\qquad x\in\R.
\]
Set
%
%
\begin{equation}
a_\infty:=\|a\|_{\mathcal H_{\varrho_0}},
\quad
r_a:=\frac{a_0\varrho_0}{4a_\infty},\quad 
%
M_p
=
\left(\frac2{a_0}\right)^{1/p},\quad 
m_p
=
a_\infty^{-1/p}
\cos\left(
\frac1p
\arccos\frac{a_0}{2a_\infty}
\right),
\label{eq:Mp-explicit} \multilabel{eq:mp-explicit}
\end{equation}
and the Liouville distortion ratio
\begin{equation}
\gamma_p
:=
\frac{m_p}{M_p}
=
\left(
\frac{a_0}{2a_\infty}
\right)^{1/p}
\cos\left(
\frac1p
\arccos\frac{a_0}{2a_\infty}
\right)>0.
\label{eq:liouville-distortion} \multilabel{eq:main-gamma-p}
\end{equation}
Fix $T>0$ and an initial analytic radius
\[
0<\rho_{\rm in}<r_a.
\]
In the parabolic case let $\kappa=0$ and $\gamma_p=1$, and  in the two nonsmoothing cases,
let $\kappa>0$ be a strip-loss constant.
for the normalized operator $\widetilde A_0$.

\noindent
Choose the final radius as follows:
\begin{equation}
0<\rho_{\rm out}
<
\gamma_p e^{-\kappa T}\rho_{\rm in}.
\label{eq:original-compatible-radii}
\end{equation}
Then there exist constants
\[
C_T,c_T,\Gamma>0,
\qquad
N_T\in\N,
\]
and $\eta_T>0$ such that, with
\begin{equation}
m(n)
=
\left\lfloor
\eta_Tn^{2/(p-1)}
\right\rfloor,
\label{eq:main-growing-depth}
\end{equation}
one has
\begin{equation}
\left\|
\mathcal M_{2m(n),n}^{\rm orig}(t)f
-
e^{tA_0}f
\right\|_{\rho_{\rm out}}
\le
C_Te^{\Gamma t}\,
t\,
e^{-c_Tn^{2/(p-1)}}
\|f\|_{\rho_{\rm in}}\,,
\label{eq:main-original-analytic-rate}
\end{equation}
for every $n\ge N_T$, every $0\le t\le T$, and every
$f\in\mathcal A_{\rho_{\rm in}}$.

\noindent
Consequently, the parabolic and Schr\"odinger-type second-order
models converge with the rate
\[
O(e^{-c_Tn^2}),
\]
while the third-order dispersive model converges with the rate
\[
O(e^{-c_Tn}).
\]
\end{theorem}
}

\bwarning{
\begin{remark}
For $p\in\{2,3\}$, choose $r_p>0$ such that
\[
q_p:=a^{-1/p}\in H^\infty(S_{r_p}),
\qquad
m_p:=\inf_{S_{r_p}}\Re q_p>0,
\qquad
M_p:=\sup_{S_{r_p}}|q_p|<\infty.
\]
Then
\[
m_p|\Im z|
\le
|\Im\eta_p(z)|
\le
M_p|\Im z|,
\qquad z\in S_{r_p}.
\]
Consequently,
\[
\chi_p(S_{\rho_y})\subset S_{\rho_y/m_p},
\qquad
0<\rho_y<m_pr_p,
\]
and
\[
\eta_p(S_{\rho_x})\subset S_{M_p\rho_x},
\qquad
0<\rho_x<r_p.
\]

Hence, whenever
\[
0<\rho_y<m_p\rho_x,
\]
the normalizing map satisfies
\[
\mathcal J_p:\mathcal A_{\rho_x}\longrightarrow\mathcal A_{\rho_y},
\]
whereas, whenever
\[
0<M_p\rho_x<\rho_y,
\]
one has
\[
\mathcal J_p^{-1}:\mathcal A_{\rho_y}\longrightarrow\mathcal A_{\rho_x}.
\]
Both maps are bounded between the indicated analytic spaces.

Therefore, for every $T>0$ and every initial radius
$0<\rho_0^x<r_p$, any final radius satisfying
\[
0<\rho_*^x
<
\frac{m_p}{M_p}e^{-\kappa T}\rho_0^x
\]
is admissible.  More precisely, one may choose radii
\[
0<\rho_0^y<m_p\rho_0^x,
\qquad
0<\rho_1^y<\rho_0^y,
\]
and
\[
M_p\rho_*^x
<
\rho_*^y
<
\rho_1^ye^{-\kappa T}.
\]
\end{remark}
}

\bwarning{
\begin{theorem}[Unified super-fast analytic Richardson rate]
\label{thm:unified-analytic-rate}
Let $A_0$ be one of the three model generators above and let
$p\in\{2,3\}$ denote its principal differential order.

Assume that the coefficients are analytic in the following sense.

\emph{(P) Parabolic case.}
With
\[
\phi=b-a',
\qquad
p_0=c+\frac12(a''-b'),
\]
assume that $a$ is real,
\[
a(x)\ge a_0>0,
\qquad
a,\phi,p_0\in\mathcal W_R
\]
for some $R>0$. Choose $0<\rho_0<R$ sufficiently small that
\[
\rho_0 M_1(a;\rho_0)\le\frac{a_0}{2}.
\]

\emph{(S)--(D) Nonsmoothing cases.}
Assume that $a,b,c$ are such that the normalizing conjugation described
below transforms $A_0$ into
\[
A
=
c_pD^p+\frac12(vD+Dv)+w,
\qquad
c_2=i,\quad c_3=1,
\]
where $v$ is real and
\[
v,w\in\mathcal W_R
\]
for some $R>0$. Fix $0<\rho_0<R$.

Then, for every $T>0$ and every admissible smaller strip
$0<\rho_*<\rho_0$, the depth-$m$ product expansion satisfies
\[
C^{\rm an}_{m,T}
\le
C_TK_T^m\Gamma((p+1)m+1).
\]
Moreover, after Richardson cancellation,
\[
\|\mathcal M_{2m,n}(t)f-e^{tA}f\|_{\rho_*}
\le
C_Te^{\Gamma t}t
\left(
\frac{K_*t^2m^{p-1}}{n^2}
\right)^m
\|f\|_{\rho_0},
\qquad
0\le t\le T.
\]

Consequently, for $\eta>0$ sufficiently small and
\[
m(n)=\left\lfloor\eta n^{2/(p-1)}\right\rfloor,
\]
there exist constants $C_T,c_T>0$ such that, for all sufficiently large
$n$,
\[
\|\mathcal M_{2m(n),n}(t)f-e^{tA}f\|_{\rho_*}
\le
C_Te^{\Gamma t}t\,
e^{-c_Tn^{2/(p-1)}}\|f\|_{\rho_0},
\qquad
0\le t\le T.
\]
After conjugating back when necessary, the same rate holds for the
original generator $A_0$.

Hence
\[
p=2
\quad\Longrightarrow\quad
e^{-cn^2},
\]
for the parabolic and Schr\"odinger-type models, whereas
\[
p=3
\quad\Longrightarrow\quad
e^{-cn},
\]
for the third-order dispersive model.
\end{theorem}

Theorem~\ref{thm:unified-analytic-rate} combines the uniform analytic
propagation estimates, the all-order product expansion, the exact
Richardson-weight estimate, and the growing-depth optimization proved
later. The common exponent
\[
\frac{2}{p-1}
\]
results from balancing the $n^{-2m}$ Richardson gain against the
factorial growth of the all-order constants. Thus the Gaussian-type and
exponential rates are two instances of the same formula.
}

\subsection*{Organization}

Section~\ref{sec:preliminaries} recalls the basic semigroup and Chernoff
framework. Section~\ref{sec:abstract-richardson} develops the abstract
Richardson--Chernoff extrapolation theory on general power scales,
including the even-power specialization, stability estimates, and
resolvent approximation.

Section~\ref{sec:centered-differences} constructs the centered difference
Chernoff approximation, proves their stability and finite even generator
expansions, and derives the corresponding product expansions by Duhamel
iteration. It then treats the three model classes; their fixed-depth
consequences together prove
Theorem~\ref{thm:unified-sobolev-rates}.

Section~\ref{sec:analytic-regime} develops the analytic-strip framework,
uniform propagation estimates, all-order defect bounds, and
Richardson-weight estimates. These results lead to the growing-depth
optimization and prove
Theorem~\ref{thm:unified-analytic-rate}.


\section{Preliminaries}\label{sec:preliminaries}
Let us recall the basic semigroup and Chernoff framework; see \cite{Butko,EngelNagel,Remizov25} for more details.
Let $X$ be a complex Banach space. 

\begin{definition}[$C_0$-semigroup]\label{def:c0-semigroup}
A family
$(T(t))_{t\ge0}\subset\mathcal L(X)$ is a semigroup if
\[
 T(0)=I,\qquad T(t+s)=T(t)T(s).
\]
The family is called strongly continuous if $\lim_{t\to0^+}T(t)f=f\quad(f\in X)$.\\
A strongly continuous semigroup is called \emph{$C_0$-semigroup} or \emph{evolution semigroup} or \emph{evolution operator} or \emph{dynamical system}.\\
Its generator is
\[
 Af=\lim_{t\to 0^+}\frac{T(t)f-f}{t},\quad \mbox{with domain}\quad
 D(A)=\left\{f\in X: \lim_{t\to 0^+}\frac{T(t)f-f}{t} \text{ exists}\right\}.
\]
\end{definition}

\begin{definition}[Closed operator]\label{def:closed-operator}
%
%
The operator $(A,D(A))$ is said to be \emph{closed} if its graph

\[
\mathcal G(A)
:=
\{(x,Ax):x\in D(A)\}
\subset X\times X
\]

is a closed subspace of \(X\times X\).

Equivalently, \(A\) is closed if, whenever a sequence \((x_n)_{n\geq1}\subset D(A)\) satisfies

\[
x_n\longrightarrow x
\quad\text{in }X,
\qquad
Ax_n\longrightarrow y
\quad\text{in }X,
\]

for some \(x,y\in X\), then

\[
x\in D(A)
\qquad\text{and}\qquad
Ax=y.
\]

When \(A\) is closed, its domain \(D(A)\), equipped with the graph norm

\[
\|x\|_{D(A)}
:=
\|x\|_X+\|Ax\|_X,
\]

is itself a Banach space.
\end{definition}

{
\begin{remark}\label{rem:model-closures}
For each model we initially regard the differential expression on
\(C_c^\infty(\R)\subset L^2(\R)\). Under the coefficient assumptions
imposed below, these minimal operators are closable and their closures
are the natural Sobolev realizations
\[
\begin{aligned}
A\coloneqq aD^2+bD+c, \qquad D(\overline{A})&=H^2(\R),\\
A\coloneqq iaD^2+bD+c, \qquad D(\overline{A})&=H^2(\R),\\
A\coloneqq aD^3+bD+c, \qquad D(\overline{A})&=H^3(\R).
\end{aligned}
\]
We shall abuse notation and denote the closure of the operator $A$ again by $A$.

\end{remark}
}

The associated abstract Cauchy problem is
\begin{equation}
 \begin{cases}
  u'(t)=Au(t),&t>0,\\
  u(0)=f,
 \end{cases}
 \qquad u(t)=T(t)f=e^{tA}f.
 \tag{CP}\label{eq:cauchy-problem}
\end{equation}
We have the following (see \cite[Theorem~II.6.7]{EngelNagel}):

\begin{theorem}[Classical well-posedness and $C_0$-semigroup generation]
\label{thm:acp-semigroup}
Let $(A,D(A))$ be a closed, densely defined linear operator on a Banach
space $X$, and assume that its resolvent set is nonempty: 
$
\rho(A)\neq\varnothing.
$
Then the following assertions are equivalent:
%
\begin{enumerate}
\item
For every initial datum $f\in D(A)$, the abstract Cauchy problem \textup{(CP)}
admits a unique classical solution, that is,
\[
u\in C^1([0,\infty),X),
\qquad
u(t)\in D(A)\quad\text{for every }t\ge0.
\]

\item
$(A,D(A))$ is the generator of a strongly continuous semigroup
$(T(t))_{t\ge0}$ on $X$.
\end{enumerate}
Whenever these equivalent conditions hold, the unique classical solution of the evolution associated with
\textup{(CP)} is represented by the semigroup orbit
\[
u(t)=T(t)f=e^{tA}f.
\]

\end{theorem}

Structurally, classical semigroup theory organizes the
semigroup, Chernoff approximations, its infinitesimal generator, and its resolvent into 
the following diagram: differentiation at the origin of the evolution yields
the generator, the Laplace transform of the evolution yields the
resolvent, and approximation/inversion formulas recover the evolution
from infinitesimal or resolvent data.  

\begin{figure}[htbp]
    \centering
    \includegraphics[scale=.4 ]{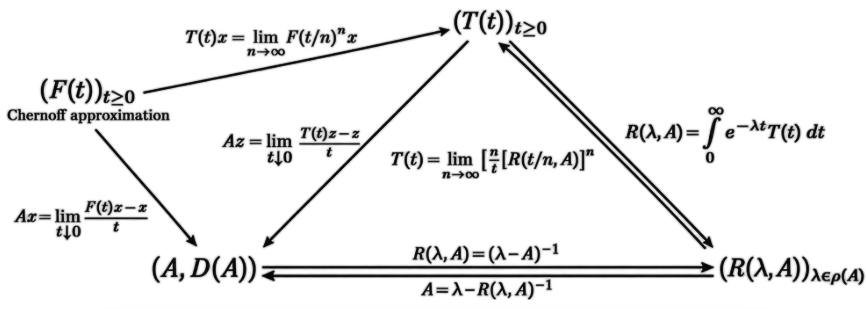}
    \caption{Diagram summarizing $C_0$-semigroups philosophy} 
    \label{fig:my_image}
\end{figure}
 The present paper refines the
Chernoff-approximation side of this classical picture quantitatively:
starting from a family tangent to the generator, it develops
Richardson cancellation for the iterated products and then transfers
fixed-depth information to the resolvent through the same Laplace
transform mechanism.

\begin{definition}[Chernoff tangency]\label{def:chernoff-tangency}
A strongly continuous family
$S:[0,\infty)\to\mathcal L(X)$ with $S(0)=I$ is
\emph{Chernoff tangent} to $A$ on a core $\mathcal D\subset D(A)$ if
\[
 \lim_{h\to 0^+}\frac{S(h)f-f}{h}=Af,\qquad f\in\mathcal D.
\]
\end{definition}
We use the standard (see \cite{Butko}): 

\begin{theorem}[Chernoff theorem]
\label{thm:cher-thm}
Let $S:[0,\infty)\to\mathcal L(X)$ be a family of bounded operator \emph{Chernoff tangent} to 
a closed linear (bounded or unbounded) operator $(A,D(A))$ which generate a $C_0$-semigroup $e^{t A}$, and 
\[
 \|S(h)^k\|_{\mathcal L(X)}\le Me^{\omega kh},\qquad
 h\ge0,\quad k\in\mathbb N,
\]
for some $M\ge 0$ and $\omega\in \R$. Then
\[
 e^{t A}f=\lim_{n\to \infty}S(t/n)^nf
\]
for every $f\in X$, uniformly for $t$ in compact subsets of
$[0,\infty)$.

Such a function $S$ is then called a Chernoff function for the operator $A$, and $S$ is called Chernoff-equivalent to $C_0$-semigroup $e^{t A}$.
\end{theorem}


\section{Abstract Richardson--Chernoff extrapolation}\label{sec:abstract-richardson}

We start with a purely algebraic mechanism, independent of any
differential-operator structure: if the products of a Chernoff function admit an asymptotic expansion in a prescribed strictly increasing sequence of powers of the step parameter, then a finite
linear combination of such products at different  step reduces significantly the error by canceling the polynomial part. 

\subsection{Chernoff products and asymptotic power scales}\label{subsec:power-scales}

Let
\[
\boldsymbol\nu=(\nu_1,\nu_2,\ldots),\qquad
0<\nu_1<\nu_2<\cdots ,
\]
and let $\mathcal S$ be a Chernoff function for the generator $A$ of
$T(t)=e^{tA}$. We assume the power bound
\begin{equation}
 \|\mathcal S(h)^k\|\le M e^{\omega kh},
 \qquad h\ge0,\quad k\in\mathbb N,
\label{eq:chernoff-power-bound}
\end{equation}
which implies
\begin{equation}
 \|T(t)\|\le M e^{\omega t},\qquad t\ge0 .
\label{eq:semigroup-growth-bound}
\end{equation}

At a fixed depth $m$, let $\mathcal D_m$ be a dense Banach subspace
of $X$, equipped with a norm
$\|\cdot\|_{\mathcal D_m}$. 
%
Suppose there are operator-valued functions
$E_{\nu_1}(t),\ldots,E_{\nu_{m-1}}(t)$, constants
$h_0,C_m>0$ and $\Omega_m\ge0$, and remainders
$R^{\mathrm{prod}}_{m,N}(t)$ such that, whenever
$N\ge1$ and $0\le t\le h_0N$,
\begin{equation}
 \mathcal S(t/N)^Nf
 =
 T(t)f+
 \sum_{r=1}^{m-1}\frac{E_{\nu_r}(t)f}{N^{\nu_r}}
 +R^{\mathrm{prod}}_{m,N}(t)f,
 \qquad
 \|R^{\mathrm{prod}}_{m,N}(t)f\|_X
 \le
 C_m e^{\Omega_m t}
 \frac{t^{\nu_m+1}}{N^{\nu_m}}
 \|f\|_{\mathcal D_m}.
\label{eq:product-expansion}
\end{equation}


\subsection{Richardson extrapolation on a prescribed power scale}\label{subsec:general-richardson}

Fix distinct positive refinement factors $j_1,\dots,j_m$, 
and let
$c_1,\dots,c_m$ so that
\begin{equation}
 \sum_{\ell=1}^m c_\ell=1,\qquad
 \sum_{\ell=1}^m c_\ell j_\ell^{-\nu_r}=0,
 \quad r=1,\dots,m-1 .
\label{eq:general-richardson-system}
\end{equation}
Equivalently, $(c_\ell)$ solves the generalized Vandermonde system
\[
\begin{pmatrix}
1&\cdots&1\\
j_1^{-\nu_1}&\cdots&j_m^{-\nu_1}\\
\vdots&&\vdots\\
j_1^{-\nu_{m-1}}&\cdots&j_m^{-\nu_{m-1}}
\end{pmatrix}
\begin{pmatrix}c_1\\ \vdots\\ c_m\end{pmatrix}
=
\begin{pmatrix}1\\0\\\vdots\\0\end{pmatrix}.
\]
For strictly ordered positive nodes $x_\ell=j_\ell^{-1}$ and strictly
increasing exponents $0,\nu_1,\ldots,\nu_{m-1}$, this matrix is
nonsingular.

\begin{theorem}[Richardson--Chernoff extrapolation on a prescribed power scale]
\label{thm:general-richardson}
Assume that $\mathcal S$ satisfies the Chernoff conditions above and
that its products admit the depth-$m$ expansion \eqref{eq:product-expansion} on the
power scale $\boldsymbol\nu$.  Let $c_\ell$ satisfy \eqref{eq:general-richardson-system} and
define
\[
 \mathcal M_{\boldsymbol\nu,n}(t)
 :=
 \sum_{\ell=1}^m c_\ell
 \mathcal S\!\left(\frac{t}{j_\ell n}\right)^{j_\ell n}.
\]
Let $j_*:=\min_{1\le\ell\le m}j_\ell$. Then, for every
$f\in\mathcal D_m$ and every $0\le t\le h_0j_*n$,
\[
 \|\mathcal M_{\boldsymbol\nu,n}(t)f-T(t)f\|_X
 \le
 C'_m e^{\Omega_m t}
 \frac{t^{\nu_m+1}}{n^{\nu_m}}
 \|f\|_{\mathcal D_m},
\]
where $C'_m=C_m\sum_{\ell=1}^m|c_\ell|j_\ell^{-\nu_m}$.
\end{theorem}

\begin{proof}
Apply \eqref{eq:product-expansion} with $N=j_\ell n$, multiply by $c_\ell$, and sum over
$\ell$.  The semigroup terms combine to $T(t)f$ because
$\sum_\ell c_\ell=1$, and the coefficient of
$n^{-\nu_r}E_{\nu_r}(t)f$ is
$\sum_\ell c_\ell j_\ell^{-\nu_r}=0$.  Hence
\[
 \mathcal M_{\boldsymbol\nu,n}(t)f-T(t)f
 =
 \sum_{\ell=1}^m c_\ell R^{\mathrm{prod}}_{m,j_\ell n}(t)f,
\]
and the pointwise bound in \eqref{eq:product-expansion} gives the asserted estimate.
\end{proof}


\subsection{The even-power specialization and explicit weights}\label{subsec:even-richardson}

\begin{theorem}[Richardson--Chernoff extrapolation]
\label{thm:abstract}
Assume that $\mathcal S_2$ satisfies the Chernoff conditions above and
that its products admit the even expansion \eqref{eq:product-expansion} through depth
$m$, with $\nu_r=2r$.  For $j=1,\dots,m$, set
\begin{equation}
c_j^{(m)} = (-1)^{m-j}\,\frac{2j^{2m}}{(m-j)!\,(m+j)!},
\label{eq:even-richardson-weights}
\end{equation}
and define the multiproduct
\begin{equation}
\mathcal M_{2m,n}(t) := \sum_{j=1}^mc_j^{(m)}\,\mathcal S_2\!\left(\frac t{jn}\right)^{jn}.
\label{eq:even-richardson-multiproduct}
\end{equation}
Then for every $f\in\mathcal D_m$ and every $0\le t\le h_0n$,
\begin{equation}
\big\|\mathcal M_{2m,n}(t)f-e^{tA}f\big\|_X
\le C'_m e^{\Omega_m t}
\frac{t^{2m+1}}{n^{2m}}
\|f\|_{\mathcal D_m}.
\label{eq:even-richardson-rate}
\end{equation}
\end{theorem}

\begin{proof}
Write $x_j=1/j^2$. By Lagrange interpolation formula, the solution of the system 
\begin{equation}
\sum_{j=1}^mc_j^{(m)}=1,\qquad
\sum_{j=1}^m\frac{c_j^{(m)}}{j^{2r}}=0,
\quad r=1,\dots,m-1.
\label{eq:even-richardson-moment-system}
\end{equation}
is given by the product formula
$c_j^{(m)}=\prod_{\ell\ne j}j^2/(j^2-\ell^2)$ which simplifies to \eqref{eq:even-richardson-weights}.

Applying \eqref{eq:product-expansion} with $N=jn$ and multiplying by $c_j^{(m)}$,
\[
c_j^{(m)}\mathcal S_2(t/(jn))^{jn}f = c_j^{(m)}e^{tA}f + \sum_{r=1}^{m-1}\frac{c_j^{(m)}E_{2r}(t)f}{j^{2r}n^{2r}} + c_j^{(m)}R^{\mathrm{prod}}_{m,jn}(t)f.
\]
Summing over $j=1,\dots,m$ and using \eqref{eq:even-richardson-moment-system}, the $e^{tA}f$ terms combine
to exactly $e^{tA}f$, and each of the $m-1$ intermediate sums
$\sum_jc_j^{(m)}j^{-2r}E_{2r}(t)f$ vanishes identically yielding:
\[
\mathcal M_{2m,n}(t)f - e^{tA}f = \sum_{j=1}^mc_j^{(m)}R^{\mathrm{prod}}_{m,jn}(t)f.
\]
Applying \eqref{eq:product-expansion}  at each $N=jn$ and $0\le t\le h_0n$,
\[
\|R^{\mathrm{prod}}_{m,jn}(t)f\|_X
\le C_m e^{\Omega_m t}
\frac{t^{2m+1}}{j^{2m}n^{2m}}
\|f\|_{\mathcal D_m}.
\]
Therefore
\[
\big\|\mathcal M_{2m,n}(t)f-e^{tA}f\big\|_X
\le C_m e^{\Omega_m t}
\frac{t^{2m+1}}{n^{2m}}
\Big(\sum_{j=1}^m\frac{|c_j^{(m)}|}{j^{2m}}\Big)
\|f\|_{\mathcal D_m},
\]
which is \eqref{eq:even-richardson-rate} with $C'_m=C_m\sum_j|c_j^{(m)}|j^{-2m}$.
\end{proof}

\begin{remark}[Recursive form]\label{rem:recursive-richardson}
For an integer refinement factor $\rho\ge2$, the same cancellations may
be generated by
\[
 \mathscr R_{2k+2}^{(n)}(t)
 =\frac{\rho^{2k}\mathscr R_{2k}^{(\rho n)}(t)
        -\mathscr R_{2k}^{(n)}(t)}
 {\rho^{2k}-1},\qquad \mathscr R_2^{(n)}(t)=\mathcal S_2(t/n)^n.
\]
Then \eqref{eq:product-expansion} yields, pointwise in $t$,
\[
 \|\mathscr R_{2m}^{(n)}(t)f-e^{tA}f\|_X
 \le C_k e^{\Omega_m t}
 \frac{t^{2m+1}}{n^{2m}}\|f\|_{\mathcal D_m}.
\]
\end{remark}

\subsection{Stability of the Richardson extrapolation}\label{subsec:richardson-stability}

The Chernoff power bound \eqref{eq:chernoff-power-bound} upgrades immediately to every constituent of the
Richardson extrapolation. With $h=t/(jn)$ and $k=jn$, so that $kh=t$,
\begin{equation}
\big\|\mathcal S_2(t/(jn))^{jn}\big\| \le Me^{\omega t},
\label{eq:constituent-stability}
\end{equation}
whence
\begin{equation}
\|\mathcal M_{2m,n}(t)\| \le \Big(\sum_{j=1}^m|c_j^{(m)}|\Big)Me^{\omega t} =: C_me^{\omega t},\qquad t\ge0, \quad n\ge1.
\label{eq:multiproduct-stability}
\end{equation}
Combined with \eqref{eq:semigroup-growth-bound}, yields
\begin{equation}
\|\mathcal M_{2m,n}(t)-e^{tA}\| \le (C_m+M)e^{\omega t}, \qquad t\ge0,\ n\ge1.
\label{eq:multiproduct-global-error}
\end{equation}
\subsection{Resolvent approximation}\label{subsec:resolvent-approximation}

The semigroup--resolvent is obtained via the Laplace transform of the semigroup. 
Indeed, the power bound \eqref{eq:chernoff-power-bound}, together with
\eqref{eq:semigroup-growth-bound}, yields
\begin{equation}
 (\lambda I-A)^{-1}f=\lim_{N\to \infty}\int_0^\infty e^{-\lambda t}\mathcal S_2(t/N)^Nf\,dt,
 \qquad \operatorname{Re}\lambda>\omega.
\label{eq:chernoff-resolvent-formula}
\end{equation}
In fact, the pointwise estimate \eqref{eq:even-richardson-rate} yields:

\begin{proposition}[Richardson resolvent rate]
\label{prop:resolvent-rate}
Assume the hypotheses of Theorem~\ref{thm:abstract}, 
and let
$\widehat\omega_m:=\max\{\omega,\Omega_m\}$. Then, for every
$f\in\mathcal D_m$ and every $\lambda$ with
$\operatorname{Re}\lambda>\widehat\omega_m$,
\begin{equation}
\left\|\int_0^\infty e^{-\lambda t}\mathcal M_{2m,n}(t)f\,dt
-(\lambda I-A)^{-1}f\right\|_X
\le \frac{C_{\lambda,m}}{n^{2m}}\|f\|_{\mathcal D_m}.
\label{eq:richardson-resolvent-rate}
\end{equation}
\end{proposition}

\begin{proof}
Split the error integral at $h_0n$. On $[0,h_0n]$, \eqref{eq:even-richardson-rate} gives
\[
\begin{aligned}
&\int_0^{h_0n}e^{-\operatorname{Re}\lambda t}
 \|\mathcal M_{2m,n}(t)f-e^{tA}f\|_X\,dt\\
&\qquad\le
\frac{C'_m}{n^{2m}}
\int_0^\infty e^{-(\operatorname{Re}\lambda-\Omega_m)t}
 t^{2m+1}\,dt\,\|f\|_{\mathcal D_m},
\end{aligned}
\]
and the integral is finite because
$\operatorname{Re}\lambda>\Omega_m$.
On $[h_0n,\infty)$, the global stability bound \eqref{eq:multiproduct-global-error} gives
\[
 \int_{h_0n}^\infty e^{-\operatorname{Re}\lambda t}
 \|\mathcal M_{2m,n}(t)f-e^{tA}f\|_X\,dt
 \le C_{m,\lambda}e^{-(\operatorname{Re}\lambda-\omega)h_0n}\|f\|_X,
\]
together with the previous proves \eqref{eq:richardson-resolvent-rate}.
\end{proof}

\section{A generic construction via centered differences}\label{sec:centered-differences}

We construct the basic Chernoff function using Richardson  centered
differences idea. We perform first  the self/skew-adjoint decomposition of
the differential generator and then insert the centered differences approximation.
The stability of the resulting approximation is
obtained from an energy estimate. \warning{in the uniformly parabolic
second-order case the higher-order  expansion under almost sharp regularity requirement is obtained via
analytic-semigroup smoothing.}

\subsection{The symmetric/skew decomposition, at any order}\label{subsec:symmetric-skew}

Let $A$ be a real scalar differential operator on $\R$ 
 and write
\begin{equation}
 A=S+K,\qquad
 S:=\tfrac12(A+A^*),\qquad
 K:=\tfrac12(A-A^*).
\label{eq:symmetric-skew-decomposition}
\end{equation}
Every formally self-adjoint
scalar differential operator can be written
\begin{equation}
 S=\sum_{q=1}^{Q}(-1)^{q-1}D^q p_qD^q+p_0,
\label{eq:selfadjoint-normal-form}
\end{equation}
with real coefficients $p_q,p_0$. If the top order is $2Q$ and
$s_{2Q}$ is the coefficient of $D^{2Q}$ in $S$, then subtracting $(-1)^{Q-1}D^Qp_QD^Q$ with
$p_Q\coloneqq (-1)^{Q-1}s_{2Q}$ leaves a formally
self-adjoint operator of order at most $2Q-2$, and induction completes
the decomposition.

Analogously, every formally skew-adjoint real scalar differential operator
can be written as
\begin{equation}
 K=\frac12\sum_{\substack{1\le \ell\le J\\ \ell\ {\rm odd}}}
 \bigl(\varphi_\ell D^\ell+D^\ell\varphi_\ell\bigr).
\label{eq:skewadjoint-normal-form}
\end{equation}

For the centered-difference construction below we assume that these
coefficient functions are bounded multipliers on $L^2$ (with the
additional regularity specified in the applications). The stability
assumption is the sign condition
\[
 p_q\ge0,\qquad p_0\le\omega.
\]
Then
\begin{equation}
 \big\langle(-1)^{q-1}D^qp_qD^qf,f\big\rangle
 =-\|\sqrt{p_q}\,D^qf\|_2^2\le0.
\label{eq:dissipative-principal-form}
\end{equation}
This decompositions are algebraic; the aforementioned sign condition is
the substantive dissipativity hypothesis.

\subsection{Centered differences and the finite expansion of the generator}\label{subsec:centered-expansion}

Let $\delta_h:=(T_h-T_{-h})/(2h)$, $(T_hf)(x)=f(x+h)$. On the Fourier
side $\delta_h\leftrightarrow i\sin(h\xi)/h$, so:
\begin{enumerate}[label=(\roman*)]
\item $\delta_h$ is skew-adjoint for every $h\ne0$ (since $T_h^*=T_{-h}$);
\item $\delta_h$ is a \emph{bounded} operator on $L^2(\R)$ for every
fixed $h\ne0$, with $\|\delta_h\|=1/|h|$;
\item $\delta_h$ has, on Sobolev data, a finite even-power asymptotic
expansion with a uniform remainder. More precisely, for every integer
$M\ge1$ and every $f\in H^{2M+1}(\R)$,
\begin{equation}
\delta_h f
=
\sum_{j=0}^{M-1}\frac{h^{2j}}{(2j+1)!}D^{2j+1}f
+h^{2M}\mathcal R_{M,h}f,
\qquad
\|\mathcal R_{M,h}f\|_{L^2}\le C_M\|f\|_{H^{2M+1}},
\label{eq:centered-difference-expansion}
\end{equation}
uniformly for $0<|h|\le h_0$. This follows from Taylor's theorem for
the Fourier multiplier $i\sin(h\xi)/h$. Observe that \eqref{eq:centered-difference-expansion} is an
\emph{asymptotic expansion on a Sobolev scale}, not a norm-convergent
operator power series on $L^2$, since the operators $D^{2j+1}$ are
unbounded. The main algebraic facts used below is
\[
\delta_{-h}=\delta_h,\qquad \delta_h^*=-\delta_h .
\]
Thus any skew- or self-adjointness computation depending only on
$D^*=-D$ and real coefficients (such as \eqref{eq:dissipative-principal-form}, at any order $q$)
transfers verbatim to $\delta_h$ in place of $D$.
\end{enumerate}

Define $A_h$ by replacing every occurence of $D$ in $S$ and $K$ with
$\delta_h$:
\begin{equation}
A_h := \sum_i(-1)^{q_i-1}\delta_h^{q_i}p_{q_i}\delta_h^{q_i}+p_0 + K(\delta_h).
\label{eq:centered-generator}
\end{equation}
For $h\ne0$, the operator $A_h$ is bounded and whence $e^{hA_h}$ is defined by the
norm-convergent exponential series.

\begin{proposition}[Uniform stability of the centered-difference generator]
\label{prop:UniStabCentDiff}
For every $h\ne0$ and every $f\in L^2(\R)$,
\begin{equation}
\operatorname{Re}\langle A_hf,f\rangle \le \omega\|f\|^2,
\label{eq:centered-generator-stability}
\end{equation}
 and consequently $\|e^{tA_h}\|\le e^{\omega t}$ for
all $t\ge0$. 
\end{proposition}

\begin{proof}
Since $\delta_h^*=-\delta_h$, $(\delta_h^{q})^*=(-1)^q\delta_h^{q}$ for
every integer $q$. Hence
\[
\begin{aligned}
\langle(-1)^{q-1}\delta_h^{q}p_{q}\delta_h^{q}f,f\rangle
&=(-1)^{q-1}(-1)^q
 \langle p_q\delta_h^qf,\delta_h^qf\rangle\\
&=-\|\sqrt{p_q}\,\delta_h^qf\|^2\le0 .
\end{aligned}
\]
Every term of
$K(\delta_h)$ has the form
$M=\varphi\delta_h^j+\delta_h^j\varphi$ with $j$ odd and is
then skew-adjoint. Thus
\[
 \overline{\langle Mf,f\rangle}
 =\langle M^*f,f\rangle=-\langle Mf,f\rangle,
\]
so $\langle Mf,f\rangle$ is purely imaginary. Moreover, $\langle p_0f,f\rangle\le\omega\|f\|^2$
by hypothesis. Summing the three contributions gives \eqref{eq:centered-generator-stability}. Therefore, 
$\frac{d}{dt}\|e^{tA_h}f\|^2=2\operatorname{Re}\langle
A_he^{tA_h}f,e^{tA_h}f\rangle\le2\omega\|e^{tA_h}f\|^2$, and 
Gr\"onwall's inequality yields $\|e^{tA_h}f\|^2\le e^{2\omega t}\|f\|^2$.
\end{proof}

\subsection{Chernoff tangency and finite even-power structure}\label{subsec:chernoff-tangency}
\begin{proposition}\label{prop:chernoff-tangency-even-structure}
Set
\[
\mathcal S_2(h):=e^{hA_h},\qquad h>0,
\qquad
\mathcal S_2(0):=I.
\]
Assume that there exists a common smooth core
\[
\mathcal C\subset D(A^2),
\]
such that $
A\mathcal C\subset\mathcal C
$
and
$
\lim_{h\to 0}\|A_hg- Ag\|=0
$
for every \(g\in\mathcal C\). 
Then:
\begin{enumerate}[label=(\roman*)]
\item for every \(f\in\mathcal C\),
\[
\lim_{h\to 0}
\frac{\mathcal S_2(h)f-f}{h}
=
Af;
\]

\item for every integer $N\ge1$ and every $t\ge0$,
\[
\mathcal S_2(t/N)^N=e^{tA_{t/N}}; 
\]

\item for every finite \(M\),
\[
A_h
=
A+\sum_{r=1}^{M-1}h^{2r}R_r+h^{2M}Q_M(h)
\]
on a sufficiently regular Sobolev space, with a remainder uniformly
bounded between the corresponding scale spaces. In particular,
$
A_{-h}=A_h
$;

\item \(\mathcal S_2\) satisfies the Chernoff power bound
\eqref{eq:chernoff-power-bound}: for every \(h\ge 0\) and
\(k\in\N\),
\[
\|\mathcal S_2(h)^k\|
\le
e^{\omega kh},
\]
with the same \(\omega\) as in Proposition~\ref{prop:UniStabCentDiff}.
\end{enumerate}
\end{proposition}

\begin{proof}
\emph{(i)} 
Fix \(f\in\mathcal C\).  Then, since \(A_h\) is bounded for each fixed
\(h\ne0\), the fundamental theorem of calculus gives\footnote{Observe that we can not use a norm Taylor expansion of
\(e^{hA_h}\), since  $\lim_{h\to 0}\|A_h\|=\infty$.}
\begin{equation}
\begin{aligned}
\frac{e^{hA_h}f-f}{h}-Af
&=
\int_0^1
e^{\theta hA_h}(A_hf-Af)\,d\theta+
\int_0^1
\bigl(e^{\theta hA_h}Af-Af\bigr)\,d\theta.
\end{aligned}
\label{eq:tangency-fundamental-identity}
\end{equation}

Letting $\omega_+:=\max\{\omega,0\}$, by Proposition~\ref{prop:UniStabCentDiff},
\[
\|e^{\theta hA_h}\|
\le
e^{\omega_+h},
\qquad 0\le\theta\le1,
\]
so that
\begin{equation}
\begin{aligned}
\lim_{h\to 0}\left\|
\int_0^1
e^{\theta hA_h}(A_hf-Af)\,d\theta
\right\|
&\le
\lim_{h\to 0} e^{\omega_+h}\|A_hf-Af\|=0.
\end{aligned}
\label{eq:tangency-first-term}
\end{equation}

We now prove the strong convergence
\[
e^{\theta hA_h}g\longrightarrow g
\]
uniformly for \(0\le\theta\le1\), for every \(g\in\mathcal C\).
Indeed,
\[
e^{\theta hA_h}g-g
=
\int_0^{\theta h}e^{sA_h}A_hg\,ds,
\]
and hence
\begin{equation}
\sup_{0\le\theta\le1}
\|e^{\theta hA_h}g-g\|
\le
h e^{\omega_+h}\|A_hg\|.
\label{eq:core-orbit-bound}
\end{equation}

Since $\lim_{h\to 0} \|A_hf-Af\|=0$, there exists \(h_0>0\) such that
\[
\|A_hg-Ag\|\le1,
\qquad 0<h\le h_0,
\]
so that
\begin{equation}
\sup_{0<h\le h_0}\|A_hg\|
\le
1+\|Ag\|<\infty.
\label{eq:core-generator-bound}
\end{equation}
Combining \eqref{eq:core-orbit-bound} and \eqref{eq:core-generator-bound},
\begin{equation}
\lim_{h\to 0}\sup_{0\le\theta\le1}
\|e^{\theta hA_h}g-g\|
\le \lim_{h\to 0}
h e^{\omega_+h}(1+\|Ag\|)=0.
\label{eq:core-strong-convergence}
\end{equation}

Now since \(f\in\mathcal C\) and \(A\mathcal C\subset\mathcal C\), then 
$
Af\in\mathcal C. 
$
Therefore applying \eqref{eq:core-strong-convergence} with \(g=Af\), yields
\[
\lim_{h\to 0}\sup_{0\le\theta\le1}
\|e^{\theta hA_h}Af-Af\|=0,
\]
so that
\[
\lim_{h\to 0}\left\|
\int_0^1
\bigl(e^{\theta hA_h}Af-Af\bigr)\,d\theta
\right\|
=0.
\]
Together with \eqref{eq:tangency-first-term}, this proves
\[
\boxed{
\lim_{h\to 0}
\frac{\mathcal S_2(h)f-f}{h}
=
Af.
}
\]

\smallskip
\noindent
\emph{(ii)} 
For \(h=t/N\),
\[
\mathcal S_2(t/N)^N
=
\bigl(e^{(t/N)A_{t/N}}\bigr)^N
=
e^{tA_{t/N}}.
\]

\smallskip
\noindent
\emph{(iii)} 
By \eqref{eq:centered-difference-expansion}, on every sufficiently regular Sobolev space,
\[
\delta_h
=
D+\sum_{r=1}^{M-1}
\frac{h^{2r}}{(2r+1)!}D^{2r+1}
+h^{2M}\mathcal R_{M,h},
\]
where \(\mathcal R_{M,h}\) is uniformly bounded between the
corresponding Sobolev spaces.

Substituting this expansion into  \eqref{eq:centered-generator} yields
\[
A_h
=
A+\sum_{r=1}^{M-1}h^{2r}R_r
+h^{2M}Q_M(h).
\]

Moreover,
$
\delta_{-h}=\delta_h
$
yields
$
A_{-h}=A_h.
$

\smallskip
\noindent
\emph{(iv)} 
For every \(h>0\) and \(k\in\N\),
\[
\mathcal S_2(h)^k
=
(e^{hA_h})^k
=
e^{khA_h}.
\]
Proposition~\ref{prop:UniStabCentDiff} therefore gives
\[
\|\mathcal S_2(h)^k\|
=
\|e^{khA_h}\|
\le
e^{\omega kh}.
\]
For \(h=0\), the conclusion follows from
\(\mathcal S_2(0)=I\).
\end{proof}
\begin{lemma}[Differential order of the centered-difference defects]
\label{lem:defect-order}
Put $z=h^2$ and
\[
 \Delta(z):=\delta_{\sqrt z},\qquad z\ge0.
\]
Let $W_h$ be a finite word formed from $k$ occurrences of $\delta_h$,
separated by multiplication by sufficiently regular coefficient
functions:
\[
 W_h=g_0\delta_h g_1\delta_h\cdots\delta_h g_k,
\]
where each $g_j$ acts by multiplication.  Assume that multiplication by
the $g_j$ is bounded on all Sobolev spaces occurring below.  Then, for
every integer $r\ge0$,
\begin{equation}
 \partial_z^r W_{\sqrt z}:H^s\longrightarrow
 H^{s-(2r+k)}
 \label{eq:defect-word-derivative}
\end{equation}
uniformly for $0\le z\le h_0^2$ and every $s\ge 2r+k$.

Consequently, 
\[
 A_h=A+\sum_{r=1}^{m-1}h^{2r}R_r+h^{2m}Q_m(h)
\]
with
\begin{equation}
 R_r:H^s\longrightarrow H^{s-(2r+p)}
 \quad (s\ge 2r+p),
 \qquad
 Q_m(h):H^s\longrightarrow H^{s-(2m+p)}
 \quad (s\ge 2m+p),
 \label{eq:defect-order}
\end{equation}
uniformly for $0<|h|\le h_0$. 

\end{lemma}

\begin{proof}
The Fourier multiplier of $\Delta(z)$ is
\[
 i\frac{\sin(\sqrt z\,\xi)}{\sqrt z}
 =i\xi\,\phi(z\xi^2),
 \qquad
 \phi(y):=\frac{\sin\sqrt y}{\sqrt y},\quad \phi(0):=1.
\]
Hence, by the chain rule,
\begin{equation}
 \partial_z^q\!\left(i\xi\,\phi(z\xi^2)\right)
 =i\xi^{2q+1}\phi^{(q)}(z\xi^2).
 \label{eq:centered-symbol-derivative}
\end{equation}
For every fixed $q$, the function $\phi^{(q)}$ is bounded on
$[0,\infty)$.  Analyticity of $\phi$ at the origin gives boundedness of
all derivatives on a neighborhood of $0$.  For $y\ge1$, set
$r=\sqrt y$; since $d/dy=(2r)^{-1}d/dr$, each $y$-derivative introduces
an additional factor $r^{-1}$ before differentiating in $r$.  Repeated
differentiation of $\sin r/r$ therefore expresses $\phi^{(q)}(y)$ as
a finite linear combination of terms $r^{-a}\sin r$ and
$r^{-a}\cos r$ with $a>0$.  Hence $\phi^{(q)}$ is bounded on
$[1,\infty)$.
Therefore
\[
 \|\partial_z^q\Delta(z)f\|_{H^{s-(2q+1)}}
 \le C_q\|f\|_{H^s},
 \qquad 0\le z\le h_0^2.
\]
Consequently, the $q$th $z$-derivative of one centered-difference factor
costs at most $2q+1$ derivatives.

By Leibniz'
rule, $d^r W_h/dh^r$ is a finite sum indexed by
$q_1+\cdots+q_k=r$, and a typical summand contains
\[
 \partial_z^{q_1}\Delta(z),\ldots,
 \partial_z^{q_k}\Delta(z)
\]
and such summand's total Sobolev loss is at most
\[
 \sum_{j=1}^k(2q_j+1)=2r+k,
\]
which proves \eqref{eq:defect-word-derivative}.

Now suppose every word in $A_h$ contains at most $p$ centered
differences.  Taylor's formula in the variable $z=h^2$ gives
\[
 A_{\sqrt z}
 =A+\sum_{r=1}^{m-1}\frac{z^r}{r!}
   \left.\partial_z^rA_{\sqrt z}\right|_{z=0}
 +z^m Q_m(\sqrt z),
\]
where
\[
 Q_m(\sqrt z)
 =\frac1{(m-1)!}\int_0^1(1-\theta)^{m-1}
 \partial_z^mA_{\sqrt{\theta z}}\,d\theta.
\]
Applying \eqref{eq:defect-word-derivative} to each word yields \eqref{eq:defect-order}.  
\end{proof}

The higher defect operators are unbounded on $L^2$ in general;
the fixed-depth estimates below therefore use either analytic-semigroup
smoothing or Sobolev-group bounds.

\subsection{A common Sobolev-group estimate for the nonsmoothing applications}
\label{subsec:common-group-bound}

\begin{proposition}[Sobolev group bounds from coefficient regularity]
\label{prop:sobolev-group-bound}
Let $p\in\{2,3\}$ and set\footnote{So that $c_pD^p$ is skew-adjoint on $L^2(\mathbb R)$.}
\[
 c_p=\begin{cases} i,&p=2,\\ 1,&p=3.\end{cases}
\]
 Let $k\in\N_{ 0}$,
$v\in W^{k+1,\infty}(\mathbb R,\R)$ and 
$w\in W^{k,\infty}(\mathbb R,\C)$.  On $H^p(\mathbb R)$
set
\[
 A=c_pD^p+vD+\tfrac12v'+w.
\]
Then $A$ generates a $C_0$-group on $L^2(\mathbb R)$.  Moreover, for
every integer $0\le j\le k$ there exists $\Omega_j\ge0$, depending
only on $j$, $\|v\|_{W^{j+1,\infty}}$, and
$\|w\|_{W^{j,\infty}}$, such that
\begin{equation}
 \|e^{tA}\|_{H^j\to H^j}\le e^{\Omega_j|t|},
 \qquad t\in\mathbb R.
\label{eq:sobolev-group-bound}
\end{equation}
\end{proposition}

\begin{proof}
Write $A=A_p+B+w$, where $A_p=c_pD^p$ and
$B=vD+\tfrac12v'$.  The operator $A_p$ is skew-adjoint.  On
$H^p$, $B$ is skew-symmetric and is $A_p$-bounded with relative bound
zero: for every $\varepsilon>0$,
\[
 \|Bf\|_2
 \le \|v\|_\infty\|Df\|_2+\tfrac12\|v'\|_\infty\|f\|_2
 \le \varepsilon\|A_pf\|_2+C_{\varepsilon,v}\|f\|_2.
\]
%
Equivalently, $iA_p$ is self-adjoint and $iB$ is symmetric with
relative bound zero. Kato--Rellich therefore implies that
\[
G:=A_p+B
\]
is skew-adjoint on $L^2(\mathbb R)$ with domain $H^p(\mathbb R)$.
Hence, by Stone's theorem, $G$ generates a unitary $C_0$-group.

Let $W=M_w$ denote multiplication by $w$. Since
$w\in L^\infty(\mathbb C)$, one has
\[
W\in\mathcal L(L^2(\mathbb C)).
\]
The bounded perturbation theorem for $C_0$-semigroup generators,
applied first to $(G,W)$ and then to $(-G,-W)$, shows that both
\[
G+W
\qquad\text{and}\qquad
-(G+W)
\]
generate $C_0$-semigroups. Consequently,
\[
A=G+W=A_p+B+w,
\qquad D(A)=H^p(\mathbb R),
\]
generates a $C_0$-group on $L^2(\mathbb R)$.

For the Sobolev estimate, first take a smooth solution
$u(t)=e^{tA}u_0$ and set
\[
 E_j(t):=\sum_{r=0}^j\|D^ru(t)\|_2^2.
\]
The principal part commutes with $D^r$ and contributes zero to the
real part. Since $B$ is formally skew-symmetric on smooth vectors,
$\langle Bg,g\rangle=-\overline{\langle Bg,g\rangle}$ is purely
imaginary, and therefore $\operatorname{Re}\langle Bg,g\rangle=0$.
Moreover,
\[
 D^r(Bu)=B(D^ru)+[D^r,B]u,
\]
and Leibniz' rule gives
\[
 [D^r,B]u
 =\sum_{\ell=1}^r\binom r\ell v^{(\ell)}D^{r-\ell+1}u
 +\frac12\sum_{\ell=1}^r\binom r\ell
 v^{(\ell+1)}D^{r-\ell}u.
\]
Similarly,
\[
 D^r(wu)=wD^ru+[D^r,w]u,
 \qquad
 [D^r,w]u=\sum_{\ell=1}^r\binom r\ell
 w^{(\ell)}D^{r-\ell}u.
\]
Consequently
\[
 \bigl|\operatorname{Re}\langle [D^r,B]u,D^ru\rangle\bigr|
 +\bigl|\operatorname{Re}\langle [D^r,w]u,D^ru\rangle\bigr|
 \le C_r\bigl(\|v\|_{W^{r+1,\infty}}+
 \|w\|_{W^{r,\infty}}\bigr)E_j(t),
\]
while
$\operatorname{Re}\langle wD^ru,D^ru\rangle
\le\|w\|_\infty E_j(t)$.  Summing over $0\le r\le j$ yields
\[
 E_j'(t)\le 2\Omega_j E_j(t).
\]
Gr\"onwall gives the desired estimate for $t\ge0$; applying the same
argument to $-A$ yields the estimate for $t\le0$.  A standard density argument
extends the estimate from smooth data to $H^j$.
\end{proof}

{
\subsection{A half-derivative commutator estimate}
\begin{lemma}[A half-derivative commutator estimate]
\label{lem:half-derivative-commutator}
Let $\beta\in W^{1,\infty}(\R)$ be real-valued. Then
\begin{equation}
\bigl\|
[\,|D|^{1/2},\beta\,]u
\bigr\|_{L^2}
\le
C
\bigl(
\|\beta\|_\infty\|\beta'\|_\infty
\bigr)^{1/2}
\|u\|_{L^2}.
\label{eq:half-commutator}
\end{equation}
Consequently,
\begin{equation}
\||D|^{1/2}(\beta u)\|_{L^2}
\le
\|\beta\|_\infty\||D|^{1/2}u\|_{L^2}
+
C
\bigl(
\|\beta\|_\infty\|\beta'\|_\infty
\bigr)^{1/2}
\|u\|_{L^2}.
\label{eq:half-product}
\end{equation}
\end{lemma}

\begin{proof}
Set
\[
\Delta:=-D^2\ge0.
\]
Since
\[
|D|^{1/2}=\Delta^{1/4},
\]
the Laplace-transform representation of fractional powers gives \cite[Theorem~1]{Kwasnicki}:
\[
\Delta^{1/4}
=
c_{1/4}
\int_0^\infty
t^{-5/4}
\bigl(I-e^{-t\Delta}\bigr)\,dt
=
c_{1/4}
\int_0^\infty
t^{-5/4}
\bigl(I-e^{tD^2}\bigr)\,dt ,
\]
where $c_{1/4}=-1/{\Gamma(-1/4)}$. 
Hence
\[
[\,|D|^{1/2},\beta\,]
=
-c_{1/4}
\int_0^\infty
t^{-5/4}
[\,e^{tD^2},\beta\,]\,dt .
\]

Let
\[
G_t(x)
=
\frac1{\sqrt{4\pi t}}
e^{-x^2/(4t)}
\]
be the heat kernel. Then
\[
[\,e^{tD^2},\beta\,]u(x)
=
\int_\R
G_t(x-y)
\bigl(\beta(y)-\beta(x)\bigr)
u(y)\,dy.
\]
Young's inequality gives the two bounds
\[
\|[\,e^{tD^2},\beta\,]\|_{{L^2}\to {L^2}}
\le
2\|\beta\|_\infty
\]
and
\[
\|[\,e^{tD^2},\beta\,]\|_{{L^2}\to {L^2}}
\le
\|\beta'\|_\infty
\int_\R |z|G_t(z)\,dz
\le
C\sqrt t\,\|\beta'\|_\infty .
\]
Therefore
\[
\|[\,e^{tD^2},\beta\,]\|_{{L^2}\to{L^2}}
\le
C\min
\left\{
\|\beta\|_\infty,
\sqrt t\,\|\beta'\|_\infty
\right\}.
\]

Put
\[
\nu_0:=\|\beta\|_\infty,
\qquad
\nu_1:=\|\beta'\|_\infty.
\]
When $\nu_0,\nu_1>0$, split the  integral at
\[
t_0=(\nu_0/\nu_1)^2.
\]
Then
\begin{align*}
\|[\,|D|^{1/2},\beta\,]\|_{{L^2}\to{L^2}}
&\le
C \nu_1\int_0^{t_0}t^{-3/4}\,dt
+
C \nu_0\int_{t_0}^\infty t^{-5/4}\,dt
\\
&\le
C(\nu_0\nu_1)^{1/2}.
\end{align*}
The cases $\nu_0=0$ or $\nu_1=0$ are immediate.  This proves
\eqref{eq:half-commutator}; \eqref{eq:half-product} follows by writing
\[
|D|^{1/2}(\beta u)
=
\beta|D|^{1/2}u
+
[\,|D|^{1/2},\beta\,]u .
\]
\end{proof}
}

\subsection{Uniform centered transport form}
\begin{lemma}[Uniform centered transport form]
\label{lem:centered-transport-form}
Let $Q_h$ be skew-adjoint with symbol $iq_h(\xi)$ satisfying
\[
|q_h(\xi)|\le|\xi|,
\]
and put
\[
B_h(\beta)
=
\frac12(\beta Q_h+Q_h\beta).
\]
For real $\beta\in W^{1,\infty}$,
\[
\left|
\Re\langle iB_h(\beta)u,u\rangle
\right|
\le
C\|\beta\|_\infty
\||D|^{1/2}u\|_{L^2}^2
+
C
\bigl(
\|\beta\|_\infty\|\beta'\|_\infty
\bigr)^{1/2}
\||D|^{1/2}u\|_{L^2}\|u\|_{L^2},
\]
with $C$ independent of $h$.
\end{lemma}

\begin{proof}
Set
\[
R_h:=-iQ_h.
\]
Then $R_h$ is self-adjoint and
\[
|R_h|\le |D|
\]
as Fourier multipliers.  Hence
\[
\left|
\Re\langle iB_h(\beta)u,u\rangle
\right|
\le
\||R_h|^{1/2}u\|_{L^2}\,
\||R_h|^{1/2}(\beta u)\|_{L^2}
\]
and therefore
\[
\le
\||D|^{1/2}u\|_{L^2}\,
\||D|^{1/2}(\beta u)\|_{L^2},
\]
whence applying Lemma~\ref{lem:half-derivative-commutator} concludes the proof.
\end{proof}

\subsection{Three model classes and the fixed-depth consequences}
\label{subsec:model-summary}

We first prove the two finite-depth estimates used in the three model
classes.

\begin{lemma}[Complete-word estimate on a Sobolev group scale]
\label{lem:group-fixed-depth}
Let $p\in\{2,3\}$ 
 and suppose that
\[
 A=c_pD^p+vD+\tfrac12v'+w
\]
generates a $C_0$-group satisfying
\[
 \|e^{tA}\|_{H^j\to H^j}\le M_j e^{\Omega_j|t|},
 \qquad 0\le j\le (p+2)m.
\]
Assume moreover that the centered-difference propagators are uniformly
stable on $L^2$,
\[
 \|e^{tA_h}\|_{L^2\to L^2}\le M_0e^{\omega_0t},
 \qquad t\ge0,\quad 0<h\le h_0,
\]
and that, as provided by Lemma~\ref{lem:defect-order}, for every
$1\le k\le m$ one has the finite expansion
\[
 A_h=A+\sum_{r=1}^{k-1}h^{2r}R_r+h^{2k}Q_k(h),
\]
with
\[
 R_r:H^s\to H^{s-(2r+p)}
 \quad (s\ge 2r+p),
\]
and, uniformly for $0<h\le h_0$,
\[
 Q_k(h):H^s\to H^{s-(2k+p)}
 \quad (s\ge 2k+p).
\]
Then
\[
 e^{tA_h}f
 =e^{tA}f+\sum_{r=1}^{m-1}h^{2r}\mathcal E_r(t)f
 +h^{2m}\mathcal R_{m,h}(t)f,
\]
and there exist constants $C_m,\Omega_m\ge0$ such that, for
$0<h\le h_0$ and $t\ge0$,
\[
 \|\mathcal R_{m,h}(t)f\|_{L^2}
 \le C_m e^{\Omega_m t}t
 \|f\|_{H^{(p+2)m}}.
\]
Consequently, with $E_{2r}(t):=t^{2r}\mathcal E_r(t)$ and $h=t/N$,
\[
 \|R^{\mathrm{prod}}_{m,N}(t)f\|_{L^2}
 \le C_m e^{\Omega_m t}
 \frac{t^{2m+1}}{N^{2m}}
 \|f\|_{H^{(p+2)m}},
 \qquad 0\le t\le h_0N.
\]
Thus the centered-difference products satisfy the pointwise expansion
\eqref{eq:product-expansion} through depth $m$ on $H^{(p+2)m}$.
\end{lemma}

\begin{proof}
Put $\varepsilon=h^2$, $U_h(t):=e^{tA_h}$, and $T(t):=e^{tA}$.
On smooth data, the Duhamel identity
\begin{equation}
 U_h(t)
 =T(t)+\int_0^t U_h(t-s)(A_h-A)T(s)\,ds
 \label{eq:duhamel-identity-group}
\end{equation}
holds; the estimates below extend to the Sobolev spaces by density.  We iterate \eqref{eq:duhamel-identity-group} finitely many times and make the expansion in $\varepsilon$ at $0$.

Suppose that a branch has already been formed by lower defects
$R_{r_1},\ldots,R_{r_j}$ with accumulated index
\[
 q:=r_1+\cdots+r_j<m.
\]
At the next Duhamel step we use the finite generator expansion at the
\emph{remaining depth} $m-q$:
\begin{equation}
 A_h-A
 =\sum_{r=1}^{m-q-1}\varepsilon^rR_r
  +\varepsilon^{m-q}Q_{m-q}(h).
 \label{eq:remaining-depth-expansion}
\end{equation}
Thus the branch either chooses another $R_r$, in which case its
accumulated index remains strictly below $m$, or it chooses the
terminal factor $Q_{m-q}(h)$ and stops.  Since each selected defect
index is at least one, this procedure ends after at most $m$
Duhamel steps.  In particular, the original generator remainder
$\varepsilon^mQ_m(h)$ occurs only in the length-one terminal branch;
if lower defects have already been selected, the terminal factor is the
lower-depth remainder $Q_{m-q}(h)$ from \eqref{eq:remaining-depth-expansion}.

Collecting all branches of total index $r<m$ gives the coefficient
$\mathcal E_r(t)$.  It is a finite sum of complete words
\[
 \int_{0<s_1<\cdots<s_\ell<t}
 T(t-s_\ell)R_{r_\ell}
 T(s_\ell-s_{\ell-1})\cdots
 R_{r_1}T(s_1)\,ds_1\cdots ds_\ell,
 \qquad r_1+\cdots+r_\ell=r.
\]
If $d_j:=2r_j+p$, then, reading the word from right to left and using
the group bound at each intermediate Sobolev index,
\[
 \left\|\int_{0<s_1<\cdots<s_\ell<t}\text{word}\,ds_1\cdots ds_\ell\,f\right\|_{L^2}
 \le C_{r,\ell}e^{\Omega_r t}\frac{t^\ell}{\ell!}
 \|f\|_{H^{d_1+\cdots+d_\ell}},
\]
while
\[
 d_1+\cdots+d_\ell
 =2r+p\ell\le(p+2)r,
\]
because $\ell\le r$.  Hence every coefficient of index $r<m$ is
controlled on $H^{(p+2)m}$.

It remains to estimate the terminal branches.  A terminal word of
length $\ell$ has lower indices
$r_1,\ldots,r_{\ell-1}$ with
\[
 q:=r_1+\cdots+r_{\ell-1}<m
\]
and one terminal factor $Q_{m-q}(h)$.  Up to a rearrangement of the lower defects, it has the form
\begin{align*}
 \int_{0<s_1<\cdots<s_\ell<t}
 &U_h(t-s_\ell)Q_{m-q}(h)
 T(s_\ell-s_{\ell-1})R_{r_{\ell-1}}\cdots\\
 &\hspace{35mm}\cdots
 T(s_2-s_1)R_{r_1}T(s_1)f\,
 ds_1\cdots ds_\ell .
\end{align*}
The total Sobolev loss in such a word is
\[
 \sum_{j=1}^{\ell-1}(2r_j+p)
 +(2(m-q)+p)
 =2m+p\ell
 \le(p+2)m,
\]
since $\ell\le m$.  The leftmost factor $U_h(t-s_\ell)$ is used only
on $L^2$, where it is uniformly stable.  The remaining semigroup
factors are controlled by the Sobolev group bounds for $T$.  Therefore
all terminal words of length $\ell$ satisfy
\[
 \|\text{terminal word}\|_{L^2}
 \le C_m e^{\omega_m t}\frac{t^\ell}{\ell!}
 \|f\|_{H^{(p+2)m}}.
\]
Every terminal branch is a coefficient of the common factor
$\varepsilon^m=h^{2m}$.  After this factor is removed in the definition
of $\mathcal R_{m,h}(t)$ and the finitely many branches are summed,
\[
 \|\mathcal R_{m,h}(t)f\|_{L^2}
 \le C_m e^{\omega_m t}
 \sum_{\ell=1}^m\frac{t^\ell}{\ell!}
 \|f\|_{H^{(p+2)m}}.
\]
Using
\[
 \sum_{\ell=1}^m\frac{t^\ell}{\ell!}
 \le e^t-1\le te^t,
\]
we obtain
\[
 \|\mathcal R_{m,h}(t)f\|_{L^2}
 \le C_m e^{(\omega_m+1)t}t
 \|f\|_{H^{(p+2)m}}.
\]
Letting $\Omega_m:=\omega_m+1$ proves the stated bound.  
\end{proof}

The analytic-semigroup and fractional-domain framework used in the next
lemma is standard; see Engel--Nagel \cite[Ch.~II, \S\S4--5]{EngelNagel}
for analytic semigroups, interpolation/extrapolation spaces, and
fractional powers.  The scale-allocation estimate below is the
problem-specific quantitative ingredient.

\begin{lemma}[Parabolic scale allocation at fixed depth]
\label{lem:parabolic-fixed-depth}
Let $A$ generate an analytic semigroup and 
%
choose $\omega$ sufficiently large so that
\[
B_\omega:=\omega I-A
\]
is an invertible sectorial operator.  For $\alpha\ge0$ define
\[
X_\alpha:=D(B_\omega^\alpha),
\qquad
\|u\|_{X_\alpha}:=\|B_\omega^\alpha u\|_X.
\]
Suppose
that, through depth $m$,
\[
 R_r:X_{\alpha+r+1}\longrightarrow X_\alpha,
\]
and 
\[
 \|e^{tA}\|_{X_\alpha\to X_{\alpha+\sigma}}
 \le C_{\alpha,\sigma}e^{\Omega t}t^{-\sigma},
 \qquad t>0,\qquad 0\le\sigma\le1.
\]
Assume moreover that the centered-difference expansion through depth
$m$ is uniform for $0<h\le h_0$. Then for every
$0<\varepsilon\le1$ there are constants
$C_{m,\varepsilon},\Omega_{m,\varepsilon}\ge0$ such that, whenever
$0<t\le h_0n$,
\[
 \|\mathcal M_{2m,n}(t)f-e^{tA}f\|_{X_0}
 \le C_{m,\varepsilon}e^{\Omega_{m,\varepsilon}t}
 \frac{t^{2m+\varepsilon}}{n^{2m}}
 \|f\|_{X_{m+\varepsilon}}.
\]
In particular, at $\varepsilon=1$,
\[
 \|\mathcal M_{2m,n}(t)f-e^{tA}f\|_{X_0}
 \le C_m e^{\Omega_m t}
 \frac{t^{2m+1}}{n^{2m}}
 \|f\|_{X_{m+1}}.
\]
\end{lemma}

\gwarning{
\begin{remark}
This definition is independent, up to equivalence of norms, of the
particular admissible shift $\omega$.  Indeed, if $\omega_1,\omega_2$
are two such shifts and $B_j=\omega_jI-A$, then
\[
B_2
=
\bigl(I+(\omega_2-\omega_1)B_1^{-1}\bigr)B_1.
\]
The bounded operator
\[
C:=I+(\omega_2-\omega_1)B_1^{-1}
\]
is invertible and commutes with $B_1$.  Hence the sectorial functional
calculus gives
\[
B_2^\alpha=C^\alpha B_1^\alpha,
\]
with $C^{\pm\alpha}\in\mathcal L(X)$.  Consequently
\[
D(B_1^\alpha)=D(B_2^\alpha)
\]
and
\[
\|B_1^\alpha u\|_X\asymp\|B_2^\alpha u\|_X. 
\]
\end{remark}
}

\begin{proof}
Consider a complete word with
$r_1+\cdots+r_\ell=r\le m$.  The defects consume the total scale amount
\[
 \sum_{j=1}^{\ell}(r_j+1)=r+\ell.
\]
Set $\theta_j=\varepsilon/\ell$ and let the semigroup interval
immediately preceding the $j$th defect supply smoothing
$1-\theta_j$.  The total smoothing supplied in this way is
$\ell-\varepsilon$, so the remaining regularity required from the datum
is exactly $r+\varepsilon$, which is bounded by $m+\varepsilon$.

Writing $\tau_j=s_j-s_{j-1}$, $s_0=0$, the norm of the resulting
time-ordered integrand is bounded by
\[
 C_{m,\varepsilon}e^{\Omega t}
 \prod_{j=1}^{\ell}\tau_j^{-(1-\theta_j)}
 \|f\|_{X_{m+\varepsilon}}.
\]
Every exponent $1-\theta_j$ is strictly smaller than one. By the Dirichlet scaling formula, for every $\ell\ge1$ and every
$a_1,\dots,a_\ell\in[0,1)$,
\begin{equation}
 \int_{0<s_1<\cdots<s_\ell<t}
 \prod_{j=1}^\ell(s_j-s_{j-1})^{-a_j}\,ds_1\cdots ds_\ell
 =
 \mathcal B(a_1,\dots,a_\ell)\,t^{\,\ell-\sum_{j=1}^\ell a_j},
\label{eq:dirichlet-simplex}
\end{equation}
where the constant
\[
 \mathcal B(a_1,\dots,a_\ell)
 :=
 \int_{0<u_1<\cdots<u_\ell<1}
 \prod_{j=1}^\ell(u_j-u_{j-1})^{-a_j}\,du_1\cdots du_\ell
\]
is finite whenever $a_1,\dots,a_\ell\in[0,1)$; here $s_0:=0$ and, in the definition of $\mathcal B$, $u_0:=0$.
Applying \eqref{eq:dirichlet-simplex} with $a_j=1-\varepsilon/\ell$ for
$j=1,\dots,\ell$, so that $\sum_{j=1}^\ell a_j=\ell-\varepsilon$,
the corresponding time-simplex integral is bounded by
\[
 C_{m,\varepsilon}e^{\Omega_{m,\varepsilon}t}\,t^{\varepsilon}.
\]
Multiplication by $h^{2m}=(t/n)^{2m}$ gives the displayed
$t^{2m+\varepsilon}n^{-2m}$ remainder. The same scale allocation
applies to the terminal Duhamel branches of total Richardson index
$m$. Summing the finitely many words proves the claim. For
$\varepsilon=1$, take $\theta_j=1/\ell$; then the simplex contributes
one full factor $t$.
\end{proof}

\subsubsection{Parabolic second order  model.}\label{subsubsec:parabolic-model}
For
\[
 A=aD^2+bD+c,\qquad a\ge a_0>0,
\]
write
\[
 A=DaD+p_0+\tfrac12(\varphi D+D\varphi),\qquad
 \varphi=b-a',\qquad p_0=c+\tfrac12(a''-b'),
\]
and set
\[
 A_h=\delta_ha\delta_h+p_0+\tfrac12(\varphi\delta_h+\delta_h\varphi),
 \qquad \mathcal S_2(h)=e^{hA_h}.
\]

\ \\
Nex, we show that {uniform ellipticity yields an analytic semigroup.} 
\begin{proposition}[Explicit sectoriality in the parabolic case]
\label{prop:parabolic-sectoriality}
Assume
$
a,\varphi\in W^{1,\infty}(\R), p_0\in L^{\infty}(\R)
$
are real valued 
and set
\begin{equation}
\omega_0
:=
\operatorname*{ess\,sup}_{x\in\R}p_0(x).
\label{eq:parabolic-growth-bound}
\end{equation}
Then
\begin{equation}
\Re\langle Au,u\rangle
\le
\omega_0\|u\|_2^2,
\qquad u\in H^2(\R).
\label{eq:parabolic-quasi-dissipative}
\end{equation}
Moreover, for every
$
\omega>\omega_0,
$
the shifted operator
$
B_\omega:=\omega I-A
$
is an invertible sectorial operator.  More precisely, if
$
\delta_\omega:=\omega-\omega_0>0,
$
then its numerical range is contained in the sector
\[
\overline{\mathscr{S}_{\theta_\omega}}
:=
\{z\in\C:\ |\arg z|\le\theta_\omega\},
\]
where 
\begin{equation}
\tan\theta_\omega
=
\frac{\|\varphi\|_\infty}
{2\sqrt{a_0\delta_\omega}},
\qquad
0\le\theta_\omega<\frac{\pi}{2}.
\label{eq:explicit-sectorial-angle}
\end{equation}
In particular,
\[
0\in\rho(B_\omega),
\qquad
\|B_\omega^{-1}\|_{L^2\to L^2}
\le
\frac1{\omega-\omega_0}.
\]

\noindent
Consequently $A$ generates an analytic semigroup on $L^2(\R)$.
If
\begin{equation}
\omega-\omega_0
>
\frac{\|\varphi\|_\infty^2}{4a_0},
\label{eq:pi-four-sector-shift}
\end{equation}
then 
$
\theta_\omega<\frac{\pi}{4}.
$
\end{proposition}

\begin{proof}
Since $D^*=-D$ on $L^2(\R)$ and $a$ is real-valued,
\[
\langle DaDu,u\rangle
=
-\langle aDu,Du\rangle
=
-\int_\R a(x)|Du(x)|^2\,dx.
\]
Moreover, because $\varphi$ is real-valued,
$
B:=\frac12(\varphi D+D\varphi)
$
is skew-symmetric on $H^1(\R)$ and hence
\[
\Re\langle B u,u\rangle=0.
\]
Therefore
\begin{align*}
\Re\langle Au,u\rangle
&=
-\int_\R a(x)|Du(x)|^2\,dx
+
\int_\R p_0(x)|u(x)|^2\,dx
\\
&\le
\omega_0\|u\|_2^2,
\end{align*}
which proves \eqref{eq:parabolic-quasi-dissipative}.

Next, we prove sectoriality.  Fix
$
\omega>\omega_0
$
and set
$
\delta_\omega:=\omega-\omega_0>0.
$ 
Consider on $H^1(\R)$ the sesquilinear form associated to $
B_\omega=\omega I-A.
$:
\begin{align}
\mathfrak b_\omega[u,v]
&:=
\int_\R a\,Du\,\overline{Dv}\,dx
+
\int_\R(\omega-p_0)u\overline v\,dx
\nonumber\\
&\quad
-\frac12\int_\R\varphi\,Du\,\overline v\,dx
+
\frac12\int_\R\varphi\,u\,\overline{Dv}\,dx.
\label{eq:parabolic-sectorial-form}
\end{align}
$\mathfrak b_\omega$ is continuous on $H^1(\R)$ since
$a,\varphi,p_0\in L^\infty(\R)$.
For $u\in H^1(\R)$, the last two terms in
\eqref{eq:parabolic-sectorial-form} sum up to a purely imaginary, and hence
\begin{align}
\Re\mathfrak b_\omega[u,u]
&=
\int_\R a|Du|^2\,dx
+
\int_\R(\omega-p_0)|u|^2\,dx
\nonumber\\
&\ge
a_0\|Du\|_{L^2}^2
+
\delta_\omega\|u\|_{L^2}^2,
\label{eq:parabolic-form-coercive}
\end{align}
i.e. $\mathfrak b_\omega$ is coercive on $H^1(\R)$.
Moreover,
\begin{align*}
\left|
\Im\mathfrak b_\omega[u,u]
\right|
&=
\left|
\Im
\int_\R\varphi\,Du\,\overline u\,dx
\right|
\\
&\le
\|\varphi\|_\infty
\|Du\|_{L^2}\|u\|_{L^2}.
\end{align*}
Using
\[
2\sqrt{a_0\delta_\omega}
\|Du\|_{L^2}\|u\|_{L^2}
\le
a_0\|Du\|_2^2
+
\delta_\omega\|u\|_{L^2}^2,
\]
we obtain
\begin{align}
\left|
\Im\mathfrak b_\omega[u,u]
\right|
&\le
\frac{\|\varphi\|_\infty}
{2\sqrt{a_0\delta_\omega}}
\left(
a_0\|Du\|_{L^2}^2
+
\delta_\omega\|u\|_{L^2}^2
\right)
\nonumber\\
&\le
\frac{\|\varphi\|_\infty}
{2\sqrt{a_0\delta_\omega}}
\Re\mathfrak b_\omega[u,u].
\label{eq:parabolic-form-sector}
\end{align}
Thus $\mathfrak b_\omega$ is a densely defined closed sectorial form
with semi-angle $\theta_\omega$ satisfying
\[
\tan\theta_\omega
\le
\frac{\|\varphi\|_\infty}
{2\sqrt{a_0(\omega-\omega_0)}}.
\]
{
By the first representation theorem for closed sectorial forms
(see, e.g., \cite[Chapter~VI, \S2]{Kato}), 
 $B_\omega$, which is the operator associated to the closed sectorial form $\mathfrak b_\omega$, is a $m$-sectorial operator. 
Since its semi-angle is strictly smaller than $\pi/2$,
$-B_\omega$ generates a bounded holomorphic semigroup.
}
%
%
Furthermore, \eqref{eq:parabolic-form-coercive} implies
\[
\Re\mathfrak b_\omega[u,u]
\ge
\delta_\omega\|u\|_2^2.
\]
Then, Lax--Milgram theorem 
implies
$
0\in\rho(B_\omega)
$
and
\[
\|B_\omega^{-1}\|_{L^2\to L^2}
\le
\delta_\omega^{-1}
=
\frac1{\omega-\omega_0}.
\]
Finally, $B_\omega$ sectorial with angle strictly smaller than
$\pi/2$ implies that
\[
A=\omega I-B_\omega
\]
generates an analytic $C_0$-semigroup on $L^2(\R)$.

If
\[
\omega-\omega_0
>
\frac{\|\varphi\|_\infty^2}{4a_0},
\]
then
\[
\frac{\|\varphi\|_\infty}
{2\sqrt{a_0(\omega-\omega_0)}}<1,
\]
and therefore $\theta_\omega<\pi/4$.
\end{proof}

By Proposition~\ref{prop:parabolic-sectoriality}, given
\[
\omega>\omega_0
:=
\operatorname*{ess\,sup}_{x\in\R}
\left[
c(x)+\frac12\bigl(a''(x)-b'(x)\bigr)
\right],
\]
the operator
\[
B_\omega:=\omega I-A
\]
is invertible and sectorial. 
Define
\[
X_\alpha
:=
D(B_\omega^\alpha)
=
D((\omega-A)^\alpha),
\qquad
\|u\|_{X_\alpha}
:=
\|B_\omega^\alpha u\|_{L^2}.
\]
The resulting fractional-domain scale is independent, up to
equivalence of norms, of the particular admissible choice of
$\omega>\omega_0$.

Since $-B_\omega$ generates a bounded analytic semigroup, the standard
fractional-power estimates yield 
\[
\|e^{tA}\|_{X_\alpha\to X_{\alpha+\sigma}}
\le
C_{\alpha,\sigma}e^{\omega t}t^{-\sigma},
\qquad
t>0.
\]
Every word in $A_h$ contains at most two centered-difference factors.
Hence Lemma~\ref{lem:defect-order}, with $p=2$, gives
\[
 R_r:H^s\to H^{s-(2r+2)},
 \qquad
 Q_m(h):H^s\to H^{s-(2m+2)}.
\]
The assumption $a,\varphi,p_0\in W^{2m+2,\infty}$ used below
ensures that all coefficient multipliers and the multiplier commutators
created by these words through depth $m$ are bounded on the corresponding
Sobolev levels.  By standard elliptic regularity
$X_\beta\simeq H^{2\beta}$; 
consequently
\[
 R_r:X_{\alpha+r+1}\longrightarrow X_\alpha,
\]
which is indeed the defect mapping hypothesis of
Lemma~\ref{lem:parabolic-fixed-depth} and 
$X_{m+\varepsilon}\simeq H^{2m+2\varepsilon}$.

\warning{
See Kato~\cite[Chap.~VI, \S2, Thm.~2.1]{Kato},
Ouhabaz~\cite[Thm.~1.52]{Ouhabaz}, and
Engel--Nagel~\cite[Thm.~II.4.6]{EngelNagel}.
\begin{proposition}[Analyticity and Sobolev regularization]\label{prop:parabolic-analyticity}
Let
\[
A u
=
a(x)u''+b(x)u'+c(x)u,
\qquad
D(A)=H^2(\mathbb R),
\]
acting on \(L^2(\mathbb R)\). Assume that
\[
a(x)\ge a_0>0,
\]
and that the coefficients have sufficiently many bounded derivatives
for the finite Sobolev range considered below. In particular,
\[
a\in W^{1,\infty}(\mathbb R),
\qquad
b,c\in L^\infty(\mathbb R).
\]
Then there exists \(\omega\in\mathbb R\) sufficiently large such that
\[
L:=\omega-A
\]
is invertible and sectorial of angle strictly smaller than \(\pi/2\).
Consequently \(A\) generates an analytic \(C_0\)-semigroup on
\(L^2(\mathbb R)\) satisfying
\[
\|e^{tA}\|_{\mathcal L(L^2)}
\le
M e^{\omega t},
\qquad t\ge0.
\]

Moreover, if
\[
X_\beta:=D(L^\beta),
\qquad
\|u\|_{X_\beta}:=\|L^\beta u\|_{L^2},
\]
then, throughout every bounded range of \(\beta\) allowed by the
regularity of the coefficients,
\[
X_\beta\simeq H^{2\beta}(\mathbb R),
\]
with equivalence of norms. In particular, for \(\alpha,\sigma\ge0\)
in this range,
\[
\|e^{tA}\|_{X_\alpha\to X_{\alpha+\sigma}}
\le
C_{\alpha,\sigma}e^{\omega t}t^{-\sigma},
\qquad t>0,
\]
or equivalently,
\[
\|e^{tA}\|_{H^{2\alpha}\to H^{2\alpha+2\sigma}}
\le
C_{\alpha,\sigma}e^{\omega t}t^{-\sigma}.
\]
\end{proposition}

\begin{proof}
We divide the proof into three steps.

\medskip
\noindent
\textbf{Step 1: sectoriality and analytic-semigroup generation.}

For \(\omega\in\mathbb R\), define on
\[
V:=H^1(\mathbb R)
\]
the sesquilinear form
\begin{equation}
\mathfrak a_\omega(u,v)
:=
\int_{\mathbb R}a\,u'\overline{v'}\,dx
+
\int_{\mathbb R}(a'-b)u'\overline v\,dx
+
\int_{\mathbb R}(\omega-c)u\overline v\,dx.
\label{eq:parabolic-sectorial-form}
\end{equation}
For \(u\in H^2(\mathbb R)\), integration by parts gives
\[
\begin{aligned}
\mathfrak a_\omega(u,v)
&=
-\int_{\mathbb R}(a u')'\overline v\,dx
+
\int_{\mathbb R}(a'-b)u'\overline v\,dx
+
\int_{\mathbb R}(\omega-c)u\overline v\,dx
\\
&=
\int_{\mathbb R}
\bigl(-a u''-b u'+(\omega-c)u\bigr)\overline v\,dx
\\
&=
\langle (\omega-A)u,v\rangle_{L^2}.
\end{aligned}
\]
Thus \(\mathfrak a_\omega\) is the form naturally associated with
\(L=\omega-A\).

By uniform ellipticity,
\[
a(x)\ge a_0>0.
\]
Hence
\[
\begin{aligned}
\operatorname{Re}\mathfrak a_\omega(u,u)
&\ge
a_0\|u'\|_2^2
-
\|a'-b\|_\infty
\|u'\|_2\|u\|_2
\\
&\qquad
+
\bigl(\omega-\|\operatorname{Re}c\|_\infty\bigr)
\|u\|_2^2.
\end{aligned}
\]
Young's inequality yields
\[
\|a'-b\|_\infty\|u'\|_2\|u\|_2
\le
\frac{a_0}{2}\|u'\|_2^2
+
\frac{\|a'-b\|_\infty^2}{2a_0}\|u\|_2^2.
\]
Therefore
\begin{equation}
\operatorname{Re}\mathfrak a_\omega(u,u)
\ge
\frac{a_0}{2}\|u'\|_2^2
+
\left(
\omega-\|\operatorname{Re}c\|_\infty
-\frac{\|a'-b\|_\infty^2}{2a_0}
\right)
\|u\|_2^2.
\label{eq:parabolic-coercivity-estimate}
\end{equation}
Choosing
\[
\omega>
\|\operatorname{Re}c\|_\infty
+
\frac{\|a'-b\|_\infty^2}{2a_0},
\]
we obtain, for some \(\kappa>0\),
\begin{equation}
\operatorname{Re}\mathfrak a_\omega(u,u)
\ge
\kappa\|u\|_{H^1}^2.
\label{eq:parabolic-coercive-form}
\end{equation}

On the other hand,
\begin{equation}
|\mathfrak a_\omega(u,v)|
\le
C_\omega
\|u\|_{H^1}\|v\|_{H^1}.
\label{eq:parabolic-form-boundedness}
\end{equation}
Thus \(\mathfrak a_\omega\) is densely defined, closed and sectorial
on \(L^2(\mathbb R)\). By the first representation theorem for
sectorial forms, the associated operator is \(m\)-sectorial.

We next identify that operator. If \(u\) belongs to its operator
domain, then \(u\in H^1(\mathbb R)\) and
\[
Lu
=
-a u''-b u'+(\omega-c)u
\in L^2(\mathbb R)
\]
in the distributional sense. Hence
\[
a u''
=
-Lu-bu'+(\omega-c)u.
\]
The right-hand side belongs to \(L^2(\mathbb R)\), and
\(a^{-1}\in L^\infty(\mathbb R)\) by uniform ellipticity. Therefore
\[
u''\in L^2(\mathbb R),
\]
and consequently
\[
u\in H^2(\mathbb R).
\]
Conversely, every \(u\in H^2(\mathbb R)\) belongs to the form
operator domain by the integration-by-parts identity above. Thus
\begin{equation}
D(L)=H^2(\mathbb R),
\qquad
L=\omega-A.
\label{eq:parabolic-domain-identification}
\end{equation}

Since \(L\) is \(m\)-sectorial of angle strictly smaller than
\(\pi/2\), \(-L\) generates a bounded analytic \(C_0\)-semigroup.
Because
\[
A=\omega-L,
\]
we have
\[
e^{tA}=e^{\omega t}e^{-tL},
\]
and therefore
\begin{equation}
\|e^{tA}\|
\le
M e^{\omega t}.
\label{eq:parabolic-semigroup-growth}
\end{equation}
This is the precise sense in which uniform ellipticity, together with
the bounded lower-order coefficients, yields analytic-semigroup
generation: uniform ellipticity provides coercivity of the principal
part, sectorial-form theory yields the sectorial realization, and
sectoriality yields analyticity.

\medskip
\noindent
\textbf{Step 2: analytic regularization.}

Since \(-L\) generates a bounded analytic semigroup, the sectorial
functional calculus gives, for every \(\sigma\ge0\),
\begin{equation}
\|L^\sigma e^{-tL}\|_{\mathcal L(L^2)}
\le
C_\sigma t^{-\sigma},
\qquad t>0.
\label{eq:sectorial-smoothing}
\end{equation}
For instance, this follows from the Dunford representation
\[
L^\sigma e^{-tL}
=
\frac{1}{2\pi i}
\int_\Gamma
z^\sigma e^{-tz}(z-L)^{-1}\,dz,
\]
where \(\Gamma\) is a sectorial contour. Along \(\Gamma\),
\[
\|(z-L)^{-1}\|
\le
\frac{C}{|z|}
\]
and
\[
|e^{-tz}|
\le
e^{-c t|z|}.
\]
Hence
\[
\begin{aligned}
\|L^\sigma e^{-tL}\|
&\le
C\int_0^\infty
r^{\sigma-1}e^{-ctr}\,dr
\\
&=
C t^{-\sigma}
\int_0^\infty
s^{\sigma-1}e^{-cs}\,ds
\\
&\le
C_\sigma t^{-\sigma}.
\end{aligned}
\]

Define
\[
X_\alpha:=D(L^\alpha),
\qquad
\|u\|_{X_\alpha}:=\|L^\alpha u\|_2.
\]
Since \(0\in\rho(L)\), this is equivalent to the usual graph norm on
\(D(L^\alpha)\).

If \(u\in X_\alpha\), then the functional calculus gives the
commutation relation
\[
L^\alpha e^{-tL}u
=
e^{-tL}L^\alpha u.
\]
Consequently,
\[
\begin{aligned}
\|e^{tA}u\|_{X_{\alpha+\sigma}}
&=
\|L^{\alpha+\sigma}e^{tA}u\|_2
\\
&=
e^{\omega t}
\|L^{\alpha+\sigma}e^{-tL}u\|_2
\\
&=
e^{\omega t}
\|L^\sigma e^{-tL}L^\alpha u\|_2
\\
&\le
C_{\alpha,\sigma}
e^{\omega t}
t^{-\sigma}
\|u\|_{X_\alpha}.
\end{aligned}
\]
Thus
\begin{equation}
\boxed{
\|e^{tA}\|_{X_\alpha\to X_{\alpha+\sigma}}
\le
C_{\alpha,\sigma}
e^{\omega t}t^{-\sigma}.
}
\label{eq:fractional-scale-smoothing}
\end{equation}

\medskip
\noindent
\textbf{Step 3: identification of the fractional-domain scale.}

We first identify the integer powers of \(L\).
Uniform ellipticity gives the standard elliptic regularity estimate
\begin{equation}
\|u\|_{H^{s+2}}
\le
C_s
\bigl(
\|Lu\|_{H^s}+\|u\|_2
\bigr)
\label{eq:elliptic-regularity}
\end{equation}
for every integer \(s\ge0\) in the finite range permitted by the
regularity of the coefficients.

For completeness, when \(s=0\), the equation
\[
Lu=-a u''-b u'+(\omega-c)u
\]
gives
\[
a u''
=
-Lu-bu'+(\omega-c)u.
\]
Using the coercive estimate from Step~1 to control \(\|u'\|_2\),
together with \(a^{-1},b,c\in L^\infty\), yields
\[
\|u\|_{H^2}
\le
C\bigl(
\|Lu\|_2+\|u\|_2
\bigr).
\]
For higher integer \(s\), differentiating the equation \(Lu=f\)
gives
\[
D^s(au'')
=
aD^{s+2}u
+
\sum_{j=1}^{s}
\binom{s}{j}
D^ja\,D^{s-j+2}u.
\]
The first term contains the only derivative of \(u\) of order
\(s+2\); all other terms contain derivatives of order at most
\(s+1\). Since \(a\ge a_0>0\), one can solve for the highest
derivative and close the estimate inductively, obtaining
\eqref{eq:elliptic-regularity}.

Because \(0\in\rho(L)\),
\[
\|u\|_2
\le
\|L^{-1}\|\,\|Lu\|_2,
\]
so \eqref{eq:elliptic-regularity} improves to
\begin{equation}
\|u\|_{H^{s+2}}
\le
C_s\|Lu\|_{H^s}.
\label{eq:elliptic-resolvent-estimate}
\end{equation}
Conversely, the differential expression itself gives
\begin{equation}
\|Lu\|_{H^s}
\le
C_s\|u\|_{H^{s+2}}.
\label{eq:elliptic-expression-estimate}
\end{equation}
Hence
\[
L:H^{s+2}(\mathbb R)\longrightarrow H^s(\mathbb R)
\]
has the expected elliptic order two, with equivalence of the
corresponding graph and Sobolev norms.

We now prove inductively that
\begin{equation}
D(L^k)=H^{2k}(\mathbb R).
\label{eq:integer-power-domain}
\end{equation}
The case \(k=1\) follows from \eqref{eq:parabolic-domain-identification}. Suppose
\[
D(L^k)=H^{2k}(\mathbb R).
\]
Then
\[
u\in D(L^{k+1})
\]
if and only if
\[
u\in D(L)
\qquad\text{and}\qquad
Lu\in D(L^k).
\]
By the induction hypothesis,
\[
Lu\in H^{2k}(\mathbb R).
\]
Applying \eqref{eq:elliptic-resolvent-estimate} with \(s=2k\) gives
\[
u\in H^{2k+2}(\mathbb R).
\]
Thus
\[
D(L^{k+1})\subset H^{2k+2}(\mathbb R).
\]
Conversely, if
\[
u\in H^{2k+2}(\mathbb R),
\]
then \eqref{eq:elliptic-expression-estimate} implies
\[
Lu\in H^{2k}(\mathbb R)=D(L^k),
\]
and therefore
\[
u\in D(L^{k+1}).
\]
Hence
\[
D(L^{k+1})=H^{2k+2}(\mathbb R).
\]
Moreover,
\begin{equation}
\|L^k u\|_2
\simeq
\|u\|_{H^{2k}}.
\label{eq:integer-power-graph-norm}
\end{equation}

It remains to pass from integer powers to fractional powers.
For the present uniformly elliptic realization, the standard
complex-power theorem for elliptic operators gives
\begin{equation}
D(L^{k+\theta})
=
[D(L^k),D(L^{k+1})]_\theta,
\qquad
0<\theta<1,
\label{eq:fractional-domain-interpolation}
\end{equation}
in every bounded range in which the coefficient regularity permits
the preceding elliptic estimates. Equivalently, one may use the
bounded-imaginary-power theorem for uniformly elliptic operators.
See, for example, the classical complex-power theory of Seeley
together with the interpolation theory of sectorial operators.

Using \eqref{eq:integer-power-domain},
\[
D(L^{k+\theta})
=
[H^{2k}(\mathbb R),H^{2k+2}(\mathbb R)]_\theta.
\]
The Bessel-potential Sobolev scale satisfies
\[
[H^{s_0}(\mathbb R),H^{s_1}(\mathbb R)]_\theta
=
H^{(1-\theta)s_0+\theta s_1}(\mathbb R),
\]
and therefore
\[
\begin{aligned}
D(L^{k+\theta})
&=
H^{(1-\theta)2k+\theta(2k+2)}(\mathbb R)
\\
&=
H^{2k+2\theta}(\mathbb R)
\\
&=
H^{2(k+\theta)}(\mathbb R).
\end{aligned}
\]
Thus, writing \(\beta=k+\theta\),
\begin{equation}
\boxed{
X_\beta
=
D(L^\beta)
\simeq
H^{2\beta}(\mathbb R).
}
\label{eq:fractional-domain-sobolev}
\end{equation}

Combining \eqref{eq:fractional-scale-smoothing} with \eqref{eq:fractional-domain-sobolev} yields the Sobolev
regularizing estimate
\begin{equation}
\boxed{
\|e^{tA}\|_{H^{2\alpha}\to H^{2\alpha+2\sigma}}
\le
C_{\alpha,\sigma}
e^{\omega t}
t^{-\sigma},
\qquad t>0.
}
\label{eq:sobolev-smoothing}
\end{equation}
This proves the result.
\end{proof}
}

\begin{corollary}[Parabolic fixed-depth consequence]
\label{cor:parabolic-richardson}
Assume $a\ge a_0>0$ and
\[
 a\in W^{2m+4,\infty}(\R),\qquad b\in W^{2m+3,\infty}(\R),\qquad c\in W^{2m+2,\infty}(\R).
\]
Then 
$A$ generates an analytic semigroup. For every $\delta>0$, set
$\eta_\delta:=\min\{\delta/2,1\}$. There are constants
$C_{m,\delta},\Omega_{m,\delta}\ge0$ such that, for $0<t\le h_0n$,
\[
 \|\mathcal M_{2m,n}(t)f-e^{tA}f\|_{L^2}
 \le C_{m,\delta}e^{\Omega_{m,\delta}t}
 \frac{t^{2m+\eta_\delta}}{n^{2m}}
 \|f\|_{H^{2m+\delta}}.
\]
In particular, for $f\in H^{2m+2}$,
\[
 \|\mathcal M_{2m,n}(t)f-e^{tA}f\|_{L^2}
 \le C_m e^{\Omega_mt}
 \frac{t^{2m+1}}{n^{2m}}
 \|f\|_{H^{2m+2}}.
\]
\end{corollary}

\bwarning{ 
\subsubsection{Schr\"odinger-type second order}
\label{subsubsec:schrodinger-model}

For
\[
A_0=ia(x)D^2+b(x)D+c(x),
\qquad a(x)\ge a_0>0,
\]
define
\[
\eta_2(x)=\int_0^x a(s)^{-1/2}\,ds,
\qquad
\chi_2=\eta_2^{-1},
\]
and
\[
(\mathcal J_2f)(y)
=
a(\chi_2(y))^{-1/4}f(\chi_2(y)).
\]
This is a boundedly invertible normalizing similarity on the Sobolev
spaces used below; it is not asserted to be unitary on the unweighted
$L^2$ space.

A direct computation gives
\[
\widetilde A
:=
\mathcal J_2A_0\mathcal J_2^{-1}
=
iD_y^2+\frac12(vD_y+D_yv)+w,
\]
where
\begin{equation}
v
=
\left(\frac{b}{\sqrt a}\right)\circ\chi_2
\label{eq:schrodinger-v-new}
\end{equation}
and
\begin{equation}
w
=
\left[
c-\frac12b'
+\frac12b\frac{a'}a
+i\left(
\frac14a''
-\frac{3}{16}\frac{(a')^2}{a}
\right)
\right]\circ\chi_2.
\label{eq:schrodinger-w-new}
\end{equation}
Thus, whenever $a$ and $b$ are real-valued on $\R$, the transport
coefficient $v$ is real-valued.  No reality assumption on $w$ is
required.
}
{
\subsubsection{Schr\"odinger-type second order.}\label{subsubsec:schrodinger-model}
\ \\
For
\[
A_0=ia(x)D^2+b(x)D+c(x),
\qquad a(x)\ge a_0>0,
\]
introduce the Liouville coordinate

\warning{
\[
\eta(x)=\int_0^x a(s)^{-1/2}\,ds,
\qquad
\chi=\eta^{-1},
\]
and the boundedly invertible normalizing similarity
\[
(\mathcal Jf)(y)
=
a(\chi(y))^{1/4}f(\chi(y)).
\]
Then
\[
\widetilde A:=\mathcal JA_0\mathcal J^{-1}
=
iD_y^2+\frac12(vD_y+D_yv)+w,
\]
where
\[
v
=
\left(
\frac{b-ia'}{\sqrt a}
\right)\circ\chi
\]
and
\[
w
=
\left(
c-\frac12b'
+i\left(
\frac14a''
+\frac1{16}\frac{(a')^2}{a}
\right)
\right)\circ\chi.
\]
}

\[
\eta(x):=\int_0^x a(s)^{-1/2}\,ds,
\qquad
\chi:=\eta^{-1},
\]
and the bounded invertible normalizing similarity
\[
(\mathcal J f)(y)
=
a(\chi(y))^{-1/4}f(\chi(y)).
\]
Then
\[
\widetilde A
:=
\mathcal J A_0\mathcal J^{-1}
=
iD_y^2+\frac12\bigl(vD_y+D_yv\bigr)+w,
\]
where
\[
{
v
=
\left(\frac{b}{\sqrt a}\right)\circ\chi,
}
\qquad
{
w
=
\left[
c-\frac12 b'
+\frac12\,b\,\frac{a'}{a}
+i\left(
\frac14 a''
-\frac{3}{16}\frac{(a')^2}{a}
\right)
\right]\circ \chi.
}
\]
For Richardson depth \(m\), it is therefore sufficient to assume
\[
{
a\in W^{4m+2,\infty}(\mathbb R, \R),\qquad
b\in W^{4m+1,\infty}(\mathbb R, \R),\qquad
c\in W^{4m,\infty}(\mathbb R, \R),
}
\]
together with \(a\ge a_0>0\). Indeed, then  
\[
v\in W^{4m+1,\infty}(\mathbb R,\R),
\qquad
w\in W^{4m,\infty}(\mathbb R,\C).
\]
Proposition~\ref{prop:sobolev-group-bound}, applied with \(k=4m\),
therefore supplies the Sobolev group bounds required by
Lemma~\ref{lem:group-fixed-depth}.

%
Since 
$\widetilde A_h= i\delta_h^2+\tfrac12(v\delta_h+\delta_hv)+w$ contains at most two centered-difference factors per
word, Lemma~\ref{lem:defect-order} gives
\[
 R_r:H^s\to H^{s-(2r+2)},
 \qquad
 Q_m(h):H^s\to H^{s-(2m+2)}.
\]
Thus the hypotheses of Lemma~\ref{lem:group-fixed-depth} hold with
$p=2$; equivalently, a complete word of total index $r$ and length
$\ell$ loses at most
\[
 2r+2\ell\le4r.
\]

\begin{corollary}[Schr\"odinger fixed-depth consequence]
\label{cor:schrodinger-richardson}
Under the above assumptions, for every\footnote{Observe that $\mathcal J H^{4m}(\R)=\mathcal J^{-1}H^{4m}(\R)=H^{4m}(\R)$.} $f\in \mathcal J^{-1}H^{4m}=H^{4m}$
and $0\le t\le h_0n$,
\[
 \|\mathcal J^{-1}\mathcal M_{2m,n}(t)\mathcal J f-e^{t A_0}f\|_{L^2}
 \le C_m e^{\Omega_m t}
 \frac{t^{2m+1}}{n^{2m}}
 \|f\|_{H^{4m}}.
\]
\end{corollary}
}

\subsubsection{Third-order dispersive model.}
\label{subsubsec:third-order-model}
\ \\
Consider
$A_0=a(x)D^3+b(x)D+c(x)$, $a\ge a_0>0$, and  set
\[
 q=a^{1/3},\qquad P=qD-q'.
\]
Then
\[
 P^3=aD^3+\alpha D+\beta,
\quad
 \alpha=-\frac23a''+\frac59\frac{(a')^2}{a},
\quad
 \beta=-\frac13a'''+\frac89\frac{a'a''}{a}
       -\frac59\frac{(a')^3}{a^2}.
\]
With $d=b-\alpha$,
\[
 V=(c-\beta)-\tfrac12d'+\tfrac12d\frac{a'}a,
 \qquad
 \eta(x)=\int_0^xa(s)^{-1/3}\,ds,\qquad \chi=\eta^{-1},
\]
and
\[
 (\mathcal Jf)(y)=a(\chi(y))^{-1/3}f(\chi(y)),
\]
one obtains
\[
 \widetilde A:=\mathcal JA_0\mathcal J^{-1}
 =D_y^3+vD_y+\tfrac12v'+w,
 \qquad
 v=(d\,a^{-1/3})\circ\chi,\qquad w=V\circ\chi,
\]
with centered approximation
\[
 \widetilde A_h=\delta_h^3+\tfrac12(v\delta_h+\delta_hv)+w.
\]
If
\[
 a\in W^{5m+3,\infty},\qquad
 b\in W^{5m+1,\infty},\qquad
 c\in W^{5m,\infty},
\]
then, because $\alpha$ uses at most two derivatives of $a$, $\beta$ at
most three, and $\chi'=a^{1/3}\circ\chi$ with $a\ge a_0>0$, repeated
product, chain, and Fa\`a di Bruno estimates give
\[
 v\in W^{5m+1,\infty},\qquad w\in W^{5m,\infty};
\]
the positive lower bound on $a$ ensures that composition with $\chi$
preserves the required $W^{k,\infty}$ bounds up to $k=5m+1$.
Proposition~\ref{prop:sobolev-group-bound} supplies the required group
bounds.  Since $\widetilde A_h$ contains at most three
centered-difference factors per word, Lemma~\ref{lem:defect-order} gives
\[
 R_r:H^s\to H^{s-(2r+3)},
 \qquad
 Q_m(h):H^s\to H^{s-(2m+3)}.
\]
Hence Lemma~\ref{lem:group-fixed-depth} applies with $p=3$; a complete
word of total index $r$ and length $\ell$ loses at most
\[
 2r+3\ell\le5r.
\]

\begin{corollary}[Third-order fixed-depth consequence]
\label{cor:dispersive-richardson}
Under the above assumptions, if\footnote{Observe that $\mathcal J H^{5m}(\R)=\mathcal J^{-1}H^{5m}(\R)=H^{5m}(\R)$.} $f\in \mathcal J^{-1}H^{5m}(\R)=H^{5m}(\R)$ and
$0\le t\le h_0n$, then
\[
 \|\mathcal J^{-1}\mathcal M_{2m,n}(t)\mathcal J f-e^{tA_0}f\|_{L^2}
 \le C_m e^{\Omega_m t}
 \frac{t^{2m+1}}{n^{2m}}
 \|f\|_{H^{5m}}.
\]
\end{corollary}

\begin{remark}[Removal of the quadratic derivative]\label{rem:remove-quadratic-derivative}
The absence of a $D^2$ term is not restrictive.
Indeed, for
\[
 \widehat A=aD^3+\beta D^2+\gamma D+\delta,\qquad a(x)\ne0,
\]
choose a nonvanishing multiplier $\mu$ with
\[
 \frac{\mu'}{\mu}=-\frac{\beta}{3a}.
\]
Since $\mu^{-1}D\mu=D+\mu'/\mu$, the coefficient of $D^2$ in
$\mu^{-1}\widehat A\mu$ is $3a\,\mu'/\mu+\beta=0$. Hence
\[
 \mu^{-1}\widehat A\mu=aD^3+bD+c
\]
for suitable lower-order coefficients $b,c$. Globally, this conjugation
is an equivalence on the function spaces used below whenever $\mu$ and
$\mu^{-1}$ are bounded multipliers there; for example, this holds if
$\beta$ and $a$ are uniformly bounded away from both $0$ and $\infty$, and $\beta/a\in L^1(\R)$.
\end{remark}

\warning{
\begin{remark}[Nonlinear PDEs leading to the three model classes]
\label{rem:nonlinear-parents}
The three linear model classes above arise naturally as linearization 
(variational or tangent equations) of nonlinear PDEs.

\emph{Parabolic case.}
Consider the scalar quasilinear convection--diffusion--reaction equation
\[
 u_t=\partial_x\!\bigl(d(u)u_x\bigr)-\partial_xF(u)+R(u),
 \qquad d(u)>0.
\]
Let $\xi=x-st$ and let $U(\xi)$ be a travelling-wave profile.  Writing
\[
 u(t,x)=U(\xi)+\varepsilon v(t,\xi)+O(\varepsilon^2)
\]
and retaining the terms linear in $\varepsilon$ gives
\begin{align*}
 v_t={}&d(U)v_{\xi\xi}
 +\bigl(2d'(U)U'+s-F'(U)\bigr)v_\xi\\
 &+\bigl(d''(U)(U')^2+d'(U)U''-F''(U)U'+R'(U)\bigr)v.
\end{align*}
Thus the tangent equation has exactly the form
$a(\xi)D^2+b(\xi)D+c(\xi)$.  For viscous Burgers,
$d\equiv\nu$, $F(u)=u^2/2$, and $R\equiv0$, this reduces to
\[
 v_t=\nu v_{\xi\xi}+(s-U)v_\xi-U'v.
\]
For the stationary normalization (see e.g. Guba--Lorenz
\cite[Sec.~1, Eq.~(1)]{GubaLorenz11}) $\nu=1$, $s=0$,  the corresponding Burgers
linearization as $Lv=v''-(Uv)'$.  Likewise, Fisher--KPP,
$d\equiv D$, $F\equiv0$, $R(u)=ru(1-u)$, gives
\[
 v_t=Dv_{\xi\xi}+sv_\xi+r(1-2U)v.
\]

\emph{Schr\"odinger case.}
A direct nonlinear parent of the second model is the semilinear
Schr\"odinger-type equation
\[
 u_t=ia(x)u_{xx}+b(x)u_x+c(x)u+i\gamma(x)|u|^2u.
\]
Since the Fr\'echet derivative of $|u|^2u$ at $u=0$ vanishes, its
linearization at the zero state is exactly
\[
 v_t=\bigl(ia(x)D^2+b(x)D+c(x)\bigr)v.
\]
For a nonzero coherent state the scalar equation is replaced, in
general, by a coupled Hamiltonian linearization whose entries are
second-order Schr\"odinger operators.  For example, consider the cubic focusing NLS
\[
 i\psi_t+\psi_{xx}+|\psi|^2\psi=0.
\]
For a real standing wave $\psi=e^{i\omega t}Q(x)$, the perturbation
$\psi=e^{i\omega t}(Q+r+is)$ yields at first order
\[
 \partial_t\binom r s
 =\begin{pmatrix}0&L_-\\-L_+&0\end{pmatrix}\binom r s,
 \qquad
 L_-=-D^2+\omega-Q^2,\qquad
 L_+=-D^2+\omega-3Q^2.
\]
This standard NLS/GP linearization is written explicitly, for the
corresponding potential problem and general power nonlinearity, in
Gang--Weinstein \cite[Sec.~2.3, Eqs.~(2.2)--(2.4)]{GangWeinstein11}.
Thus the scalar model treated here is the exact tangent evolution at the
zero state, while the same second-order Schr\"odinger operators occur as
blocks of the linearization about nonzero coherent states.

\emph{Third-order dispersive case.}
Consider more generally
\[
 q_t=a(x)q_{xxx}+H(x,q,q_x),
\]
and let $U(x)$ be a stationary profile.  Substituting
$q=U+\varepsilon v+O(\varepsilon^2)$ gives directly
\[
 v_t=a(x)v_{xxx}
 +H_{q_x}(x,U,U')v_x
 +H_q(x,U,U')v,
\]
which is exactly of the form $a(x)D^3+b(x)D+c(x)$.  This is the
Fr\'echet-linearized (symmetry) equation used by G\"urses--Karasu
\cite[Eq.~(2)]{GursesKarasu95}; for the autonomous class
$q_t=q_{xxx}+F(q,q_x)$ they display precisely
$\sigma_t=\sigma_{xxx}+F_{q_x}\sigma_x+F_q\sigma$ in
\cite[Eq.~(11)]{GursesKarasu95}.  As the familiar gKdV specialization,
\[
 u_t+u_{xxx}+\partial_x f(u)=0
\]
linearized in the frame $\xi=x-st$ about a travelling profile $U$ gives
\[
 v_t=-v_{\xi\xi\xi}
 +\bigl(s-f'(U)\bigr)v_\xi
 -f''(U)U'v.
\]
For KdV, $f(u)=3u^2$, hence
$b=s-6U$ and $c=-6U'$.  The sign of the cubic principal part may of
course be reversed by the reflection $\xi\mapsto-\xi$.

These examples explain why the three classes are useful test cases for
the general machinery developed here: they are tangent equations for
nonlinear parabolic fronts, nonlinear Schr\"odinger dynamics, and
KdV-type dispersive waves.  The centered-difference and Richardson
arguments themselves are not tied to these nonlinear equations; they
apply whenever the relevant linearized generator can be put into one of
the structural forms treated above.
\end{remark}

Thus the three fixed-depth regimes reduce to
\[
 H^{2m+\delta}\quad\longrightarrow\quad H^{4m}
 \quad\longrightarrow\quad H^{5m},\qquad \delta>0.
\]
The three lemmas above contain the finite-depth argument needed for these
three statements; the analytic theory below is independent of those
Sobolev estimates and adds quantitative control in the extrapolation
depth.
}

\section{Analytic regime: quantitative all-order bounds and super-fast convergence}
\label{sec:analytic-regime}

Subsection~\ref{subsec:model-summary} derives the even product expansion
\eqref{eq:product-expansion} at every fixed finite depth under suitable 
Sobolev regularity assumptions.  For smooth data this yields formally 
arbitrary algebraic order  through Theorem~\ref{thm:abstract}.
The goal of this section is to make it rigorous under analytic
regularity.  
In particular we shall prove, for the three model
classes treated above: 
\begin{equation}
 C^{\rm an}_{m,T}\le C_TK_T^m\Gamma((p+1)m+1),
 \qquad p=2\text{ or }3,
\label{eq:analytic-bridge-preview}
\end{equation}
where $p$ is the principal differential order.  
Combined with the exact
Richardson weights, this allows the extrapolation depth to grow with
$n$ and yields Gaussian-type convergence $e^{-cn^2}$ for the two
second-order models and exponential convergence $e^{-cn}$ for the
third-order model.

\subsection{Analytic strip spaces and coefficient norms}
\label{subsec:analytic-spaces}

For ${\rho_*}>0$ write
\[
\mathscr{S}_R:=\{z\in\C:\ |\Im z|<{\rho_*}\},
\]
and denote by $\mathcal{H}_{\rho_*}$ the Banach algebra of bounded
holomorphic functions on $\mathscr S_{\rho_*}$, equipped with
\[
\|g\|_{\mathcal{H}_{\rho_*}}
:=
\sup_{z\in \mathscr S_{\rho_*}}|g(z)|.
\]

For $\rho>0$ and $s\ge0$, define\footnote{We use the standard Fourier-multiplier notation: 
$
\langle D\rangle:=(I-D^2)^{1/2},
$
 and 
$
|D|:=(-D^2)^{1/2}.
$ 
Thus, writing
$
\langle\xi\rangle:=(1+|\xi|^2)^{1/2},
$
one has
$
\widehat{\langle D\rangle^s f}(\xi)
=
\langle\xi\rangle^s\widehat f(\xi),
$
and 
$ 
\widehat{e^{\rho|D|}f}(\xi)
=
e^{\rho|\xi|}\widehat f(\xi).
$}

\[
\mathcal A_\rho^s
:=
\left\{
f:\ e^{\rho|D|}\langle D\rangle^s f\in L^2(\R)
\right\},
\qquad
\|f\|_{\rho,s}
:=
\|e^{\rho|D|}\langle D\rangle^s f\|_{L^2},
\]
and abbreviate
\[
\mathcal A_\rho:=\mathcal A_\rho^0,
\qquad
\|f\|_\rho:=\|f\|_{\rho,0}.
\]

If $0<\rho'<\rho$, then
\begin{equation}
\|D^qf\|_{\rho',s}
\le
\frac{q!}{(\rho-\rho')^q}\,
\|f\|_{\rho,s},
\qquad q\in\N_0.
\label{eq:analytic-cauchy-fourier}
\end{equation}

For later energy estimates it is convenient to introduce the equivalent
boundary norm
\[
\|f\|_{\partial,\rho}^2
:=
\|f(\cdot+i\rho)\|_{L^2}^2
+
\|f(\cdot-i\rho)\|_{L^2}^2.
\]
Indeed, by Plancherel,
\[
\|f\|_{\partial,\rho}^2
=
2\int_\R
\cosh(2\rho\xi)|\widehat f(\xi)|^2\,d\xi,
\]
and therefore
\begin{equation}
\|f\|_\rho^2
\le
\|f\|_{\partial,\rho}^2
\le
2\|f\|_\rho^2.
\label{eq:boundary-fourier-equivalence}
\end{equation}

\bwarning{The principal advantage of the strip formulation is that bounded
holomorphic coefficients act as multipliers without loss of analytic
radius.}

\begin{lemma}[Bounded analytic multipliers]
\label{lem:analytic-multiplier}
Let $g\in \mathcal{H}_{\rho_*}$ and $0<\rho<{\rho_*}$. Then
\begin{equation}
\|gf\|_\rho
\le
\sqrt2\,\|g\|_{\mathcal{H}_{\rho_*}}\|f\|_\rho.
\label{eq:analytic-multiplier}
\end{equation}
Moreover, if $0<\rho<\rho_*$, Cauchy's estimate gives
\begin{equation}
\|g^{(q)}\|_{\mathcal{H}_{\rho}}
\le
\frac{q!}{(\rho_*-\rho)^q}\,
\|g\|_{\mathcal{H}_{\rho_*}},
\qquad q\in\N_0.
\label{eq:coefficient-cauchy} 
\end{equation}
\end{lemma}

\begin{proof}
On the two boundary lines of $\mathscr S_\rho$,
\[
\|g(\cdot\pm i\rho)f(\cdot\pm i\rho)\|_{L^2}
\le
\|g\|_{\mathcal{H}_{\rho_*}}
\|f(\cdot\pm i\rho)\|_{L^2}.
\]
Hence
\[
\|gf\|_{\partial,\rho}
\le
\|g\|_{\mathcal{H}_{\rho_*}}\|f\|_{\partial,\rho},
\]
whence \eqref{eq:analytic-multiplier}  follows from
\eqref{eq:boundary-fourier-equivalence}.  The derivative bound is the
standard Cauchy estimate on nested strips.
\end{proof}

\noindent
Moreover, under the notations in Theorem~\ref{thm:unified-analytic-rate}: 

\begin{lemma}[Quantitative analytic Liouville coordinates]
\label{lem:analytic-liouville}
Let $\varrho_0>0$ and assume
\[
a\in \mathcal{H}_{\varrho_0}, 
\qquad
a(\bar z)=\overline{a(z)},
\]
and
\[
a(x)\ge a_0>0,
\qquad x\in\R.
\]
%
\warning{
Set
\[
a_\infty:=\|a\|_{\mathcal H_{\varrho_0}}
\]
and
\begin{equation}
r_a
:=
\frac{a_0\rho_0}{4a_\infty}.
\label{eq:quantitative-ra}
\end{equation}
}
Then
\begin{equation}
\Re a(z)\ge \frac{a_0}{2},
\qquad z\in \mathscr S_{r_a}.
\label{eq:quantitative-complex-ellipticity}
\end{equation}

For $p\in\{2,3\}$ let
\[
q_p(z):=a(z)^{-1/p},
\]
where the principal branch is taken on the right half-plane.
Then $q_p$ is bounded and holomorphic on $\mathscr S_{r_a}$ and
\begin{equation}
\Re q_p(z)\ge m_p,
\qquad
|q_p(z)|\le M_p,
\qquad z\in \mathscr S_{r_a}.
\label{eq:mp-Mp-bounds}
\end{equation}

\warning{
where
\begin{equation}
M_p
:=
\left(\frac{2}{a_0}\right)^{1/p}
\label{eq:Mp-explicit}
\end{equation}
and
\begin{equation}
m_p
:=
a_\infty^{-1/p}
\cos\left(
\frac1p
\arccos\frac{a_0}{2a_\infty}
\right).
\label{eq:mp-explicit}
\end{equation}
}
Define
\[
\eta_p(z)
:=
\int_0^z q_p(\zeta)\,d\zeta.
\]
Then $\eta_p$ is injective on $\mathscr S_{r_a}$ and
\begin{equation}
S_{m_pr_a}
\subset
\eta_p(S_{r_a}).
\label{eq:eta-covers-strip}
\end{equation}
Consequently
\[
\chi_p:=\eta_p^{-1}
\]
is holomorphic on $\mathscr S_{m_pr_a}$ and satisfies
\begin{equation}
\chi_p(\mathscr S_{\rho_y})
\subset
\mathscr S_{\rho_y/m_p},
\qquad
0<\rho_y<m_pr_a.
\label{eq:chi-strip-inclusion}
\end{equation}
Moreover,
\begin{equation}
\eta_p(\mathscr S_{\rho_x})
\subset
\mathscr S_{M_p\rho_x},
\qquad
0<\rho_x<r_a.
\label{eq:eta-strip-inclusion}
\end{equation}

\warning{
In particular, the Liouville distortion ratio satisfies
\begin{equation}
\gamma_p
:=
\frac{m_p}{M_p}
\ge
\left(
\frac{a_0}{2a_\infty}
\right)^{1/p}
\cos\left(
\frac1p
\arccos\frac{a_0}{2a_\infty}
\right)
>0.
\label{eq:liouville-distortion}
\end{equation}
}
\end{lemma}

\begin{proof}
Since $a_\infty\ge a_0$, one has
\[
r_a
=
\frac{a_0\varrho_0}{4a_\infty}
\le
\frac{\varrho_0}{4}.
\]
Cauchy's estimate on $\mathscr S_{\rho_0/2}$ yields
\[
\|a'\|_{\mathcal H_{\rho_0/2}}
\le
\frac{2a_\infty}{\varrho_0}.
\]
Hence, for $z=x+iy$ with $|y|<r_a$,
\[
|a(x+iy)-a(x)|
\le
|y|\,
\|a'\|_{\mathcal H_{\rho_0/2}}
\le
\frac{2a_\infty r_a}{\varrho_0}
=
\frac{a_0}{2}.
\]
Since $a(x)\ge a_0$ on $\R$,
\[
\Re a(z)
\ge
a_0-|a(z)-a(x)|
\ge
\frac{a_0}{2},
\]
which proves \eqref{eq:quantitative-complex-ellipticity}.

Thus $a(\mathscr S_{r_a})$ lies in the open right half-plane and has no zeros
there.  The principal branch of $a^{-1/p}$ is therefore holomorphic
on $\mathscr S_{r_a}$.  Since
\[
|a(z)|\ge \Re a(z)\ge\frac{a_0}{2},
\]
we have
\[
|q_p(z)|
=
|a(z)|^{-1/p}
\le
\left(\frac{2}{a_0}\right)^{1/p}
=
M_p.
\]

To obtain a quantitative lower bound for $\Re q_p$, write
\[
a(z)=|a(z)|e^{i\theta(z)}.
\]
Since
\[
\cos\theta(z)
=
\frac{\Re a(z)}{|a(z)|}
\ge
\frac{a_0}{2a_\infty},
\]
we have
\[
|\theta(z)|
\le
\arccos\frac{a_0}{2a_\infty}.
\]
Consequently,
\begin{align*}
\Re q_p(z)
&=
|a(z)|^{-1/p}
\cos\frac{\theta(z)}p
\\
&\ge
a_\infty^{-1/p}
\cos\left(
\frac1p
\arccos\frac{a_0}{2a_\infty}
\right)
=
m_p.
\end{align*}

For $z_1,z_2\in S_{r_a}$, convexity of the strip yields
\[
\frac{\eta_p(z_2)-\eta_p(z_1)}
     {z_2-z_1}
=
\int_0^1
q_p(z_1+\theta(z_2-z_1))\,d\theta.
\]
Therefore
\[
\Re
\frac{\eta_p(z_2)-\eta_p(z_1)}
     {z_2-z_1}
\ge m_p>0,
\]
and hence $\eta_p$ is injective. 
Furthermore, for $z=x+iy$,
\begin{equation}
\Im\eta_p(x+iy)
=
\int_0^y
\Re q_p(x+is)\,ds.
\label{eq:imag-eta-formula}
\end{equation}
Thus
\begin{equation}
m_p|y|
\le
|\Im\eta_p(x+iy)|
\le
M_p|y|.
\label{eq:imag-eta-two-sided}
\end{equation}
In particular,
\[
\Im\eta_p(x+ir_a)\ge m_pr_a,
\qquad
\Im\eta_p(x-ir_a)\le -m_pr_a.
\]
Moreover,
\[
\partial_x\Re\eta_p(x+iy)
=
\Re q_p(x+iy)
\ge m_p,
\]
so that
\[
\Re\eta_p(x+iy)\longrightarrow\pm\infty
\qquad
\text{as }x\longrightarrow\pm\infty,
\]
uniformly for $|y|\le r_a$.

Fix $w$ with $|\Im w|<m_pr_a$.  For $X$ sufficiently large, on
the boundary of
\[
\Omega_X
=
\{z:\ |\Re z|<X,\ |\Im z|<r_a\},
\]
the four components of
\[
\eta_p(\partial\Omega_X)-w
\]
lie respectively above, to the right of, below, and to the left of
the origin.  Hence this closed curve has winding number one about
the origin.  The argument principle gives a zero of
$\eta_p(z)-w$ in $\Omega_X$; injectivity makes it unique.  This proves
\eqref{eq:eta-covers-strip}.

If $w=\eta_p(z)$, then \eqref{eq:imag-eta-two-sided} gives
\[
|\Im z|
\le
\frac{|\Im w|}{m_p},
\]
which is \eqref{eq:chi-strip-inclusion}.  The upper estimate in
\eqref{eq:imag-eta-two-sided} gives
\eqref{eq:eta-strip-inclusion}.  
\end{proof}
Finally, the following lemma allows us to transfer estimates between the original and normalized coordinates.

\begin{lemma}[Liouville conjugation on analytic strip spaces]
\label{lem:liouville-analytic-spaces}
Assume the hypotheses of Lemma~\ref{lem:analytic-liouville}.  Set
\[
\vartheta_2:=\frac14,
\qquad
\vartheta_3:=\frac13,
\]
and, for $p=2,3$, define
\[
(\mathcal J_pf)(y)
:=
a(\chi_p(y))^{-\vartheta_p}
f(\chi_p(y)).
\]
Then the following assertions hold. 
If
\begin{equation}
0<\rho_y<m_p\rho_x,
\qquad
0<\rho_x<r_a,
\label{eq:J-compatible-radii}
\end{equation}
then
\begin{equation}
\mathcal J_p:
\mathcal A_{\rho_x}
\longrightarrow
\mathcal A_{\rho_y}
\label{eq:J-map-analytic}
\end{equation}
is bounded. 
Conversely, if
\begin{equation}
0<M_p\rho_x<\rho_y<m_pr_a,
\label{eq:Jinv-compatible-radii}
\end{equation}
then
\begin{equation}
\mathcal J_p^{-1}:
\mathcal A_{\rho_y}
\longrightarrow
\mathcal A_{\rho_x}
\label{eq:Jinv-map-analytic}
\end{equation}
is bounded.

In particular, $\mathcal J_p$ and $\mathcal J_p^{-1}$ identify the
original and normalized analytic scales up to the radius distortion
controlled by $m_p$ and $M_p$.
\end{lemma}

\begin{proof}
We first record an elementary trace estimate.  Let
\[
0<r<R
\]
and let $Y:\R\to[-r,r]$ be measurable.  If
$f\in\mathcal A_R$, then
\begin{equation}
\int_\R
|f(x+iY(x))|^2\,dx
\le
C_{\rm tr}(R,r)^2
\|f\|_R^2,
\label{eq:variable-graph-trace}
\end{equation}
where one may take
\[
C_{\rm tr}(R,r)^2
:=
2+\frac{4r^2}{(R-r)^2}.
\]
Indeed,
\[
f(x+iY(x))
=
f(x)
+
i\int_0^{Y(x)}f'(x+is)\,ds,
\]
so Cauchy--Schwarz gives
\begin{align*}
|f(x+iY(x))|^2
&\le
2|f(x)|^2
+
2r
\int_{-r}^r
|f'(x+is)|^2\,ds.
\end{align*}
After integration in $x$,
\[
\int_\R|f(x+iY(x))|^2\,dx
\le
2\|f\|_{L^2}^2
+
4r^2
\sup_{|s|\le r}
\|f'(\cdot+is)\|_{L^2}^2.
\]
By Plancherel,
\[
\|f'(\cdot+is)\|_{L^2}
\le
\frac1{R-r}\|f\|_R,
\qquad |s|\le r,
\]
which proves \eqref{eq:variable-graph-trace}.

We now prove the bound for $\mathcal J_p$.  Fix
$|\sigma|\le\rho_y<{m_p}\rho_x<m_p r_a$ and write
\[
\chi_p(x+i\sigma)
=
X_\sigma(x)+iY_\sigma(x).
\]
By \eqref{eq:chi-strip-inclusion},
\[
|Y_\sigma(x)|
\le
\frac{\rho_y}{m_p}
<\rho_x.
\]
Moreover,
\[
\chi_p'(w)
=
\frac1{q_p(\chi_p(w))},
\]
and therefore
\[
\frac{dX_\sigma}{dx}
=
\Re\chi_p'(x+i\sigma)
=
\frac{\Re q_p(\chi_p(x+i\sigma))}
     {|q_p(\chi_p(x+i\sigma))|^2}
\ge
\frac{m_p}{M_p^2}.
\]
Thus $X_\sigma$ is strictly increasing from $\R$ onto $\R$, and
\[
dx
\le
\frac{M_p^2}{m_p}\,dX_\sigma.
\]
Since
\[
|a(\chi_p(w))|^{-\vartheta_p}
\le
\left(\frac2{a_0}\right)^{\vartheta_p},
\]
the variable-graph estimate gives
\begin{align*}
\|
(\mathcal J_pf)(\cdot+i\sigma)
\|_2^2
&\le
\left(\frac2{a_0}\right)^{2\vartheta_p}
\frac{M_p^2}{m_p}
\times
C_{\rm tr}
\left(
\rho_x,\frac{\rho_y}{m_p}
\right)^2
\|f\|_{\rho_x}^2.
\end{align*}
Applying this at $\sigma=\pm\rho_y$ and using the equivalence between
the boundary norm and $\|\cdot\|_{\rho_y}$ proves
\eqref{eq:J-map-analytic}.

\noindent
For the inverse transformation,
\[
(\mathcal J_p^{-1}g)(x)
=
a(x)^{\vartheta_p}g(\eta_p(x)).
\]
Fix $|\sigma|\le\rho_x<\rho_y/M_p<\gamma_p r_a$ and write
\[
\eta_p(x+i\sigma)
=
U_\sigma(x)+iV_\sigma(x).
\]
By \eqref{eq:eta-strip-inclusion},
\[
|V_\sigma(x)|
\le
M_p\rho_x
<\rho_y.
\]
Moreover,
\[
\frac{dU_\sigma}{dx}
=
\Re q_p(x+i\sigma)
\ge m_p,
\]
so
\[
dx\le \frac1{m_p}\,dU_\sigma.
\]
Since
\[
|a(x+i\sigma)|^{\vartheta_p}
\le a_\infty^{\vartheta_p},
\]
another application of \eqref{eq:variable-graph-trace} gives
\begin{align*}
\|
(\mathcal J_p^{-1}g)(\cdot+i\sigma)
\|_{L^2}^2
&\le
\frac{a_\infty^{2\vartheta_p}}{m_p}
C_{\rm tr}
\left(
\rho_y,M_p\rho_x
\right)^2
\|g\|_{\rho_y}^2.
\end{align*}
Taking $\sigma=\pm\rho_x$ proves
\eqref{eq:Jinv-map-analytic}.
\end{proof}

\warning{
\begin{lemma}[Analytic Liouville coordinates]
\label{lem:analytic-liouville}
Let $\rho_0>0$ and assume
\[
a\in \mathcal{H}_{\rho_0}, 
\qquad
a(\bar z)=\overline{a(z)},
\]
and
\[
a(x)\ge a_0>0,
\qquad x\in\R.
\]
For $p\in\{2,3\}$ define on the real axis
\[
\eta_p(x)
:=
\int_0^x a(s)^{-1/p}\,ds .
\]
Then there exist $r_p>0$ and $\widetilde r_p>0$ such that:

\begin{enumerate}
\item the principal branch
\[
q_p(z):=a(z)^{-1/p}
\]
is bounded and holomorphic on $\mathscr S_{r_p}$;

\item there exists $m_p>0$ such that
\[
\Re q_p(z)\ge m_p,
\qquad z\in \mathscr S_{r_p};
\]

\item
\[
\eta_p(z):=\int_0^zq_p(\zeta)\,d\zeta
\]
is injective on $\mathscr S_{r_p}$;

\item
\[
S_{\widetilde r_p}\subset\eta_p(\mathscr S_{r_p}),
\qquad
\widetilde r_p:=m_pr_p,
\]
and therefore
\[
\chi_p:=\eta_p^{-1}
\]
is holomorphic on $\mathscr S_{\widetilde r_p}$ with
\[
\chi_p(\mathscr S_{\widetilde r_p})\subset \mathscr S_{r_p}.
\]
\end{enumerate}

In particular, if $F\in \mathcal{H}_{r_p}\ \bwarning{H^\infty(S_{r_p})}$, then
\[
F\circ\chi_p\in \mathcal{H}_{\widetilde r_p}
\]
and
\[
\|F\circ\chi_p\|_{\mathcal{H}_{\widetilde r_p}}
\le
\|F\|_{\mathcal{H}_{r_p}}.
\]
\end{lemma}

\begin{proof}
Since $a\in \mathcal{H}_{\rho_0}$, its restriction to the real axis is
automatically bounded above:
\[
a(x)\le \|a\|_{\mathcal{H}_{\rho_0}}.
\]
\bwarning{Thus no separate upper ellipticity bound is required.}

Choose $0<r_p<\rho_0$ sufficiently small.  Cauchy's estimate and
$a(x)\ge a_0$ imply, after decreasing $r_p$ if necessary,
\[
|a(x+iy)-a(x)|\le \frac{a_0}{2},
\qquad |y|<r_p.
\]
Hence $a$ has no zeros in $\mathscr S_{r_p}$ and the branch $a^{-1/p}$ is
holomorphic there.  Since $a^{-1/p}$ is positive on $\R$, up to a
strip shrinkage, we have
\[
\Re q_p(z)\ge m_p>0.
\]

For $z_1,z_2\in \mathscr S_{r_p}$, convexity of the strip gives
\[
\frac{\eta_p(z_2)-\eta_p(z_1)}{z_2-z_1}
=
\int_0^1
q_p(z_1+\theta(z_2-z_1))\,d\theta ,
\]
so
\[
\Re
\frac{\eta_p(z_2)-\eta_p(z_1)}{z_2-z_1}
\ge m_p.
\]
Thus $\eta_p(z_1)\neq\eta_p(z_2)$ whenever $z_1\neq z_2$.

Moreover, since $\eta_p(x)\in\R$ for $x\in\R$,
\[
\Im\eta_p(x+iy)
=
\int_0^y\Re q_p(x+is)\,ds.
\]
Consequently,
\[
\Im\eta_p(x+ir_p)\ge m_pr_p,
\qquad
\Im\eta_p(x-ir_p)\le-m_pr_p.
\]
Also,
\[
\frac{\partial}{\partial x}\Re\eta_p(x+iy)
=
\Re q_p(x+iy)\ge m_p,
\]
so $\Re\eta_p(x+iy)\to\pm\infty$ as $x\to\pm\infty$,
uniformly for $|y|\le r_p$.

Fix $w$ with $|\Im w|<m_pr_p$.  On the boundary of a sufficiently
large rectangle
\[
\Omega_X=\{z:\ |\Re z|<X,\ |\Im z|<r_p\},
\]
the curve $\eta_p(\partial\Omega_X)-w$ has winding number one about
the origin.  By the argument principle, $\eta_p(z)-w$ has (exacty) one zero
inside $\Omega_X$.  
Thus
\[
S_{m_pr_p}\subset\eta_p(S_{r_p}),
\]
which proves the assertion.
\end{proof} 
 }

\bwarning
{
\emph{Analytic strip spaces and coefficient norms}
\label{subsec:analytic-spaces}

For $\rho>0$ define\footnote{We use the standard Fourier-multiplier notation: 
$
\langle D\rangle:=(I-D^2)^{1/2},
$
 and 
$
|D|:=(-D^2)^{1/2}.
$ 
Thus, writing
$
\langle\xi\rangle:=(1+|\xi|^2)^{1/2},
$
one has
$
\widehat{\langle D\rangle^s f}(\xi)
=
\langle\xi\rangle^s\widehat f(\xi),
$
and 
$ 
\widehat{e^{\rho|D|}f}(\xi)
=
e^{\rho|\xi|}\widehat f(\xi).
$}

\[
 \mathcal A_\rho^s
 :=\bigl\{f:\ e^{\rho|D|}\langle D\rangle^sf\in L^2(\R)\bigr\},
 \qquad
 \|f\|_{\rho,s}:=
 \|e^{\rho|D|}\langle D\rangle^sf\|_{L^2},
\]
and abbreviate $\mathcal A_\rho=\mathcal A_\rho^0$.  If
$0<\rho'<\rho$, the Fourier multiplier estimate yields
\begin{equation}
 \|D^qf\|_{\rho',s}
 \le \frac{q!}{(\rho-\rho')^q}\|f\|_{\rho,s}
\label{eq:analytic-cauchy-estimate}
\end{equation}
since 
$\sup_{x\ge0}x^qe^{-(\rho-\rho')x}\le q!(\rho-\rho')^{-q}$.

For coefficients we use an analytic Wiener norm. For $g\in L^1(\C)$,  its Fourier transform
 gives a finite complex measure, and we define
\[
 \|g\|_{\mathfrak W_R}
 :=\int_\R e^{R|\eta|}\,d|\widehat g|(\eta),
 \qquad R>0,
\]
and, for $j\ge0$ and $0\le\rho<R$,
\[
 M_j(g;\rho)
 :=\int_\R |\eta|^j e^{\rho|\eta|}\,d|\widehat g|(\eta).
\]
Then
\begin{equation}
 M_j(g;\rho)
 \le \frac{j!}{(R-\rho)^j}\|g\|_{\mathfrak W_R}.
\label{eq:wiener-moment-estimate}
\end{equation}
\warning{The right-hand side deteriorates as $\rho\uparrow R$.  We therefore} fix
throughout a working radius $\rho_0<R$ and use only
$0\le\rho\le\rho_0$, for which all coefficient moments entering the
energy estimates are uniformly controlled. 
\warning{ The same arguments have a periodic version with Fourier series.}

Let $W_\rho=e^{\rho|D|}$.  Then: \warning{A direct Fourier-convolution calculation
also gives the commutator estimate}

\begin{lemma}\label{lem:analytic-commutator}
\begin{equation}
 \|[W_\rho,g]u\|_2
 \le \rho M_1(g;\rho)\|W_\rho u\|_2.
\label{eq:analytic-commutator}
\end{equation}
\end{lemma}

\begin{proof}[Proof of \eqref{eq:analytic-commutator}]
The operator $W_\rho$ is the Fourier multiplier with symbol
$e^{\rho|\xi|}$, whereas multiplication by $g$ corresponds on the
Fourier side to convolution with the finite complex measure
$\widehat g$.  Thus
\[
 \widehat{[W_\rho,g]u}(\xi)
 =
 \int_\R
 \bigl(e^{\rho|\xi|}-e^{\rho|\xi-\eta|}\bigr)
 \widehat u(\xi-\eta)\,d\widehat g(\eta).
\]
For real $a,b$ one has
$|e^a-e^b|\le |a-b|e^{\max\{a,b\}}$.  Hence
\[
 \bigl|e^{\rho|\xi|}-e^{\rho|\xi-\eta|}\bigr|
 \le
 \rho|\eta|e^{\rho|\eta|}e^{\rho|\xi-\eta|},
\tag{$\star$}\label{eq:analytic-exponential-difference}
\]
because
$\bigl||\xi|-|\xi-\eta|\bigr|\le|\eta|$ and
$\max\{|\xi|,|\xi-\eta|\}\le|\xi-\eta|+|\eta|$.
Taking absolute values and using \eqref{eq:analytic-exponential-difference} gives
\[
 |\widehat{[W_\rho,g]u}(\xi)|
 \le
 \rho\int_\R
 |\eta|e^{\rho|\eta|}
 |\widehat{W_\rho u}(\xi-\eta)|\,d|\widehat g|(\eta).
\]
Introduce the positive finite measure
\[
 d\mu(\eta):=|\eta|e^{\rho|\eta|}\,d|\widehat g|(\eta),
 \qquad
 \mu(\R)=M_1(g;\rho).
\]
Then
\[
 |\widehat{[W_\rho,g]u}|
 \le \rho\,\mu*|\widehat{W_\rho u}|.
\]
Young's inequality for convolution of an $L^2$ function with a finite
measure therefore yields
\[
 \|\widehat{[W_\rho,g]u}\|_2
 \le
 \rho\,\mu(\R)\|\widehat{W_\rho u}\|_2.
\]
By Plancherel,
\[
 \|[W_\rho,g]u\|_2
 \le \rho M_1(g;\rho)\|W_\rho u\|_2,
\]
which is \eqref{eq:analytic-commutator}.  Combining this with \eqref{eq:wiener-moment-estimate} gives
\[
 \|[W_\rho,g]u\|_2
 \le
 \frac{\rho}{R-\rho}\,
 \|g\|_{\mathfrak W_R}\|W_\rho u\|_2.
\]
\end{proof}
}

We finally prove the analyticity of the normalized coefficients.

\begin{proposition}[Analyticity of the normalized coefficients]
\label{prop:normalized-analytic-coefficients}
Assume the coefficient hypotheses of
Theorem~\ref{thm:unified-analytic-rate}.  Then, after restriction to
the strip $\mathscr S_{r_a}$ of Lemma~\ref{lem:analytic-liouville}, the
coefficients entering each normalized model are bounded and
holomorphic on a positive strip.

More precisely, in the Schr\"odinger case,
\[
F_2
:=
\frac{b}{\sqrt a},
\qquad 
G_2
:=
c-\frac12b'
+\frac12b\frac{a'}a
+i\left(
\frac14a''
-\frac{3}{16}\frac{(a')^2}{a}
\right)
\]
belong to $\mathcal H_{r_a}$, and
\[
v=F_2\circ\chi_2,
\qquad
w=G_2\circ\chi_2
\]
belong to $\mathcal H_{m_2r_a}$.

In the cubic case, with
\[
\alpha
=
-\frac23a''
+\frac59\frac{(a')^2}{a},
\qquad
\beta
=
-\frac13a'''
+\frac89\frac{a'a''}{a}
-\frac59\frac{(a')^3}{a^2},
\qquad
d=b-\alpha,
\qquad
V=(c-\beta)-\frac12d'
+\frac12d\frac{a'}a,
\]
the functions
\[
F_3:=d\,a^{-1/3},
\qquad
G_3:=V
\]
belong to $\mathcal H_{r_a}$, and
\[
v=F_3\circ\chi_3,
\qquad
w=G_3\circ\chi_3
\]
belong to $\mathcal H_{m_3r_a}$.

In both nonsmoothing cases $v$ is real-valued on $\R$.
\end{proposition}

\begin{proof}
Since
\[
r_a\le \frac{\varrho_0}{4},
\]
all derivatives of $a,b,c$ of any fixed order are bounded on
$S_{r_a}$ by Cauchy's estimates on 
$\mathscr S_{\varrho_0/2}$.
Furthermore,
\[
|a(z)|\ge\frac{a_0}{2},
\qquad z\in \mathscr S_{r_a},
\]
by Lemma~\ref{lem:analytic-liouville}.  Thus all reciprocal and
fractional-power expressions above are bounded and
holomorphic on $S_{r_a}$.

Lemma~\ref{lem:analytic-liouville} implies
\[
\chi_p(S_{m_pr_a})\subset S_{r_a}.
\]
Therefore composition with $\chi_p$ preserves bounded holomorphy on a smaller strip.  The reality of $v$ on $\R$ follows from
the reality of $a,b$ there.
\end{proof}

\subsection{Uniform analytic propagation for the nonsmoothing models}\label{subsec:analytic-group-propagation}

Let $Q_h$ be the skew-adjoint Fourier multiplier with symbol
$iq_h(\xi)$, where $q_h$ is real and
\begin{equation}
 |q_h(\xi)|\le |\xi|.
\label{eq:fourier-symbol-bound}
\end{equation}
For the centered difference, $q_h(\xi)=\sin(h\xi)/h$; for $h=0$ we
set $Q_0=D$.  For $p=2,3$ put
\[
 P_{2,h}=iQ_h^2,\qquad P_{3,h}=Q_h^3,
\]
and consider
\begin{equation}
 L_{p,h}=P_{p,h}+\frac12(vQ_h+Q_hv)+w.
\label{eq:normalized-analytic-generator}
\end{equation}
The normalized Schr\"odinger and cubic generators summarized in Subsection~\ref{subsec:model-summary}
are the cases $p=2$ and $p=3$ respectively.

\bwarning{
\begin{proposition}[Uniform analytic propagation in the nonsmoothing cases]
\label{prop:analytic-group}

Let
\[
L_{p,h}
=
P_{p,h}
+\frac12(vQ_h+Q_hv)+w,
\qquad p\in\{2,3\},
\]
where
\[
P_{2,h}=iQ_h^2,
\qquad
P_{3,h}=Q_h^3,
\]
and $Q_h$ is skew-adjoint with real symbol $iq_h(\xi)$ satisfying
\[
|q_h(\xi)|\le|\xi|.
\]

Assume that, for some $R>0$,
\[
v,w\in H^\infty(S_R),
\qquad
v(\bar z)=\overline{v(z)}.
\]
Fix $0<\rho_0<R$.  Then there exist
\[
C_0\ge1,\qquad \kappa,\Gamma\ge0,
\]
independent of $h$, such that
\begin{equation}
\|e^{tL_{p,h}}f\|_{\rho_0e^{-\kappa|t|}}
\le
C_0e^{\Gamma|t|}\|f\|_{\rho_0},
\qquad t\in\R.
\label{eq:analytic-group-propagation-new}
\end{equation}
The same estimate holds for $h=0$.
\end{proposition}
}

{
\begin{proposition}[Uniform analytic propagation in the group cases]
\label{prop:analytic-group}
Assume that $v$ is real-valued and $v,w\in\mathcal{H}_{\rho_*}$ for some
$\rho_*>0$.  Fix $0<\rho_0<\rho_*$.  Then there exist 
$C_0\ge 1$ and $\kappa,\Gamma\ge0$,
depending on $\rho_0$ and on the bounded analytic norms of $v,w$, but
independent of $h$, such
that
\begin{equation}
 \|e^{tL_{p,h}}f\|_{\rho_0e^{-\kappa|t|}}
 \le C_0 e^{\Gamma|t|}\|f\|_{\rho_0},
 \qquad t\in\R,
\label{eq:analytic-group-propagation-new}
\end{equation}
for $p=2,3$ and $h\in\R$.   
\warning{If $h=0$ is interpreted as $Q_0=D$, then
$L_{p,0}$ is the exact normalized differential generator and the same
estimate holds for its $C_0$-group on $L^2(\R)$.}
\end{proposition}
}

\begin{proof}
We divide the proof into three steps.

\medskip
\noindent
\emph{Step 1: a boundary-strip energy estimate for Galerkin
approximations.}
Let $h\ne0$.  
\bwarning{
Let
\[
\mathscr E
:=
\bigcup_{N\ge1}\operatorname{Ran}\Pi_N,
\qquad
\Pi_N:=\mathbf 1_{\{|\xi|\le N\}}(D).
\]
The space $\mathscr E$ is dense in every $\mathcal A_\rho$.

Fix first $h\ne0$ and $f\in\mathscr E$.  Introduce
\[
L_{p,h}^{(M)}
:=
\Pi_ML_{p,h}\Pi_M,
\qquad
u_M(t)
:=
e^{tL_{p,h}^{(M)}}f,
\]
with $M$ sufficiently large that
$f\in\operatorname{Ran}\Pi_M$.
Then
\[
u_M(t)\in\operatorname{Ran}\Pi_M
\qquad
(t\in\R),
\]
so $u_M(t)$ extends to an entire function and all the calculations
below are classical.
}

We use as analytic core the band-limited entire functions
\[
 \mathscr E:=\bigcup_{N\ge1}\operatorname{Ran}\Pi_N,
 \qquad
 \Pi_N:=\mathbf 1_{\{|\xi|\le N\}}(D).
\]
If $f\in\mathscr E$, then $\widehat f$ has compact support.  Hence the
inverse Fourier formula extends $f$ to an entire function on $\C$, and
for every $\rho>0$ and $s\ge0$ one has
$f\in\mathcal A_\rho^s$.  Moreover, $\mathscr E$ is dense in
$\mathcal A_{\rho_0}$, since
\[
 \|f-\Pi_Nf\|_{\rho_0}^2
 =\int_{|\xi|>N}e^{2\rho_0|\xi|}|\widehat f(\xi)|^2\,d\xi
 \longrightarrow0.
\]

We first establish the estimate on $\mathscr E$. \warning{ To justify the
weighted energy identity without assuming in advance that the exact
$L_{p,h}$-flow preserves analytic regularity,} Introduce an auxiliary
Fourier--Galerkin cutoff
\[
 L_{p,h}^{(M)}:=\Pi_ML_{p,h}\Pi_M,
 \qquad
 u_M(t):=e^{tL_{p,h}^{(M)}}f,
\]
with $M$ large enough that $f\in\operatorname{Ran}\Pi_M$.
The outer projection in $L_{p,h}^{(M)}$ makes
$\operatorname{Ran}\Pi_M$ invariant.  Indeed, if
$u\in\operatorname{Ran}\Pi_M$, then
\[
 L_{p,h}^{(M)}u=\Pi_ML_{p,h}u\in\operatorname{Ran}\Pi_M,
\]
and therefore the exponential series for the bounded operator
$L_{p,h}^{(M)}$ yields
\[
 u_M(t)=e^{tL_{p,h}^{(M)}}f\in\operatorname{Ran}\Pi_M
 \qquad (t\in\R).
\]
Equivalently,
$\operatorname{supp}\widehat{u_M(t)}\subset[-M,M]$ for every $t$. 

For $0<\rho<\rho_*$ define the two boundary traces
\[
U_\pm(t,x)
:=
u_M(t,x\pm i\rho(t))
\]
and the boundary energy
\begin{equation}
E(t)
:=
\|U_+(t)\|_2^2+\|U_-(t)\|_2^2.
\label{eq:boundary-energy}
\end{equation}
By Plancherel,
\begin{equation}
E(t)
=
2\int_{\R}
\cosh(2\rho(t)\xi)
|\widehat{u_M}(t,\xi)|^2\,d\xi.
\label{eq:boundary-energy-fourier}
\end{equation}
Likewise set
\begin{equation}
Z(t)
:=
\||D|^{1/2}U_+(t)\|_{L^2}^2
+
\||D|^{1/2}U_-(t)\|_{L^2}^2 .
\label{eq:boundary-half-energy}
\end{equation}

Fix for the moment one horizontal trace
\[
U_\sigma(x):=u_M(x+i\sigma),
\qquad |\sigma|<\rho_*.
\]

The principal term $P_{p,h}$ is skew-adjoint on every horizontal line. 
Hence
\begin{equation}
\Re\langle P_{p,h}U_\sigma,U_\sigma\rangle=0.
\label{eq:principal-horizontal-skew}
\end{equation}

Consider now the first-order part
\[
B_h(v)
:=
\frac12(vQ_h+Q_hv).
\]
Write
\[
v(x+i\sigma)
=
\alpha_\sigma(x)+i\beta_\sigma(x),
\qquad
\alpha_\sigma,\beta_\sigma\in\R.
\]
Since $\alpha_\sigma$ is real,
\[
B_h(\alpha_\sigma)^*
=
-B_h(\alpha_\sigma),
\]
and therefore only $\beta_\sigma$ contributes to the real part:
\[
\Re\langle B_h(v_\sigma)U_\sigma,U_\sigma\rangle
=
\Re\langle iB_h(\beta_\sigma)U_\sigma,U_\sigma\rangle .
\]

Because $v$ is real on $\R$,
\[
\beta_\sigma(x)
=
\Im\bigl(v(x+i\sigma)-v(x)\bigr),
\]
and therefore
\begin{equation}
\|\beta_\sigma\|_\infty
\le
|\sigma|
\|v'\|_{\mathcal{H}_{\rho_1}},
\qquad
\|\beta_\sigma'\|_\infty
\le
|\sigma|
\|v''\|_{\mathcal{H}_{\rho_1}},
\label{eq:imaginary-v-strip}
\end{equation}
for every fixed
\[
\rho_0<\rho_1<\rho_*.
\]

By Lemma~\ref{lem:centered-transport-form},
\begin{align}
\left|
\Re\langle
B_h(v_\sigma)U_\sigma,U_\sigma
\rangle
\right|
&\le
C\|\beta_\sigma\|_\infty
\||D|^{1/2}U_\sigma\|_2^2
\nonumber\\
&\quad
+
C
\bigl(
\|\beta_\sigma\|_\infty
\|\beta_\sigma'\|_\infty
\bigr)^{1/2}
\||D|^{1/2}U_\sigma\|_2
\|U_\sigma\|_2 .
\label{eq:transport-boundary-form}
\end{align}
Using \eqref{eq:imaginary-v-strip} and Young's inequality, it follows
\begin{equation}
\left|
\Re\langle
B_h(v_\sigma)U_\sigma,U_\sigma
\rangle
\right|
\le
C_v|\sigma|
\left(
\||D|^{1/2}U_\sigma\|_2^2
+
\|U_\sigma\|_2^2
\right),
\label{eq:transport-boundary-final}
\end{equation}
where $C_v$ depends only on bounded strip norms of $v'$ and $v''$,
and is independent of $h$.

For the zeroth-order coefficient,
\begin{equation}
\Re\langle w_\sigma U_\sigma,U_\sigma\rangle
\le
\|w\|_{\mathcal{H}_{\rho_1}}\|U_\sigma\|_2^2.
\label{eq:w-boundary-estimate}
\end{equation}

We now look at the moving-boundary energy.
Since
\[
\frac{d}{dt}u_M(t,x+i\rho(t))
=
L_{p,h}u_M(t,x+i\rho(t))
+
i\dot\rho(t)Du_M(t,x+i\rho(t)),
\]
and
\[
\frac{d}{dt}u_M(t,x-i\rho(t))
=
L_{p,h}u_M(t,x-i\rho(t))
-
i\dot\rho(t)Du_M(t,x-i\rho(t)),
\]
the contributions caused by the motion of the two boundary lines sum
to
\begin{equation}
4\dot\rho(t)
\int_{\R}
|\xi|
\sinh(2\rho(t)|\xi|)
|\widehat{u_M}(t,\xi)|^2\,d\xi.
\label{eq:moving-boundary-term}
\end{equation}

Combining
\eqref{eq:principal-horizontal-skew},
\eqref{eq:transport-boundary-final},
\eqref{eq:w-boundary-estimate},
and \eqref{eq:moving-boundary-term}, we obtain
\begin{align}
\frac{d}{dt}E(t)
&\le
4\dot\rho(t)
\int_\R
|\xi|\sinh(2\rho(t)|\xi|)
|\widehat{u_M}(t,\xi)|^2\,d\xi
+C_v\rho(t)Z(t)
+C_1E(t),
\label{eq:boundary-energy-inequality}
\end{align}
where $C_1$ is independent of $M$ and $h$.

Choose
\begin{equation}
\rho(t)=\rho_0e^{-\kappa t},
\qquad t\ge0.
\label{eq:shrinking-strip}
\end{equation}
Then
\[
\dot\rho(t)=-\kappa\rho(t).
\]

Split frequency space into
\[
\Omega_{\rm lo}(t)
:=
\{\xi:2\rho(t)|\xi|<1\},
\qquad
\Omega_{\rm hi}(t)
:=
\{\xi:2\rho(t)|\xi|\ge1\}.
\]
On $\Omega_{\rm hi}(t)$,
\[
\sinh(2\rho|\xi|)
\ge
\tanh(1)\cosh(2\rho|\xi|).
\]
Hence, if $\kappa$ is chosen sufficiently large depending only on
$C_v$,
the negative contribution in
\eqref{eq:moving-boundary-term} absorbs the high-frequency part of
$C_v\rho Z$.

On the low-frequency region,
\[
\rho(t)|\xi|\le\frac12,
\]
and therefore
\begin{align*}
\rho(t)Z_{\rm lo}(t)
&=
2\rho(t)
\int_{\Omega_{\rm lo}(t)}
|\xi|\cosh(2\rho(t)\xi)
|\widehat{u_M}(t,\xi)|^2\,d\xi
\\
&\le
C E(t).
\end{align*}
Thus \eqref{eq:boundary-energy-inequality} reduces to
\begin{equation}
E'(t)\le C_2E(t),
\qquad t\ge0,
\label{eq:boundary-gronwall}
\end{equation}
where $C_2$ is independent of $M$ and $h$.
Gronwall's inequality yields
\begin{equation}
E(t)^{1/2}
\le
e^{C_2t/2}E(0)^{1/2}.
\label{eq:boundary-propagation}
\end{equation}

The boundary norm and the Fourier analytic norm are equivalent:
\[
\|g\|_\rho^2
\le
\|g(\cdot+i\rho)\|_{L^2}^2
+
\|g(\cdot-i\rho)\|_{L^2}^2
\le
2\|g\|_\rho^2.
\]
Consequently
\begin{equation}
\|e^{tL_{p,h}^{(M)}}f\|_{\rho_0e^{-\kappa t}}
\le
\sqrt2\,e^{\Gamma t}\|f\|_{\rho_0},
\qquad t\ge0,
\label{eq:galerkin-group-positive}
\end{equation}
for some $\Gamma\ge0$, uniformly in $M$ and $h$.

Applying the same argument to $-L_{p,h}^{(M)}$ gives
\begin{equation}
\|e^{tL_{p,h}^{(M)}}f\|_{\rho_0e^{-\kappa|t|}}
\le
\sqrt2\,e^{\Gamma|t|}\|f\|_{\rho_0},
\qquad t\in\R.
\label{eq:galerkin-group-two-sided}
\end{equation}
\bwarning{
Indeed, replacing $L_{p,h}$ by $-L_{p,h}$ changes the signs of
$v,w$ but leaves all strip norms entering the estimates unchanged.
}

\medskip
\noindent
\emph{Step 2: removal of the Galerkin cutoff for $h\ne0$.}

For fixed $h\ne0$, the multiplier $Q_h$ is bounded on $L^2$, hence
$L_{p,h}$ is a bounded operator on $L^2$ because $v,w\in L^\infty$.
Moreover
\[
L_{p,h}^{(M)}
=
\Pi_ML_{p,h}\Pi_M
\longrightarrow
L_{p,h}
\]
strongly on $L^2$.
The bounded-generator Duhamel formula therefore gives, for
$f\in\mathscr E$,
\[
e^{tL_{p,h}^{(M)}}f
\longrightarrow
e^{tL_{p,h}}f
\qquad
\text{in }L^2,
\]
uniformly for $t$ in compact intervals.

For fixed $t$, the norm
\[
g\longmapsto
\|g\|_{\rho_0e^{-\kappa|t|}}
\]
is lower semicontinuous with respect to $L^2$ convergence.
Passing to the limit in
\eqref{eq:galerkin-group-two-sided}, it follows
\begin{equation}
\|e^{tL_{p,h}}f\|_{\rho_0e^{-\kappa|t|}}
\le
\sqrt2\,e^{\Gamma|t|}\|f\|_{\rho_0},
\qquad
f\in\mathscr E.
\label{eq:analytic-group-propagation}
\end{equation}

\bwarning{
Since $\mathscr E$ is dense in $\mathcal A_{\rho_0}$ and
$e^{tL_{p,h}}$ is bounded on $L^2$, the standard Cauchy argument
extends the estimate to every $f\in\mathcal A_{\rho_0}$.
}

Next, we use density to extend it to $\mathcal A_{\rho_0}$.  Fix $h\ne0$ and $t\in\R$, write
\[
 U_h(t):=e^{tL_{p,h}},\qquad
 \rho_t:=\rho_0e^{-\kappa|t|},
\]
and let $f\in\mathcal A_{\rho_0}$.  Set $f_N:=\Pi_Nf\in\mathscr E$.
Then $f_N\to f$ in $\mathcal A_{\rho_0}$, and applying
\eqref{eq:analytic-group-propagation} to $f_N-f_M$ yields
\[
 \|U_h(t)(f_N-f_M)\|_{\rho_t}
 \le \sqrt{2} e^{\Gamma|t|}\|f_N-f_M\|_{\rho_0}.
\]
Hence $U_h(t)f_N$ converges in $\mathcal A_{\rho_t}$ to some $v(t)$.
Since $\mathcal A_{\rho_t}\hookrightarrow L^2$, the same sequence
converges to $v(t)$ in $L^2$.  But $U_h(t)$ is bounded on $L^2$ and
$f_N\to f$ in $L^2$, so $U_h(t)f_N\to U_h(t)f$ in $L^2$.
Therefore $v(t)=U_h(t)f$, and passing to the limit:
\[
 \|U_h(t)f\|_{\rho_t}
 \le e^{\Gamma|t|}\|f\|_{\rho_0}.
\]
Thus \eqref{eq:analytic-group-propagation-new} holds for every $f\in\mathcal A_{\rho_0}$.

\medskip
\noindent
\emph{Step 3: the differential generator case $h=0$.}

{
Set
\[
L:=L_{p,0},
\qquad
L_N:=\Pi_NL\Pi_N.
\]
Since $v$ is real on $\R$, the operator
\[
P_{p,0}+\frac12(vD+Dv)
\]
is skew-adjoint up to the standard closure, while $w$ is bounded.
Proposition~\ref{prop:sobolev-group-bound} therefore yields a
$C_0$-group
$
T(t):=e^{tL}
$
on $L^2$. 

On $\operatorname{Ran}\Pi_N$, $L_N$ is bounded.  Moreover,
compression preserves the skew-adjoint part, and hence
\[
\|e^{tL_N}\|_{L^2\to L^2}
\le
e^{\|\Re w\|_\infty|t|}
\]
uniformly in $N$.
The boundary calculation of Step~1 applies verbatim to $L_N$, with
constants independent of $N$, and therefore
\begin{equation}
\|e^{tL_N}\Pi_Nf\|_{\rho_0e^{-\kappa|t|}}
\le
\sqrt2\,e^{\Gamma|t|}\|f\|_{\rho_0}.
\label{eq:galerkin-h0-analytic}
\end{equation}

Fix $f\in\mathcal A_{\rho_0}$ and $t\in\R$.  Since
\[
\mathcal A_{\rho_0}\subset H^p(\R),
\]
the orbit
\[
K_t:=\{T(s)f:\ |s|\le|t|\}
\]
is compact in $H^p(\R)$.
Define
\[
v_N(s):=e^{sL_N}\Pi_Nf,
\qquad
w_N(s):=\Pi_NT(s)f.
\]
Then
\[
v_N'(s)=L_Nv_N(s)
\]
and
\[
w_N'(s)
=
L_Nw_N(s)
+
\Pi_NL(I-\Pi_N)T(s)f.
\]
Hence
\[
v_N(t)-w_N(t)
=
-\int_0^t
e^{(t-s)L_N}
\Pi_NL(I-\Pi_N)T(s)f\,ds.
\]
Since
\[
L:H^p\longrightarrow L^2
\]
is bounded and $\Pi_N\to I$ strongly on $H^p$, uniformly on the
compact set $K_t$,
\[
\sup_{|s|\le|t|}
\|(I-\Pi_N)T(s)f\|_{H^p}
\longrightarrow0.
\]
Using the uniform $L^2$ bound for $e^{\tau L_N}$ therefore gives
\[
e^{tL_N}\Pi_Nf
\longrightarrow
e^{tL}f
\qquad
\text{in }L^2.
\]
Finally, lower semicontinuity of the analytic norm in
\eqref{eq:galerkin-h0-analytic} yields
\[
\|e^{tL}f\|_{\rho_0e^{-\kappa|t|}}
\le
\sqrt2\,e^{\Gamma|t|}\|f\|_{\rho_0}.
\]

Thus \eqref{eq:analytic-group-propagation-new} holds. 
}

\bwarning{
\smallskip
\noindent\emph{(a) Galerkin generators and a uniform analytic estimate.}
Set
\[
 \Pi_N:=\mathbf 1_{\{|\xi|\le N\}}(D),\qquad
 L_N:=L_{p,0}^{(N)}:=\Pi_NL_{p,0}\Pi_N.
\]
On $\operatorname{Ran}\Pi_N$ every derivative is bounded:
\[
 \|D^k\Pi_Ng\|_2\le N^k\|g\|_2.
\]
Since $v,w\in L^\infty$, it follows that $L_N$ is a bounded operator on
$L^2$.  The norm of $L_N$ may grow with $N$; what matters is that its
\emph{energy bound} is uniform in $N$.  Indeed, compression preserves
the skew-adjoint parts: if $K^*=-K$, then
\[
 (\Pi_NK\Pi_N)^*=\Pi_NK^*\Pi_N=-\Pi_NK\Pi_N.
\]
For the zeroth-order term,
\[
 \operatorname{Re}\langle \Pi_Nw\Pi_Ng,g\rangle
 =\operatorname{Re}\langle w\Pi_Ng,\Pi_Ng\rangle
 \le \|\operatorname{Re}w\|_\infty\|g\|_2^2.
\]
Hence
\[
 \|e^{tL_N}\|_{L^2\to L^2}\le e^{\omega|t|}
\]
with $\omega=\|\operatorname{Re}w\|_\infty$ independent of $N$.

Let
\[
 u_N(t):=e^{tL_N}\Pi_Nf,\qquad U_N:=W_{\rho(t)}u_N.
\]
Because $L_N$ leaves $\operatorname{Ran}\Pi_N$ invariant,
$u_N(t)=\Pi_Nu_N(t)$ and $U_N=\Pi_NU_N$.  Moreover $\Pi_N$ commutes
with $D$, with $P_{p,0}$, and with every analytic weight $W_\rho$.
Consequently the projections disappear from the weighted quadratic
forms.  For example, writing $B_0=\frac12(vD+Dv)$,
\[
 \begin{aligned}
 \langle W_\rho\Pi_NB_0\Pi_Nu_N,U_N\rangle
 &=\langle \Pi_NW_\rho B_0u_N,U_N\rangle\\
 &=\langle W_\rho B_0u_N,\Pi_NU_N\rangle\\
 &=\langle W_\rho B_0u_N,U_N\rangle.
 \end{aligned}
\]
The same observation applies to the principal and zeroth-order terms.
Therefore the entire Step~1 energy calculation applies to $L_N$ with
constants independent of $N$, and yields
\begin{equation}
 \|e^{tL_N}\Pi_Nf\|_{\rho_0e^{-\kappa|t|}}
 \le e^{\Gamma|t|}\|f\|_{\rho_0}.
\label{eq:analytic-group-bound-aux}
\end{equation}

\smallskip
\noindent\emph{(b) Convergence of the Galerkin group at a fixed time.}
Because $v,w\in\mathfrak W_R$, their finite derivatives are bounded. 
Proposition~\ref{prop:sobolev-group-bound} shows that
\(L_{p,0}\), with domain
\[
D(L_{p,0})=H^p(\R)\subset L^2(\R),
\]
generates a \(C_0\)-group
\[
T(t)=e^{tL_{p,0}}
\]
on \(L^2(\R)\). Moreover, the Sobolev spaces \(H^j(\R)\) in the
range under consideration are invariant under \(T(t)\), and
\[
\|T(t)\|_{H^j\to H^j}
\le
e^{\Omega_j|t|}.
\]
In particular, given
\[
f\in\mathcal A_{\rho_0}\subset H^p(\R)=D(L_{p,0}),
\]
the orbit \(t\mapsto T(t)f\) is continuously differentiable as an
\(L^2(\R)\)-valued map and satisfies
\[
\frac{d}{dt}T(t)f
=
L_{p,0}T(t)f
=
T(t)L_{p,0}f.
\]

%
Moreover, the orbit segment
\[
 K_t:=\{T(s)f:|s|\le |t|\}
\]
is compact in $H^p$ since  $s\mapsto T(s)f$ is continuous there.

Define
\[
 v_N(s):=e^{sL_N}\Pi_Nf,
 \qquad
 w_N(s):=\Pi_NT(s)f.
\]
Both have the same initial value $\Pi_Nf$.  Moreover
\[
 v_N'(s)=L_Nv_N(s),
\]
whereas
\[
 w_N'(s)
 =\Pi_NL_{p,0}T(s)f
 =L_Nw_N(s)+\Pi_NL_{p,0}(I-\Pi_N)T(s)f.
\]
Thus $d_N:=v_N-w_N$ satisfies
\[
 d_N'(s)
 =L_Nd_N(s)-\Pi_NL_{p,0}(I-\Pi_N)T(s)f,
 \qquad d_N(0)=0.
\]
Variation of constants, evaluated at the fixed terminal time $t$, gives
\begin{equation}
 v_N(t)-w_N(t)
 =-\int_0^t e^{(t-s)L_N}
 \Pi_NL_{p,0}(I-\Pi_N)T(s)f\,ds.
\label{eq:analytic-projected-equation}
\end{equation}
Since $L_{p,0}:H^p\to L^2$ is bounded,
\[
 \|\Pi_NL_{p,0}(I-\Pi_N)T(s)f\|_2
 \le C\|(I-\Pi_N)T(s)f\|_{H^p}.
\]
The projections satisfy $\Pi_N\to I$ strongly on $H^p$ and are
uniformly bounded there.  Hence the convergence is uniform on the
compact orbit segment $K_t$:
\begin{equation}
 \sup_{|s|\le |t|}
 \|(I-\Pi_N)T(s)f\|_{H^p}\longrightarrow0.
\label{eq:analytic-projected-derivative}
\end{equation}
Using \eqref{eq:analytic-projected-equation}, \eqref{eq:analytic-projected-derivative}, and the $N$-uniform $L^2$
bound for $e^{\tau L_N}$ on $|\tau|\le |t|$, we obtain
\[
 \|v_N(t)-w_N(t)\|_2\longrightarrow0.
\]
At the same fixed time,
\[
 w_N(t)=\Pi_NT(t)f\longrightarrow T(t)f
 \qquad\text{in }L^2,
\]
by strong convergence $\Pi_N\to I$ on $L^2$.  Therefore
\begin{equation}
 e^{tL_N}\Pi_Nf\longrightarrow e^{tL_{p,0}}f
 \quad\text{in }L^2.
\label{eq:analytic-projected-integral}
\end{equation}
\warning{This pointwise-in-$t$ convergence is exactly what is needed below; no
uniform convergence in the terminal time is required.  The only
uniformity used is \eqref{eq:analytic-projected-derivative} in the integration variable $s$ on
the fixed interval $[-|t|,|t|]$.  Notice also that only the mapping
$L_{p,0}:H^p\to L^2$ is used; no $H^{2p}$ estimate and no
Richardson-depth-dependent regularity is needed here.}

\smallskip
\noindent\emph{(c) Passing the analytic estimate to the limit.}
Fix $t$ and set $\rho_t:=\rho_0e^{-\kappa|t|}$.  By
\eqref{eq:analytic-projected-integral}, $v_N(t)\to T(t)f$ in $L^2$, whereas
\eqref{eq:analytic-group-bound-aux} gives the uniform bound
\[
 \|v_N(t)\|_{\rho_t}
 \le e^{\Gamma|t|}\|f\|_{\rho_0}.
\]
The weighted norm $g\mapsto\|W_{\rho_t}g\|_2$ is lower
semicontinuous with respect to $L^2$ convergence. 
\warning{ Equivalently, after
passing to an a.e.-convergent Fourier subsequence, this follows from
Fatou's lemma.}  Hence
\[
 \begin{aligned}
 \|T(t)f\|_{\rho_t}
 &\le \liminf_{N\to\infty}\|v_N(t)\|_{\rho_t}\\
 &\le e^{\Gamma|t|}\|f\|_{\rho_0}.
 \end{aligned}
\]
This is precisely \eqref{eq:analytic-group-propagation-new} for the generator $h=0$, with $C_0=\sqrt2$.
}

\end{proof}

\subsection{Uniform analytic propagation in the parabolic case}\label{subsec:analytic-parabolic-propagation}

For the parabolic model summarized in Subsection~\ref{subsec:model-summary},
\[
 A=DaD+\frac12(\varphi D+D\varphi)+p_0,
 \qquad
 \varphi=b-a',\qquad
 p_0=c+\frac12(a''-b'),
\]
and
\begin{equation}
 A_h=\delta_h a\delta_h
 +\frac12(\varphi\delta_h+\delta_h\varphi)+p_0.
\label{eq:parabolic-analytic-generator}
\end{equation}

{
\begin{proposition}[Uniform analytic propagation in the parabolic case]
\label{prop:analytic-parabolic}

Assume that, for some $\rho_*>0$,
\[
a,\varphi,p_0\in \mathcal{H}_{\rho_*},
\qquad
a(\bar z)=\overline{a(z)},
\qquad
\varphi(\bar z)=\overline{\varphi(z)},
\]
and
\[
a(x)\ge a_0>0,
\qquad x\in\R.
\]
Then, there exists $0<\rho_1<\rho_*$ such that
\[
\Re a(z)\ge \frac{a_0}{2},
\qquad z\in \mathscr S_{\rho_1},
\]
and, for every $0<\rho_0<\rho_1$ there exist $C_0,\Gamma\ge0$,
independent of $h$, such that
\begin{equation}
\|e^{tA_h}f\|_{\rho_0}
\le
C_0e^{\Gamma t}\|f\|_{\rho_0},
\qquad t\ge0.
\label{eq:analytic-parabolic-propagation-new}
\end{equation}
The same estimate holds for $A=A_0$.
\end{proposition}
}

\gwarning{
\begin{proposition}[Uniform analytic propagation in the parabolic case]
\label{prop:analytic-parabolic}
Let
\[
A
=
DaD+\frac12(\varphi D+D\varphi)+p_0,
\qquad
\varphi=b-a',
\qquad
p_0=c+\frac12(a''-b'),
\]
and
\[
A_h
=
\delta_h a\delta_h
+
\frac12(\varphi\delta_h+\delta_h\varphi)
+p_0.
\]
Assume that, for some $R>0$,
\[
a,\varphi,p_0\in H^\infty(S_R),
\qquad
a(\bar z)=\overline{a(z)},
\]
and
\[
a(x)\ge a_0>0,
\qquad x\in\R.
\]
Then there exists $0<R_1<R$ such that
\[
\Re a(z)\ge \frac{a_0}{2},
\qquad z\in S_{R_1}.
\]
For every $0<\rho_0<R_1$ there exist constants
$C_0\ge1$ and $\Gamma\ge0$, independent of $h$, such that
\begin{equation}
\|e^{tA_h}f\|_{\rho_0}
\le
C_0e^{\Gamma t}\|f\|_{\rho_0},
\qquad
t\ge0.
\label{eq:analytic-parabolic-propagation}
\end{equation}
The same estimate holds for $A=A_0$.
\end{proposition}
}

\begin{proof}
Since $a$ is bounded and holomorphic on $\mathscr S_{\rho_*}$, Cauchy's estimate on
a slightly smaller strip gives
\[
\sup_{x\in\R}
|a(x+iy)-a(x)|
\le
C|y|,
\]
uniformly for sufficiently small $|y|$.
Since
\[
a(x)\ge a_0>0
\qquad (x\in\R),
\]
there exist $0<\rho_1<\rho_*$ such that
\begin{equation}
\Re a(z)\ge a_*:=\frac{a_0}{2},
\qquad
z\in \mathscr S_{\rho_1}.
\label{eq:complex-strip-ellipticity}
\end{equation}

As in Proposition~\ref{prop:analytic-group}, let
\[
\mathscr E
=
\bigcup_{N\ge1}\operatorname{Ran}\Pi_N.
\]
Fix first $h\ne0$, $f\in\mathscr E$, and introduce
\[
A_h^{(M)}
=
\Pi_MA_h\Pi_M,
\qquad
u_M(t)
=
e^{tA_h^{(M)}}f.
\]
The function $u_M(t)$ remains band limited, so its boundary traces are
entire. \bwarning{and all subsequent energy calculations are justified
classically.}

Let $0<\rho_0<\rho_1$ and fix $\sigma$ with $|\sigma|\le\rho_0$ and put
\[
U_\sigma(t,x)
:=
u_M(t,x+i\sigma).
\]
Since $\delta_h$ is a constant-coefficient operator,
\[
(\delta_hu_M)(x+i\sigma)
=
\delta_hU_\sigma(x).
\]
Furthermore
\[
\delta_h^*=-\delta_h
\]
on $L^2(\R+i\sigma)$.

Set
\[
a_\sigma(x):=a(x+i\sigma),
\qquad
\varphi_\sigma(x):=\varphi(x+i\sigma),
\qquad
p_{0,\sigma}(x):=p_0(x+i\sigma).
\]
The principal quadratic form satisfies
\begin{align}
\Re\langle
\delta_h(a_\sigma\delta_hU_\sigma),
U_\sigma
\rangle
&=
-\Re\langle
a_\sigma\delta_hU_\sigma,
\delta_hU_\sigma
\rangle
\nonumber\\
&=
-\int_\R
\Re a_\sigma(x)
|\delta_hU_\sigma(x)|^2\,dx
\nonumber\\
&\le
-a_*
\|\delta_hU_\sigma\|_2^2,
\label{eq:parabolic-strip-principal}
\end{align}
by \eqref{eq:complex-strip-ellipticity}.

For the symmetrized first-order term,
\begin{align*}
&
\left|
\left\langle
\frac12
(\varphi_\sigma\delta_h+\delta_h\varphi_\sigma)
U_\sigma,
U_\sigma
\right\rangle
\right|
\\
&\qquad
\le
\frac12
\left|
\langle
\varphi_\sigma\delta_hU_\sigma,
U_\sigma
\rangle
\right|
+
\frac12
\left|
\langle
\delta_h(\varphi_\sigma U_\sigma),
U_\sigma
\rangle
\right|.
\end{align*}
Using $\delta_h^*=-\delta_h$ in the second term yields
\[
\left|
\left\langle
\frac12
(\varphi_\sigma\delta_h+\delta_h\varphi_\sigma)
U_\sigma,
U_\sigma
\right\rangle
\right|
\le
\|\varphi\|_{\mathcal{H}_{\rho_1}}
\|\delta_hU_\sigma\|_2
\|U_\sigma\|_2,
\]
whence
\begin{equation}
\left|
\Re
\left\langle
\frac12
(\varphi_\sigma\delta_h+\delta_h\varphi_\sigma)
U_\sigma,
U_\sigma
\right\rangle
\right|
\le
\frac{a_*}{2}
\|\delta_hU_\sigma\|_2^2
+
C_\varphi\|U_\sigma\|_2^2,
\label{eq:parabolic-strip-first-order}
\end{equation}
where
$
C_\varphi
=
{\|\varphi\|_{\mathcal{H}_{\rho_1}}^2}/{(2a_*)}.
$ 

Finally,
\begin{equation}
\Re\langle
p_{0,\sigma}U_\sigma,U_\sigma
\rangle
\le
\|p_0\|_{\mathcal{H}_{\rho_1}}
\|U_\sigma\|_2^2.
\label{eq:parabolic-strip-zero}
\end{equation}

Combining
\eqref{eq:parabolic-strip-principal},
\eqref{eq:parabolic-strip-first-order},
and \eqref{eq:parabolic-strip-zero}, we obtain
\begin{equation}
\frac12\frac{d}{dt}
\|U_\sigma(t)\|_2^2
+
\frac{a_*}{2}
\|\delta_hU_\sigma(t)\|_2^2
\le
\Gamma
\|U_\sigma(t)\|_2^2,
\label{eq:parabolic-boundary-energy}
\end{equation}
where
\[
\Gamma
:=
C_\varphi+\|p_0\|_{\mathcal{H}_{\rho_1}},
\]
 is independent of $h$, $M$, and $\sigma$.

Dropping the nonnegative dissipative term and applying Gronwall yields
\begin{equation}
\|U_\sigma(t)\|_{L^2}
\le
e^{\Gamma t}
\|U_\sigma(0)\|_{L^2},
\qquad
t\ge0.
\label{eq:parabolic-boundary-gronwall}
\end{equation}
Applying this with
$
\sigma=\rho_0
\ \text{and}\ 
\sigma=-\rho_0,
$
we obtain, 
\[
\|e^{tA_h^{(M)}}f\|_{\partial,\rho_0}
\le
e^{\Gamma t}
\|f\|_{\partial,\rho_0}.
\]
By the norm equivalence
\[
\|g\|_{\rho_0}
\le
\|g\|_{\partial,\rho_0}
\le
\sqrt2\,\|g\|_{\rho_0},
\]
we obtain
\begin{equation}
\|e^{tA_h^{(M)}}f\|_{\rho_0}
\le
\sqrt2\,e^{\Gamma t}\|f\|_{\rho_0},
\qquad
t\ge0,
\label{eq:parabolic-galerkin-analytic}
\end{equation}
uniformly in $M$ and $h$.

For fixed $h\ne0$, $A_h$ is bounded on $L^2$ and
\[
A_h^{(M)}
=
\Pi_MA_h\Pi_M
\longrightarrow A_h
\]
strongly on $L^2$ as $M\to \infty$.
The bounded-generator Duhamel formula gives
\[
e^{tA_h^{(M)}}f
\longrightarrow
e^{tA_h}f
\qquad\text{in }L^2
\]
for every $f\in\mathscr E$ and fixed $t\ge0$.
Lower semicontinuity of the analytic norm in
\eqref{eq:parabolic-galerkin-analytic} therefore yields
\[
\|e^{tA_h}f\|_{\rho_0}
\le
\sqrt2\,e^{\Gamma t}\|f\|_{\rho_0},
\qquad
f\in\mathscr E.
\]
Density of $\mathscr E$ in $\mathcal A_{\rho_0}$ extends this estimate
to every $f\in\mathcal A_{\rho_0}$.

It remains to treat the case $h=0$.  Set
\[
A^{(N)}
:=
\Pi_NA\Pi_N.
\]
On the real line the principal quadratic form is dissipative:
\[
\Re\langle DaDu,u\rangle
=
-\int_\R a(x)|Du(x)|^2\,dx
\le0.
\]
Since $\varphi=b-a'$ is real on $\R$, then
\[
\frac12(\varphi D+D\varphi)
\]
is skew-symmetric. \bwarning{; if one only assumes bounded complex $\varphi$, the
same $L^2$ growth estimate follows by relative boundedness of the
first-order term.  In the present applications $a,b$ are real on
$\R$, and therefore $\varphi=b-a'$ is real.}

Thus
\[
\|e^{tA^{(N)}}\|_{L^2\to L^2}
\le
e^{\omega t}
\]
with $\omega$ independent of $N$.
The boundary calculation above applies to $A^{(N)}$ with the same
constants and gives
\[
\|e^{tA^{(N)}}\Pi_Nf\|_{\rho_0}
\le
\sqrt2\,e^{\Gamma t}\|f\|_{\rho_0}.
\]

The uniformly elliptic operator $A$ generates the required analytic
semigroup on $L^2$, with domain $H^2(\R)$. \bwarning{ under the present bounded
coefficient hypotheses.}  For
$f\in\mathcal A_{\rho_0}\subset H^2(\R)$, the orbit
\[
\{e^{sA}f:0\le s\le t\}
\]
is compact in $H^2(\R)$ for each fixed $t>0$.  Since
\[
A:H^2(\R)\longrightarrow L^2(\R)
\]
is bounded and $\Pi_N\to I$ strongly on $H^2$, the same
variation-of-constants argument as in
Proposition~\ref{prop:analytic-group} gives
\[
e^{tA^{(N)}}\Pi_Nf
\longrightarrow
e^{tA}f
\qquad
\text{in }L^2.
\]
Lower semicontinuity therefore yields
\[
\|e^{tA}f\|_{\rho_0}
\le
\sqrt2\,e^{\Gamma t}\|f\|_{\rho_0}.
\]
This concludes the proof of \eqref{eq:analytic-parabolic-propagation-new}, with
$C_0=\sqrt2$.
\end{proof}

\bwarning{
\begin{proposition}[Uniform analytic propagation in the parabolic case]
\label{prop:analytic-parabolic}
Assume $a$ is real, $a(x)\ge a_0>0$, and
$a,\varphi,p_0\in\mathfrak W_R$ for some $R>0$.  Choose
$0<\rho_0<R$ sufficiently small that
\begin{equation}
 \rho_0M_1(a;\rho_0)\le \frac{a_0}{2}.
\label{eq:parabolic-dissipation}
\end{equation}
Then there exist $\kappa,\Gamma\ge0$, independent of $h$, such that
\begin{equation}
 \|e^{tA_h}f\|_{\rho_0e^{-\kappa t}}
 \le e^{\Gamma t}\|f\|_{\rho_0},
 \qquad t\ge0.
\label{eq:analytic-parabolic-propagation}
\end{equation}
The same estimate holds for $A=A_0$.
\end{proposition}

\begin{proof}
We use the same dense analytic core
$\mathscr E=\bigcup_{N\ge1}\operatorname{Ran}\Pi_N$ as in
Proposition~\ref{prop:analytic-group}.  Fix first $h\ne0$ and
$f\in\mathscr E$.  As above, 
\warning{rather than assuming that the
uncompressed flow $e^{tA_h}$ already preserves analytic regularity,}
 we
introduce
\[
 A_h^{(M)}:=\Pi_MA_h\Pi_M,
 \qquad
 u_M(t):=e^{tA_h^{(M)}}f,
\]
with $M$ large enough that $f\in\operatorname{Ran}\Pi_M$.  The
space $\operatorname{Ran}\Pi_M$ is invariant under $A_h^{(M)}$, so
$u_M(t)$ remains band limited.  By the Fourier-support estimate proved
in Step~1 of Proposition~\ref{prop:analytic-group},
\[
 \|u_M(t)\|_{\rho,s}
 \le e^{\rho M}\langle M\rangle^s\|u_M(t)\|_2,
 \qquad \rho,s\ge0.
\]
Thus $u_M(t)$ and $U_M=W_{\rho(t)}u_M(t)$ have all the analytic--Sobolev
regularity required for differentiating the energy.  Since
$\Pi_M$ commutes with $\delta_h$ and $W_\rho$, and since
$U_M\in\operatorname{Ran}\Pi_M$, the projections disappear from the
quadratic forms exactly as in Step~1.  We may therefore write $u,U$ in
place of $u_M,U_M$ and set $V=\delta_hU$.  Since $W_\rho$ commutes
with $\delta_h$ and $\delta_h^*=-\delta_h$,
\begin{align*}
 \operatorname{Re}\langle W_\rho\delta_h a\delta_hu,U\rangle
 &=-\langle aV,V\rangle
   -\operatorname{Re}\langle[W_\rho,a]\delta_hu,V\rangle\\
 &\le
 -\bigl(a_0-\rho M_1(a;\rho)\bigr)\|V\|_2^2
 \le -\frac{a_0}{2}\|V\|_2^2,
\end{align*}
where \eqref{eq:analytic-commutator} and \eqref{eq:parabolic-dissipation} were used.  The skew first-order term satisfies
the analogue of \eqref{eq:analytic-energy-identity},
\[
 \left|\operatorname{Re}\left\langle
 W_\rho\frac12(\varphi\delta_h+\delta_h\varphi)u,U
 \right\rangle\right|
 \le C\rho M_1(\varphi;\rho)Z
    +C\rho M_2(\varphi;\rho)\|U\|_2^2,
\]
with $Z=\||D|^{1/2}U\|_2^2$, and multiplication by $p_0$ contributes
at most $M_0(p_0;\rho)\|U\|_2^2$.  Consequently
\begin{align*}
 \frac12\frac d{dt}\|U\|_2^2
 &+\frac{a_0}{2}\|\delta_hU\|_2^2\\
 &\le
 \bigl(\dot\rho+C\rho M_1(\varphi;\rho)\bigr)Z
 +\bigl(C\rho M_2(\varphi;\rho)+M_0(p_0;\rho)\bigr)\|U\|_2^2.
\end{align*}
Choose $\rho(t)=\rho_0e^{-\kappa t}$ with
$\kappa\ge2CM_1(\varphi;\rho_0)$.  Then
$\dot\rho(t)=-\kappa\rho(t)$ and, since $\rho(t)\le\rho_0$ and
$M_1(\varphi;\rho)$ is increasing in $\rho$,
\[
 \dot\rho+C\rho M_1(\varphi;\rho)
 =\rho\bigl(-\kappa+CM_1(\varphi;\rho)\bigr)
 \le -\frac{\kappa}{2}\rho\le0.
\]
Thus the $Z$ term can be discarded.  We may also drop the \warning{favorable}
term $\frac{a_0}{2}\|\delta_hU\|_2^2$ from the left-hand side.  Since
$\rho(t)\le\rho_0$ and $M_j(g;\rho)$ is increasing in $\rho$,
\[
 C\rho(t)M_2(\varphi;\rho(t))+M_0(p_0;\rho(t))
 \le
 C\rho_0M_2(\varphi;\rho_0)+M_0(p_0;\rho_0)=:\Gamma.
\]
Hence
\[
 \frac12\frac d{dt}\|U(t)\|_2^2
 \le \Gamma\|U(t)\|_2^2.
\]
Gronwall's inequality gives
\[
 \|U(t)\|_2\le e^{\Gamma t}\|U(0)\|_2.
\]
Recalling that $U(t)=W_{\rho(t)}e^{tA_h^{(M)}}f$ and
$U(0)=W_{\rho_0}f$, we obtain;
\begin{equation}
 \|e^{tA_h^{(M)}}f\|_{\rho_0e^{-\kappa t}}
 \le e^{\Gamma t}\|f\|_{\rho_0},
 \qquad t\ge0.
\label{eq:analytic-parabolic-aux}
\end{equation}
The constants $\kappa$ and $\Gamma$ depend only on $\rho_0$ and the
coefficient moments appearing above; in particular, they are
independent of the Galerkin cutoff $M$ (and of $h$).
For fixed $h\ne0$, $A_h$ is bounded on $L^2$ and
$A_h^{(M)}=\Pi_MA_h\Pi_M\to A_h$ strongly on $L^2$; hence the
bounded-generator Duhamel formula gives
\[
 e^{tA_h^{(M)}}f\longrightarrow e^{tA_h}f
 \qquad\text{in }L^2
\]
for each fixed $t\ge0$ and $f\in\mathscr E$.  Lower
semicontinuity of the weighted norm therefore proves \eqref{eq:analytic-parabolic-propagation}
on $\mathscr E$.  Finally, if $f\in\mathcal A_{\rho_0}$, then
$f_N:=\Pi_Nf\in\mathscr E$ and
$f_N\to f$ in $\mathcal A_{\rho_0}$.  Applying \eqref{eq:analytic-parabolic-propagation}
to $f_N-f_M$ and identifying the resulting analytic-norm limit with
$e^{tA_h}f$ through $L^2$ convergence extends the estimate to all
$f\in\mathcal A_{\rho_0}$.

For the operator $A=A_0$, use the same fixed-time
Fourier--Galerkin passage as in Step~2 of
Proposition~\ref{prop:analytic-group}, with
$A^{(N)}=\Pi_NA\Pi_N$.  Compression preserves the negative elliptic
quadratic form and the skew first-order part, so the analytic estimate
is uniform in $N$.  Moreover $A:H^2\to L^2$ is bounded and, for
$f\in\mathcal A_{\rho_0}\subset H^2$, the semigroup orbit is
continuous in $H^2$.  For each fixed $t\ge0$, compactness of the orbit
segment $\{e^{sA}f:0\le s\le t\}$ in $H^2$ gives
\[
 \sup_{0\le s\le t}
 \|(I-\Pi_N)e^{sA}f\|_{H^2}\longrightarrow0.
\]
The corresponding variation-of-constants formula therefore yields
\[
 e^{tA^{(N)}}\Pi_Nf\longrightarrow e^{tA}f
 \qquad\text{in }L^2
\]
at that fixed time.  Lower semicontinuity of the weighted norm then
gives \eqref{eq:analytic-parabolic-propagation} for $A=A_0$. 
\warning{Notice that no uniform lower bound on the symbol
$|\sin(h\xi)|/|h|$ is needed; the parabolic quadratic form is used
only through its sign, while the analytic-radius term absorbs the
half-derivative quantity $Z$.}
\end{proof}
}

\subsection{Analytic defect bounds}\label{subsec:analytic-defects}

The centered difference has the norm-convergent strip expansion
\begin{equation}
 \delta_h
 =D+\frac{h^2}{3!}D^3+\frac{h^4}{5!}D^5+\cdots
\label{eq:analytic-shift-bound}
\end{equation}
from $\mathcal A_\rho$ to $\mathcal A_{\rho'}$ whenever
$|h|<\rho-\rho'$.  Indeed, 
for every integer $q\ge0$, by Cauchy estimate \eqref{eq:coefficient-cauchy}: 
\begin{equation}
 \|D^qf\|_{\sigma'}
 \le \frac{q!}{(\sigma-\sigma')^q}\|f\|_{\sigma},
 \qquad 0<\sigma'<\sigma\le\rho_0,
\label{eq:analytic-derivative-block}
\end{equation}
whence
\[
 \left\|
 \frac{h^{2r}}{(2r+1)!}D^{2r+1}f
 \right\|_{\rho'}
 \le
 \frac1{\rho-\rho'}
 \left(\frac{|h|}{\rho-\rho'}\right)^{2r}\|f\|_\rho.
\]
Thus the finite generator expansions of Section~\ref{sec:centered-differences} and Subsection~\ref{subsec:model-summary} are, on
analytic spaces, truncations of genuinely convergent operator
series.

\begin{lemma}[Quantitative defect estimate]
\label{lem:analytic-defects}
For each of the three centered-difference models, let $p=2$ for the
parabolic and Schr\"odinger cases and $p=3$ for the cubic case.  Under
the analytic coefficient hypotheses above, there exist $A_*,B_*>0$
such that, for $0<\rho'<\rho\le\rho_0$,
\begin{equation}
 \|R_rf\|_{\rho'}
 \le
 A_*B_*^r
 \frac{\Gamma(2r+p+1)}{(\rho-\rho')^{2r+p}}
 \|f\|_\rho,
 \qquad r\ge1.
\label{eq:analytic-defect-bound}
\end{equation}
Moreover, if
\[
 A_h=A+\sum_{r=1}^{m-1}h^{2r}R_r+h^{2m}Q_m(h),
\]
then there exist model-dependent constants $\theta_*,A_1,B_1>0$,
independent of $m$, $h$, $\rho$, and $\rho'$, such that
\[
 \|Q_m(h)f\|_{\rho'}
 \le A_1B_1^m(\rho-\rho')^{-(2m+p)}\|f\|_\rho
\]
whenever $|h|\le\theta_*(\rho-\rho')$. 
\warning{ In particular this is
stronger than the estimate obtained from \eqref{eq:analytic-defect-bound} by setting
$r=m$.}
\end{lemma}

\begin{proof}
Set
\[
 \Delta:=\rho-\rho'>0.
\]
We separate the proof into three steps.

\medskip
\bwarning{
\noindent\emph{Step 1: an elementary analytic-strip estimates.}
For every integer $q\ge0$, 
\begin{equation}
 \|D^qf\|_{\sigma'}
 \le \frac{q!}{(\sigma-\sigma')^q}\|f\|_{\sigma},
 \qquad 0<\sigma'<\sigma\le\rho_0.
\label{eq:analytic-derivative-block}
\end{equation}

We shall also use that multiplication by an analytic Wiener coefficient
costs no loss of analytic radius.  
If $g\in\mathcal{H}_{\rho_*}$ and
$0\le\sigma\le\rho_0<\rho_*$, then
\begin{equation}
 \|gf\|_{\sigma}
 \le \sqrt{2}\|g\|_{\mathcal{H}_{\rho_*}}\|f\|_{\sigma}.
\label{eq:analytic-multiplier-block}
\end{equation}

Indeed,
\[
 e^{\sigma|\xi|}
 \le e^{\sigma|\eta|}e^{\sigma|\xi-\eta|}
 \le e^{R|\eta|}e^{\sigma|\xi-\eta|},
\]
and Young's inequality for convolution with the finite measure
$e^{R|\eta|}\,d|\widehat g|(\eta)$ proves \eqref{eq:analytic-multiplier-block}.
Thus derivative blocks consume strip width according to
\eqref{eq:analytic-derivative-block}, whereas multiplication by the coefficients does not.
}

\medskip
\noindent\emph{Step 1: structure and combinatorics of the coefficient $R_r$.}
Recall that
\begin{equation}
 \delta_h
 =\sum_{q=0}^{\infty}\frac{h^{2q}}{(2q+1)!}D^{2q+1}.
\label{eq:analytic-centered-series}
\end{equation}
Consider one fixed monomial occurring in the centered generator $A_h$,
and let $k$ be the number of occurrences of $\delta_h$ in that
monomial.  By the structure of the three models,
\begin{equation}
 k\le p.
\label{eq:analytic-word-structure}
\end{equation}
\warning{Indeed, the highest-order monomials are
$\delta_h a\delta_h$ in the parabolic case, $i\delta_h^2$ in the
Schr\"odinger case, and $\delta_h^3$ in the cubic case; the lower-order
terms contain only one or zero centered-difference factors.}

To obtain a contribution of order $h^{2r}$ from a monomial containing
$k$ centered differences, choose integers
\[
 q_1,\dots,q_k\ge0,
 \qquad q_1+\cdots+q_k=r.
\]
The corresponding derivative factors are
$D^{2q_1+1},\dots,D^{2q_k+1}$, and hence their total differential order
is
\begin{equation}
 d=\sum_{j=1}^k(2q_j+1)=2r+k\le2r+p.
\label{eq:analytic-word-order}
\end{equation}
Every such contribution can therefore be written as a noncommutative word
\begin{equation}
 \mathcal W
 =M_{g_s}D^{d_s}M_{g_{s-1}}D^{d_{s-1}}\cdots
  M_{g_1}D^{d_1}M_{g_0},
 \qquad d_1+\cdots+d_s=d\le2r+p,
\label{eq:analytic-word-form}
\end{equation}
where the $g_j$ belong to the finite collection of analytic
coefficients appearing in the model. 
\warning{ We deliberately do not commute
the derivatives through the multipliers: this avoids an unnecessary
Leibniz expansion and keeps the order count \eqref{eq:analytic-word-order} exact.}

For fixed $k$, the number of choices of
$q_1+\cdots+q_k=r$ is the weak-composition number
$\binom{r+k-1}{k-1}$.  Since $k\le p$ and the generator contains only
finitely many structural monomials, there exists a model-dependent
constant $C_0>1$ such that the total number of contributions to $R_r$
is bounded by $C_0^r$.
\warning{In fact the growth is only polynomial in $r$
because $p$ is fixed, but the rough geometric bound is convenient.}  
The Taylor denominators
$\prod_j(2q_j+1)!^{-1}$ in \eqref{eq:analytic-centered-series} are at most one and hence
can only improve the estimates below.

\medskip
\noindent\emph{Step 2: estimate of a single word and proof of
\eqref{eq:analytic-defect-bound}.}
Fix a word \eqref{eq:analytic-word-form} and put
$d=d_1+\cdots+d_s$.  Allocate the available strip width proportionally
to the positive derivative orders:
\begin{equation}
 \Delta_j:=\Delta\frac{d_j}{d},
 \qquad \sum_{j=1}^s\Delta_j=\Delta.
\label{eq:analytic-radius-allocation}
\end{equation}
Starting from the right of \eqref{eq:analytic-word-form}, apply successively the
multiplication estimate \eqref{eq:analytic-multiplier} and the derivative estimate
\eqref{eq:analytic-derivative-block}, spending the radius loss $\Delta_j$ on the block
$D^{d_j}$.  Since the number of coefficient multipliers in each
structural monomial is bounded independently of $r$, there exists a
constant $C_1$, depending only on the model and the analytic coefficient
norms, such that
\begin{align*}
 \|\mathcal Wf\|_{\rho'}
 &\le C_1
 \prod_{j=1}^s\frac{d_j!}{\Delta_j^{d_j}}\|f\|_\rho\\
 &=C_1\Delta^{-d}d^d
   \prod_{j=1}^s\frac{d_j!}{d_j^{d_j}}\|f\|_\rho\\
 &\le C_1\Delta^{-d}d^d\|f\|_\rho,
\end{align*}
since $d_j!\le d_j^{d_j}$.  The elementary bound
$d^d\le e^d d!=e^d\Gamma(d+1)$ therefore yields
\begin{equation}
 \|\mathcal Wf\|_{\rho'}
 \le C_1e^d\frac{\Gamma(d+1)}{\Delta^d}\|f\|_\rho.
\label{eq:analytic-word-factor-bound}
\end{equation}
Because $d\le2r+p$,
$\Gamma(d+1)\le\Gamma(2r+p+1)$.  Moreover,
\[
 \Delta^{-d}
 =\Delta^{2r+p-d}\Delta^{-(2r+p)},
\]
and, since $0<\Delta\le\rho_0$ and $2r+p-d\ge0$,
\[
 \Delta^{2r+p-d}
 \le \max\{1,\rho_0\}^{2r+p}.
\]
Together with $e^d\le e^{2r+p}$, this gives constants $C_2,C_3>0$,
independent of $r$, $\rho$, and $\rho'$, such that
\begin{equation}
 \|\mathcal Wf\|_{\rho'}
 \le C_2C_3^r
 \frac{\Gamma(2r+p+1)}{\Delta^{2r+p}}\|f\|_\rho.
\label{eq:analytic-word-bound}
\end{equation}
Summing over the at most $C_0^r$ structural/composition terms and
absorbing the resulting geometric factor into $B_*^r$ proves
\eqref{eq:analytic-defect-bound}.

\medskip
\noindent\emph{Step 3: the terminal remainder $Q_m(h)$.}
The terminal remainder should not be estimated by summing
\eqref{eq:analytic-defect-bound}, because that would bring the factorial factor
$\Gamma(2r+p+1)$ into the tail.  Instead we exploit holomorphic
dependence on
\[
 z:=h^2.
\]
Define
\begin{equation}
 \mathfrak D(z)
 :=\sum_{q=0}^{\infty}\frac{z^q}{(2q+1)!}D^{2q+1}.
\label{eq:analytic-remainder-derivative}
\end{equation}
Let $0<\sigma'<\sigma$ and set $\delta:=\sigma-\sigma'$.  By
\eqref{eq:analytic-cauchy-fourier},
\[
 \left\|
 \frac{z^q}{(2q+1)!}D^{2q+1}f
 \right\|_{\sigma'}
 \le
 \frac1\delta
 \left(\frac{|z|}{\delta^2}\right)^q\|f\|_\sigma.
\]
Hence, for $|z|<\delta^2$,
\begin{equation}
 \|\mathfrak D(z)\|_{\mathcal A_\sigma\to\mathcal A_{\sigma'}}
 \le
 \frac1\delta\frac1{1-|z|/\delta^2}.
\label{eq:analytic-remainder-block}
\end{equation}
The series converges normally on smaller disks, so
$z\mapsto\mathfrak D(z)$ is holomorphic as an
$\mathcal L(\mathcal A_\sigma,\mathcal A_{\sigma'})$-valued map.  For
real $z=h^2$ one has $\mathfrak D(h^2)=\delta_h$.

Let $\mathfrak A(z)$ be the full centered generator obtained from
$A_h$ by replacing every occurrence of $\delta_h$ by
$\mathfrak D(z)$.  Every structural monomial contains at most $p$ such
factors by \eqref{eq:analytic-word-structure}.  Split the total strip loss
$\Delta=\rho-\rho'$ into $p$ equal pieces.  Each occurrence of
$\mathfrak D(z)$ is then assigned a radius loss at least $\Delta/p$,
whereas multiplication by the analytic coefficients costs no radius.
It follows from \eqref{eq:analytic-remainder-block} that
\[
 \mathfrak A(z):\mathcal A_\rho\longrightarrow\mathcal A_{\rho'}
\]
is holomorphic for $|z|<(\Delta/p)^2$.

Choose
\[
 R_\Delta:=\frac{\Delta^2}{4p^2}.
\]
On $|z|=R_\Delta$ one has
$|z|/(\Delta/p)^2=1/4$, and therefore each factor
$\mathfrak D(z)$ is bounded by $C_p\Delta^{-1}$.  Since there are at
most $p$ such factors in any monomial, and since all coefficient
multipliers are uniformly bounded on the intermediate strips,
\begin{equation}
 \sup_{|z|=R_\Delta}
 \|\mathfrak A(z)\|_{\mathcal A_\rho\to\mathcal A_{\rho'}}
 \le C\Delta^{-p}.
\label{eq:analytic-remainder-word-bound}
\end{equation}
Terms containing fewer than $p$ centered-difference factors satisfy an
initial bound $C\Delta^{-k}$ with $k<p$; because
$\Delta\le\rho_0$, this is also absorbed into the right-hand side of
\eqref{eq:analytic-remainder-word-bound} after changing $C$.

The Taylor coefficients of $\mathfrak A$ at $z=0$ are precisely
$A,R_1,R_2,\ldots$.  Thus, for fixed $m$,
\begin{equation}
 \mathfrak A(z)
 =A+\sum_{r=1}^{m-1}z^rR_r+z^mQ_m(z).
\label{eq:analytic-generator-remainder}
\end{equation}
Cauchy's integral formula for the Taylor remainder, applied in the
Banach space $\mathcal L(\mathcal A_\rho,\mathcal A_{\rho'})$, gives
\begin{equation}
 Q_m(z)
 =\frac1{2\pi i}\int_{|\zeta|=R_\Delta}
 \frac{\mathfrak A(\zeta)}{\zeta^m(\zeta-z)}\,d\zeta.
\label{eq:analytic-terminal-bound}
\end{equation}
If $|z|\le R_\Delta/2$, then
$|\zeta-z|\ge R_\Delta/2$ on the contour.  Using
\eqref{eq:analytic-remainder-word-bound},
\begin{align*}
 \|Q_m(z)f\|_{\rho'}
 &\le 2C\Delta^{-p}R_\Delta^{-m}\|f\|_\rho\\
 &=2C(4p^2)^m\Delta^{-(2m+p)}\|f\|_\rho.
\end{align*}
Hence there are model-dependent constants $A_1,B_1>0$, independent of
$m$, $z$, $\rho$, and $\rho'$, such that
\[
 \|Q_m(z)f\|_{\rho'}
 \le A_1B_1^m\Delta^{-(2m+p)}\|f\|_\rho.
\]
Finally set $z=h^2$.  The condition $|z|\le R_\Delta/2$ follows by imposing
\[
 |h|\le\theta_*\Delta,
 \qquad
 \theta_*:=\frac1{2\sqrt2\,p}.
\]
Since the real centered-difference remainder $Q_m(h)$ is precisely the
operator $Q_m(h^2)$ in \eqref{eq:analytic-generator-remainder}, the asserted terminal estimate
follows.

\warning{
The important distinction is that the coefficient estimate
\eqref{eq:analytic-defect-bound} contains the factorial factor
$\Gamma(2r+p+1)$, whereas the Cauchy estimate for the whole terminal
tail is only geometric in $m$.  Avoiding factorial growth in the
terminal remainder is exactly what later allows the extrapolation depth
to grow with the base resolution.}
\end{proof}

\subsection{The quantitative bridge}\label{subsec:quantitative-bridge}

The analytic propagation estimates leave a fixed positive strip reserve
on every compact time interval.  More precisely, fix $T>0$ and choose
$\rho_1<\rho_0$.  Propositions~\ref{prop:analytic-group} and
\ref{prop:analytic-parabolic} give a baseline strip of width at least
$\rho_1e^{-\kappa T}$.  Choose
$0<\rho_*<\rho_1e^{-\kappa T}$ and denote the reserve by
\[
 \Delta_T:=\rho_1e^{-\kappa T}-\rho_*>0.
\]
A derivative defect may spend part of this reserve.  Any loss spent
before a later propagation interval is subsequently reduced rather than
increased, so it is sufficient that the sum of all allocated losses be
at most $\Delta_T$.

\begin{theorem}[Analytic quantitative bridge]
\label{thm:analytic-bridge}
Assume the hypotheses of Proposition~\ref{prop:analytic-parabolic} in
the parabolic case, or Proposition~\ref{prop:analytic-group} in the
normalized Schr\"odinger/cubic cases, and let
$f\in\mathcal A_{\rho_0}$.  Let $p=2$ in the two second-order cases and
$p=3$ in the cubic case.  Then for every $T>0$ and every
$0<\rho_*<\rho_1e^{-\kappa T}$ there exist constants $C_T,K_T>0$ and
$N_T\ge1$, independent of $f$, $m$, and $N$, such that, for every $N\ge N_T$ and every $0\le t\le T$,
\begin{equation}
 \left\|
 \mathcal S_2(t/N)^Nf-e^{tA}f
 -\sum_{r=1}^{m-1}\frac{E_{2r}(t)f}{N^{2r}}
 \right\|_{\rho_*}
 \le
 C_TK_T^m\Gamma((p+1)m+1)e^{\Gamma t}
 \frac{t^{2m+1}}{N^{2m}}\,
 \|f\|_{\rho_0}.
\label{eq:analytic-bridge-estimate}
\end{equation}
In particular, the analytic product expansion realizes the pointwise
form \eqref{eq:product-expansion} on every fixed strip interval, with
\[
 {\ C^{\rm an}_{m,T}\le C_TK_T^m\Gamma((p+1)m+1). 
 \ }
\]
\warning{Thus the dependence on $f$ is entirely through $\|f\|_{\rho_0}$.}
\end{theorem}

\begin{proof}
Write
\[
 T(t):=e^{tA},\qquad T_h(t):=e^{tA_h},\qquad
 \Delta_h:=A_h-A.
\]
For every integer $q\ge1$, Lemma~\ref{lem:analytic-defects} gives the
finite generator expansion
\begin{equation}
 \Delta_h
 =\sum_{j=1}^{q-1}h^{2j}R_j+h^{2q}Q_q(h).
\label{eq:analytic-generator-depth-expansion}
\end{equation}
We use this expansion recursively, choosing at each stage only the
remaining depth needed to reach total weight $m$.

\medskip
\noindent\emph{Step 1: a fixed strip reserve on $[0,T]$.}
The proofs of Propositions~\ref{prop:analytic-group} and
\ref{prop:analytic-parabolic} 
yields the following propagation estimate.  
With the same constants
$\kappa,\Gamma$, for every $0<\sigma\le\rho_1$ and every
$0\le\tau\le T$,
\begin{equation}
 \|T(\tau)g\|_{\sigma e^{-\kappa\tau}}
 \le C_0 e^{\Gamma\tau}\|g\|_\sigma,
 \qquad
 \|T_h(\tau)g\|_{\sigma e^{-\kappa\tau}}
 \le C_0 e^{\Gamma\tau}\|g\|_\sigma,
\label{eq:analytic-propagator-bound}
\end{equation}
uniformly in the admissible $h$.  Indeed, 
the constants are controlled by the analytic-norms of the coefficients on a fixed larger strip and the corresponding Cauchy derivative bounds.

In the group cases the same estimate holds for negative times,
but only nonnegative times are needed below.

We start the Duhamel expansion in $\mathcal A_{\rho_1}$; since
$f\in\mathcal A_{\rho_0}$ and $\rho_1<\rho_0$,
$\|f\|_{\rho_1}\le\|f\|_{\rho_0}$.  Recall
\[
 \Delta_T:=\rho_1e^{-\kappa T}-\rho_*>0.
\]
The quantity $\Delta_T$ is the total analytic-radius reserve available
to the defect operators.

\medskip
\noindent\emph{Step 2: finite Duhamel expansion and weight bookkeeping.}
Duhamel's identity, with the perturbed propagator on the far left, is
\begin{equation}
 T_h(t)-T(t)
 =\int_0^t T_h(t-s)\Delta_hT(s)\,ds.
\label{eq:analytic-duhamel-bridge}
\end{equation}
Iterating \eqref{eq:analytic-duhamel-bridge} only through the leftmost $T_h$ factor and
using \eqref{eq:analytic-generator-depth-expansion} recursively produces a finite expansion
\begin{equation}
 T_h(t)
 =T(t)+\sum_{r=1}^{m-1}h^{2r}F_r(t)
   +h^{2m}\mathcal R_{m,h}(t).
\label{eq:analytic-finite-expansion}
\end{equation}
We now describe the two kinds of terms in \eqref{eq:analytic-finite-expansion}.

An ordinary coefficient word of total weight $r$ and length $\ell$ is
indexed by positive integers
\[
 r_1+\cdots+r_\ell=r,
 \qquad r_j\ge1,
\]
and has the form
\begin{align}
 \mathcal W_{\mathbf r}(t)
 :=\int_{0<s_1<\cdots<s_\ell<t}
 &T(t-s_\ell)R_{r_\ell}
 T(s_\ell-s_{\ell-1})\cdots \notag\\[-2mm]
 &\qquad\cdots R_{r_2}T(s_2-s_1)R_{r_1}T(s_1)
 \,ds_1\cdots ds_\ell .
\label{eq:analytic-ordinary-word}
\end{align}
Because every $r_j\ge1$, necessarily $\ell\le r$.

The remainder is generated by terminating a branch exactly when its
total weight reaches $m$.  More precisely, suppose a branch has already
accumulated ordinary weight
\[
 r=r_1+\cdots+r_k<m,
\]
and put $q:=m-r$.  At the next stage we use \eqref{eq:analytic-generator-depth-expansion} with this
remaining depth $q$ and select the terminal factor $Q_q(h)$.  The
resulting word is
\begin{align}
 \mathcal V_{\mathbf r,q,h}(t)
 :=\int_{0<s_1<\cdots<s_{k+1}<t}
 &T_h(t-s_{k+1})Q_q(h)T(s_{k+1}-s_k)\notag\\[-2mm]
 &\qquad\cdot R_{r_k}T(s_k-s_{k-1})\cdots
 R_{r_1}T(s_1)\,ds .
\label{eq:analytic-terminal-word}
\end{align}
The case $k=0$ is allowed and gives
$\int_0^tT_h(t-s)Q_m(h)T(s)\,ds$.  Since $q=m-r$, every terminal
branch carries the factor
\[
 h^{2r}h^{2q}=h^{2m}.
\]
Thus no word of weight larger than $m$ is formed: all ordinary words
have weight $r<m$, whereas every terminal word has weight exactly $m$.

\medskip
\noindent\emph{Step 3: compatibility of strip losses with the intervening
propagators.}
Consider first an ordinary word \eqref{eq:analytic-ordinary-word}.  Put
\[
 d_j:=2r_j+p,
 \qquad
 d:=d_1+\cdots+d_\ell=2r+p\ell.
\]
Allocate the strip reserve proportionally to the differential orders,
\begin{equation}
 \delta_j:=\Delta_T\frac{d_j}{d},
 \qquad
 \sum_{j=1}^\ell\delta_j=\Delta_T.
\label{eq:analytic-radius-partition}
\end{equation}
We start by verifying that these losses fit inside the available strip.
Without defects, propagation from radius $\rho_1$ to time $s$ leaves
radius $\rho_1e^{-\kappa s}$.  Suppose that by time $s_j$ the first
$j-1$ defects have spent losses $\delta_1,\dots,\delta_{j-1}$.  By
induction the available radius immediately before the $j$th defect is
at least
\[
 \rho_1e^{-\kappa s_j}-\sum_{i<j}\delta_i.
\]
After applying $R_{r_j}$ it is at least the same quantity minus
$\delta_j$.  Propagation over the next interval of length
$s_{j+1}-s_j$ multiplies the entire radius by
$e^{-\kappa(s_{j+1}-s_j)}$; hence the already spent loss is multiplied
by a factor at most one.  Therefore, at time $s_{j+1}$, the radius is
at least
\[
 \rho_1e^{-\kappa s_{j+1}}-\sum_{i\le j}\delta_i.
\]
At the final time $t\le T$ this gives
\[
 \rho_1e^{-\kappa t}-\sum_{j=1}^\ell\delta_j
 \ge \rho_1e^{-\kappa T}-\Delta_T
 =\rho_*.
\]
Thus the total reserve $\Delta_T$ may indeed be allocated among the
defects independently of their locations in time.

\medskip
\noindent\emph{Step 4: estimate of an ordinary complete word.}
Lemma~\ref{lem:analytic-defects}, applied successively with the losses
in \eqref{eq:analytic-radius-partition}, gives
\[
 \|R_{r_j}g\|_{\sigma-\delta_j}
 \le
 A_*B_*^{r_j}
 \frac{d_j!}{\delta_j^{d_j}}\|g\|_\sigma.
\]
Combining these bounds with \eqref{eq:analytic-propagator-bound} between consecutive defects
and using \eqref{eq:analytic-radius-partition} yields, for the integrand in
\eqref{eq:analytic-ordinary-word},
\[
 \|\text{integrand}\|_{\rho_*}
 \le
 e^{\Gamma t}A_*^\ell B_*^r
 \prod_{j=1}^\ell\frac{d_j!}{\delta_j^{d_j}}
 \|f\|_{\rho_0}.
\]
Since $\ell\le r$, the factor $A_*^\ell$ is geometric in $r$.  Moreover,
\begin{align*}
 \prod_{j=1}^\ell\frac{d_j!}{\delta_j^{d_j}}
 &=\Delta_T^{-D}D^D
   \prod_{j=1}^\ell\frac{d_j!}{d_j^{d_j}}\\
 &\le \Delta_T^{-D}D^D
 \le C_T^rD!,
\end{align*}
where we used $d_j!\le d_j^{d_j}$, $d\le(p+2)r$, and
$d^d\le e^dd!$; all fixed powers of $\Delta_T^{-1}$ and $e$ are
absorbed into $C_T^r$.

The simplex in \eqref{eq:analytic-ordinary-word} has volume $t^\ell/\ell!$.  Hence
\begin{equation}
 \|\mathcal W_{\mathbf r}(t)f\|_{\rho_*}
 \le
 C_T^r e^{\Gamma t}
 \frac{t^\ell D!}{\ell!}\|f\|_{\rho_0}.
\label{eq:analytic-word-integrand}
\end{equation}
For $0\le t\le T$, retain one factor of $t$ and absorb the remaining
$t^{\ell-1}$ into the geometric constant.  It remains to control
$D!/\ell!$.

Set
\[
 M:=d-\ell=2r+(p-1)\ell.
\]
Then
\[
 \frac{d!}{\ell!}
 =\prod_{j=\ell+1}^{d}j
 \le d^M.
\]
Furthermore,
\[
 \frac{d}{M}
 =\frac{2r+p\ell}{2r+(p-1)\ell}
 \le \frac{p+2}{p+1},
\]
because $1\le\ell\le r$.  Therefore
\[
 d^M
 \le
 \left(\frac{p+2}{p+1}\right)^M M^M
 \le
 \left(e\frac{p+2}{p+1}\right)^M M!,
\]
where we used the elementary inequality $M^M\le e^MM!$.  Finally,
$M\le(p+1)r$, so
\[
 M!\le\Gamma((p+1)r+1),
\]
and the remaining exponential factor is geometric in $r$.  Thus
\begin{equation}
 {
 \frac{d!}{\ell!}
 \le C_p^r\Gamma((p+1)r+1).}
\label{eq:analytic-ordinary-word-estimate}
\end{equation}
Substituting \eqref{eq:analytic-ordinary-word-estimate} into \eqref{eq:analytic-word-integrand} gives
\begin{equation}
 \|\mathcal W_{\mathbf r}(t)f\|_{\rho_*}
 \le
 t\,e^{\Gamma t}C_T^r
 \Gamma((p+1)r+1)\|f\|_{\rho_0}.
\label{eq:analytic-terminal-index}
\end{equation}
For fixed $r$, the number of compositions
$r=r_1+\cdots+r_\ell$ over all lengths is $2^{r-1}$, and the finite
number of structural defect types adds only another geometric factor.
Consequently
\begin{equation}
 \|F_r(t)f\|_{\rho_*}
 \le
 t\,e^{\Gamma t}C_T^r
 \Gamma((p+1)r+1)\|f\|_{\rho_0}.
\label{eq:analytic-terminal-prefactor}
\end{equation}

\medskip
\noindent\emph{Step 5: terminal words and the order-$m$ remainder.}
Consider a terminal word \eqref{eq:analytic-terminal-word}, with ordinary weight $r<m$
and $q=m-r$.  Give one half of the strip reserve to the terminal factor,
\[
 \delta_Q:=\frac{\Delta_T}{2},
\]
and distribute the remaining half among the ordinary defects exactly as
in Step~4.  By the Cauchy remainder estimate in
Lemma~\ref{lem:analytic-defects},
\begin{equation}
 \|Q_q(h)g\|_{\sigma-\delta_Q}
 \le
 A_1B_1^q\delta_Q^{-(2q+p)}\|g\|_\sigma,
\label{eq:analytic-terminal-integrand}
\end{equation}
provided $|h|\le\theta_*\delta_Q$.  Since $\delta_Q$ is fixed, the
right-hand prefactor in \eqref{eq:analytic-terminal-integrand} is bounded by $C_T^q$.

The ordinary part of the same word has total weight $r$ and is bounded,
using only the remaining reserve $\Delta_T/2$, by
$C_T^r\Gamma((p+1)r+1)$.  The full time simplex has length $k+1$ if
there are $k$ ordinary defects; its factor $t^{k+1}/(k+1)!$ is bounded
by $t$ times a geometric factor depending only on $T$.  Hence every
terminal word satisfies
\[
 \|\mathcal V_{\mathbf r,q,h}(t)f\|_{\rho_*}
 \le
 t\,e^{\Gamma t}C_T^{r+q}
 \Gamma((p+1)r+1)\|f\|_{\rho_0}.
\]
Since $r+q=m$ and $r\le m$,
\[
 \Gamma((p+1)r+1)
 \le \Gamma((p+1)m+1).
\]
The total number of terminal branches is at most geometric in $m$.
After enlarging $K_T$, we therefore obtain the central remainder bound
\begin{equation}
 \boxed{
 \|\mathcal R_{m,h}(t)f\|_{\rho_*}
 \le
 C_TK_T^m\Gamma((p+1)m+1)
 e^{\Gamma t}t\,\|f\|_{\rho_0}.}
\label{eq:analytic-terminal-estimate}
\end{equation}

\medskip
\noindent\emph{Step 6: insertion of the Chernoff step $h=t/N$.}
By construction of the second-order Chernoff family,
\[
 \mathcal S_2(t/N)^N=e^{tA_{t/N}}=T_{t/N}(t).
\]
Choose $N_T$ so that
\[
 \frac{T}{N_T}\le\theta_*\frac{\Delta_T}{2}.
\]
Then, for every $N\ge N_T$ and every $0\le t\le T$, the hypothesis of
\eqref{eq:analytic-terminal-integrand} is satisfied with $h=t/N$.  Setting
\[
 E_{2r}(t):=t^{2r}F_r(t),
\]
formula \eqref{eq:analytic-finite-expansion} becomes
\[
 \mathcal S_2(t/N)^Nf
 =e^{tA}f
 +\sum_{r=1}^{m-1}\frac{E_{2r}(t)f}{N^{2r}}
 +\left(\frac{t}{N}\right)^{2m}
  \mathcal R_{m,t/N}(t)f.
\]
Applying \eqref{eq:analytic-terminal-estimate} gives
\[
 \left\|
 \mathcal S_2(t/N)^Nf-e^{tA}f
 -\sum_{r=1}^{m-1}\frac{E_{2r}(t)f}{N^{2r}}
 \right\|_{\rho_*}
 \le
 C_TK_T^m\Gamma((p+1)m+1)e^{\Gamma t}
 \frac{t^{2m+1}}{N^{2m}}\|f\|_{\rho_0},
\]
which is \eqref{eq:analytic-bridge-estimate}.
\end{proof}

\subsection{Optimized growing-depth Richardson extrapolation}\label{subsec:optimized-richardson}

\begin{lemma}\label{lem:optimized-analytic-bound}There exists $C_T,K_*>0$ such that
\begin{equation}
 \|\mathcal M_{2m,n}(t)f-e^{tA}f\|_{\rho_*}
 \le C_T e^{\Gamma t}t
 \left(\frac{K_*t^2m^{p-1}}{n^2}\right)^m
 \|f\|_{\rho_0},\qquad 0\le t\le T,
\label{eq:optimized-preliminary-bound}
\end{equation}
for all $n\ge N_T$, with $N_T$ as in Theorem~\ref{thm:analytic-bridge}.
\end{lemma}

\begin{proof}

For the even nodes $1,2,\dots,m$, recall
\[
 c_j^{(m)}=(-1)^{m-j}\frac{2j^{2m}}{(m-j)!(m+j)!}.
\]
Then every refined step count
$N=jn$, $1\le j\le m$, also satisfies $N\ge N_T$, so the theorem applies
to every product entering the Richardson extrapolation.  
The exact weighted identity
\begin{equation}
 \sum_{j=1}^m\frac{|c_j^{(m)}|}{j^{2m}}
 =\frac{4^m-\binom{2m}{m}}{(2m)!}
 <\frac{4^m}{\Gamma(2m+1)}
\label{eq:richardson-weight-identity}
\end{equation}
then turns Theorem~\ref{thm:analytic-bridge} into
\begin{equation}
 \|\mathcal M_{2m,n}(t)f-e^{tA}f\|_{\rho_*}
 \le
 C_T(4K_T)^m e^{\Gamma t}
 \frac{\Gamma((p+1)m+1)}{\Gamma(2m+1)}
 \frac{t^{2m+1}}{n^{2m}}
 \|f\|_{\rho_0},\qquad 0\le t\le T.
\label{eq:richardson-weight-bound}
\end{equation}
Stirling's formula then yields \eqref{eq:optimized-preliminary-bound}. 
\end{proof}

\begin{corollary}[Super-fast analytic Richardson rates]
\label{cor:analytic-superfast}
Under the hypotheses of Theorem~\ref{thm:analytic-bridge}, choose
$\eta>0$ sufficiently small and
\[
 m(n)=\left\lfloor\eta n^{2/(p-1)}\right\rfloor.
\]
Then, for all sufficiently large $n$ and every $0\le t\le T$,
\begin{equation}
 \|\mathcal M_{2m(n),n}(t)f-e^{tA}f\|_{\rho_*}
 \le C_T e^{\Gamma t}t
 e^{-c_T n^{2/(p-1)}}\|f\|_{\rho_0}.
\label{eq:optimized-target-bound}
\end{equation}
Consequently the parabolic and Schr\"odinger-type second-order models
admit the rate
\[
 {\ e^{-cn^2}\ },
\]
while the third-order dispersive model admits
\[
 {\ e^{-cn}\ }.
\]
\end{corollary}

\begin{proof}
The estimate to be optimized is \eqref{eq:optimized-preliminary-bound}.  Introduce the
quantity
\[
 \Theta_{m,n}(t)
 :=\frac{K_*t^2m^{p-1}}{n^2}.
\]
Then \eqref{eq:optimized-preliminary-bound} reads
\begin{equation}
 \|\mathcal M_{2m,n}(t)f-e^{tA}f\|_{\rho_*}
 \le
 C_Te^{\Gamma t}t\,\Theta_{m,n}(t)^m\|f\|_{\rho_0},
 \qquad 0\le t\le T.
\label{eq:optimization-log-bound}
\end{equation}
We claim that we can make $\Theta_{m,n}(t)<1$ uniformly in $t$, while the
exponent $m$ tends to infinity with $n$.

Since $t\le T$,
\[
 \Theta_{m,n}(t)
 \le \frac{K_*T^2m^{p-1}}{n^2}.
\]
Thus the natural choice is
\[
 m^{p-1}\asymp n^2,
 \qquad\text{equivalently}\qquad
 m\asymp n^{2/(p-1)}.
\]
We therefore set
\[
 \alpha:=\frac{2}{p-1},
 \qquad
 m(n):=\lfloor\eta n^\alpha\rfloor,
\]
with $\eta>0$ to be chosen.  Because $m(n)\le\eta n^\alpha$,
\[
 m(n)^{p-1}
 \le \eta^{p-1}n^{\alpha(p-1)}
 =\eta^{p-1}n^2.
\]
Consequently, uniformly for $0\le t\le T$,
\begin{equation}
 \Theta_{m(n),n}(t)
 \le K_*T^2\eta^{p-1}.
\label{eq:optimization-choice}
\end{equation}
Choose $\eta>0$ sufficiently small that
\[
 K_*T^2\eta^{p-1}<1.
\]
Fix any number $q$ such that
\begin{equation}
 K_*T^2\eta^{p-1}<q<1.
\label{eq:optimization-stirling}
\end{equation}
Then \eqref{eq:optimization-choice} gives
\[
 \Theta_{m(n),n}(t)\le q,
 \qquad 0\le t\le T,
\]
and therefore \eqref{eq:optimization-log-bound} yields
\begin{equation}
 \|\mathcal M_{2m(n),n}(t)f-e^{tA}f\|_{\rho_*}
 \le
 C_Te^{\Gamma t}t\,q^{m(n)}\|f\|_{\rho_0}.
\label{eq:optimization-exponent}
\end{equation}

It remains only to convert $q^{m(n)}$ into an exponential in the base
resolution $n$.  Since
\[
 m(n)=\lfloor\eta n^\alpha\rfloor
 \ge \eta n^\alpha-1
\]
and $0<q<1$, the function $x\mapsto q^x$ is decreasing.  Hence
\begin{align*}
 q^{m(n)}
 &\le q^{\eta n^\alpha-1}\\
 &=q^{-1}\exp\!\bigl(\eta n^\alpha\log q\bigr).
\end{align*}
Because $\log q<0$, set
\[
 c_T:=-\eta\log q>0.
\]
Then
\begin{equation}
 q^{m(n)}
 \le q^{-1}e^{-c_Tn^\alpha}.
\label{eq:optimization-final-exponent}
\end{equation}
Absorbing the \warning{harmless} constant $q^{-1}$ into $C_T$ and recalling that
$\alpha=2/(p-1)$, equations \eqref{eq:optimization-exponent}--\eqref{eq:optimization-final-exponent} give
\[
 \|\mathcal M_{2m(n),n}(t)f-e^{tA}f\|_{\rho_*}
 \le
 C_Te^{\Gamma t}t\,
 e^{-c_Tn^{2/(p-1)}}\|f\|_{\rho_0},
 \qquad 0\le t\le T,
\]
which is \eqref{eq:optimized-target-bound}.

For the two second-order models, $p=2$, so
$2/(p-1)=2$ and $m(n)\asymp n^2$, giving the rate
$e^{-cn^2}$.  For the cubic model, $p=3$, so
$2/(p-1)=1$ and $m(n)\asymp n$, giving the rate $e^{-cn}$.
\end{proof}

The growing-depth rates also give an explicit representation of the
evolution semigroup itself.

\begin{corollary}[Accelerated Chernoff representation of the evolution solution]
\label{cor:evolution-representation}
For each of the three analytic model classes, the solution of the
Cauchy problem \emph{(CP)} admits the Richardson--Chernoff representation
\begin{equation}
 u(t)=e^{tA}f
 =\lim_{n\to\infty}
 \sum_{j=1}^{m(n)}c_j^{(m(n))}
 \mathcal S_2\!\left(\frac{t}{jn}\right)^{jn}f,
 \qquad 0\le t\le T,
 \label{eq:evolution-representation}
\end{equation}
with convergence in $\mathcal A_{\rho_*}$, uniformly on $[0,T]$.
For the parabolic and Schr\"odinger-type second-order equations the
error is $O(e^{-cn^2})$, whereas for the third-order dispersive equation
it is $O(e^{-cn})$.
\end{corollary}

\begin{corollary}[Fixed-depth resolvent representation]
\label{cor:resolvent-representation}
Fix the extrapolation depth $m$ and let
$\operatorname{Re}\lambda>\widehat\omega_m$ as in
Proposition~\ref{prop:resolvent-rate}. Then
\begin{equation}
 (\lambda I-A)^{-1}g
 =\lim_{n\to\infty}
 \int_0^\infty e^{-\lambda t}\mathcal M_{2m,n}(t)g\,dt,
 \label{eq:resolvent-representation}
\end{equation}
and
\[
 \left\|\int_0^\infty e^{-\lambda t}\mathcal M_{2m,n}(t)g\,dt
 -(\lambda I-A)^{-1}g\right\|
 \le \frac{C_{\lambda,m}}{n^{2m}}\|g\|_{\mathcal D_m}.
\]
For the three original operators this gives explicit representations of the solutions of the corresponding nonhomogeneous ODEs:
\[
 \lambda f-a f''-b f'-cf=g,\qquad
 \lambda f-iaf''-b f'-cf=g,\qquad
 \lambda f-a f'''-b f'-cf=g,
\]
respectively.
\end{corollary}

\warning{This is the Laplace-transform mechanism used in the Chernoff-resolvent
program of Remizov \cite{Remizov25,RemizovVedenin}. The growing-depth
exponential and super-exponential rates above concern the compact-time
evolution problem; the resolvent estimate is the fixed-depth
polynomial rate of Proposition~\ref{prop:resolvent-rate}.}

\subsection{Proof of the main
analytic theorem}
\label{subsec:analytic-main-theorem-proof}

We now 
 complete the proof of
Theorem~\ref{thm:unified-analytic-rate}.

\begin{proof}[Proof of Theorem~\ref{thm:unified-analytic-rate}]
The parabolic case is immediate from
Corollary~\ref{cor:analytic-superfast}, since no change of
variables is involved.  Indeed, $\mathcal J=I$ and the fixed-strip
propagation of Proposition~\ref{prop:analytic-parabolic} corresponds
to $\kappa=0$.  Given
\[
0<\rho_{\rm out}<\rho_{\rm in}<r_a,
\]
choose intermediate radii
\[
\rho_{\rm out}<\rho_*<\rho_1<\rho_{\rm in}.
\]
The normalized estimate of
Corollary~\ref{cor:analytic-superfast} then gives directly
\eqref{eq:main-original-analytic-rate}.

We now treat the two nonsmoothing models simultaneously.

Fix
\[
0<\rho_{\rm in}<r_a
\]
and assume
\[
0<\rho_{\rm out}
<
\frac{m_p}{M_p}
e^{-\kappa T}\rho_{\rm in}.
\]
Equivalently,
\begin{equation}
M_p\rho_{\rm out}
<
m_pe^{-\kappa T}\rho_{\rm in}.
\label{eq:radius-gap-original}
\end{equation}

Because the inequality is strict, we may choose normalized radii
\[
\rho_0^y>
\rho_1^y>
\rho_*^y
\]
such that
\begin{equation}
0<\rho_1^y<\rho_0^y<m_p\rho_{\rm in}
\label{eq:normalized-input-radii}
\end{equation}
and
\begin{equation}
M_p\rho_{\rm out}
<
\rho_*^y
<
\rho_1^y e^{-\kappa T}.
\label{eq:normalized-output-radii}
\end{equation}
Indeed, \eqref{eq:radius-gap-original} is precisely the condition
ensuring that such a choice is possible.

By Lemma~\ref{lem:liouville-analytic-spaces},
\begin{equation}
\|\mathcal J_pf\|_{\rho_0^y}
\le
C_{\rm in}
\|f\|_{\rho_{\rm in}},
\label{eq:J-input-bound}
\end{equation}
because
\[
\rho_0^y<m_p\rho_{\rm in}.
\]
Likewise,
\begin{equation}
\|\mathcal J_p^{-1}g\|_{\rho_{\rm out}}
\le
C_{\rm out}
\|g\|_{\rho_*^y},
\label{eq:Jinv-output-bound}
\end{equation}
because
\[
M_p\rho_{\rm out}<\rho_*^y.
\]

By Proposition~\ref{prop:normalized-analytic-coefficients}, the
normalized generator satisfies the coefficient assumptions of
Proposition~\ref{prop:analytic-group} and hence of
Corollary~\ref{cor:analytic-superfast}.  Therefore, for
\[
g:=\mathcal J_pf,
\]
there exist $C_T,c_T,\Gamma>0$ and $N_T\in\N$ such that
\begin{equation}
\left\|
\mathcal M_{2m(n),n}(t)g
-
e^{t\widetilde A}g
\right\|_{\rho_*^y}
\le
C_Te^{\Gamma t}\,
t\,
e^{-c_Tn^{2/(p-1)}}
\|g\|_{\rho_0^y},
\label{eq:normalized-superfast-applied}
\end{equation}
for
\[
m(n)
=
\left\lfloor
\eta_Tn^{2/(p-1)}
\right\rfloor,
\qquad
0\le t\le T,
\]
and all $n\ge N_T$.

Since
\[
\widetilde A
=
\mathcal J_pA_0\mathcal J_p^{-1},
\]
and $\mathcal J_p$ is boundedly invertible on $L^2$, then 
\begin{equation}
e^{tA_0}
=
\mathcal J_p^{-1}
e^{t\widetilde A}
\mathcal J_p.
\label{eq:semigroup-similarity}
\end{equation}
By definition,
\[
\mathcal M_{2m,n}^{\rm orig}(t)
=
\mathcal J_p^{-1}
\mathcal M_{2m,n}(t)
\mathcal J_p.
\]
Hence
\begin{align*}
&
\mathcal M_{2m(n),n}^{\rm orig}(t)f-e^{tA_0}f
\\
&\qquad
=
\mathcal J_p^{-1}
\left[
\mathcal M_{2m(n),n}(t)
-
e^{t\widetilde A}
\right]
\mathcal J_pf .
\end{align*}
Using first \eqref{eq:Jinv-output-bound}, then
\eqref{eq:normalized-superfast-applied}, and finally
\eqref{eq:J-input-bound}, we obtain
\begin{align*}
&
\left\|
\mathcal M_{2m(n),n}^{\rm orig}(t)f-e^{tA_0}f
\right\|_{\rho_{\rm out}}
\\
&\qquad
\le
C_{\rm out}
\left\|
\left[
\mathcal M_{2m(n),n}(t)-e^{t\widetilde A}
\right]\mathcal J_pf
\right\|_{\rho_*^y}
\\
&\qquad
\le
C_{\rm out}C_T
e^{\Gamma t}t
e^{-c_Tn^{2/(p-1)}}
\|\mathcal J_pf\|_{\rho_0^y}
\\
&\qquad
\le
C_{\rm out}C_{\rm in}C_T
e^{\Gamma t}t
e^{-c_Tn^{2/(p-1)}}
\|f\|_{\rho_{\rm in}}.
\end{align*}
Absorbing $C_{\rm in}C_{\rm out}$ into $C_T$ proves
\eqref{eq:main-original-analytic-rate}.

For $p=2$,
\[
\frac{2}{p-1}=2,
\]
so the parabolic and Schr\"odinger models have Gaussian-type rate
\[
e^{-c_Tn^2}.
\]
For $p=3$,
\[
\frac{2}{p-1}=1,
\]
so the cubic model has rate
\[
e^{-c_Tn}.
\]
This completes the proof.
\end{proof}

\warning{\section{Relation to other Chernoff constructions}

The rate theorems of Galkin and Remizov \cite{GR25,GR22} start from a
quantitative local Taylor remainder for a fixed Chernoff function and
propagate it through the product formula.  The present extrapolation
theorem instead starts from an asymptotic expansion of the iterated
products and cancels selected powers by a generalized Vandermonde
system.  The two mechanisms are complementary.

The explicit Remizov--Vedenin family \cite{RemizovVedenin} for a
uniformly parabolic second-order operator is a natural test case for
higher-order extrapolation.  To apply Theorem~\ref{thm:general-richardson}
at depth $m$, one must establish an expansion of its products
$S(t/N)^N$ of the form \eqref{eq:product-expansion} on an appropriate power scale
$\boldsymbol\nu$: the full integer scale when odd and even powers are
present, or the even scale when a symmetry eliminates the odd powers.
That higher-depth product expansion is not proved here.
}

\section{AI Disclosure.}\label{sec:ai-disclosure} The ideas underlying this preprint originated with the author. ChatGPT and Claude were used for mathematical discussions and editorial assistance. The author assumes full responsibility for the content.

\bibliographystyle{plain} 

\bibliography{references}
\end{document}